\documentclass[11pt, a4paper]{scrreprt}\usepackage[a4paper,left=2.8cm, right=2.8cm]{geometry}
\usepackage{amsthm}
\usepackage{mdframed}
\usepackage{tikz-cd}
\newtheorem{Definition}{Definition}[chapter]
\newtheorem{Satz}[Definition]{Satz}

\newtheorem{Theorem}[Definition]{Theorem}
\newtheorem{Proposition}[Definition]{Proposition}
\newtheorem{Lemma}[Definition]{Lemma}
\theoremstyle{definition}
\newtheorem{Bemerkung}[Definition]{Bemerkung}
\theoremstyle{plain}
\newtheorem{Behauptung}[Definition]{Behauptung}
\usepackage{ngerman}
\usepackage[utf8]{inputenc} %damit ä, ö, ü fuktioniert
\usepackage[T1]{fontenc} %damit Silbentrennung richtig
\usepackage{amsmath, amssymb} %Mengenzeichen usw.?!
\usepackage{graphicx}% http://ctan.org/pkg/graphicx für \scalebox
\usepackage{setspace}
\usepackage{bbm}
\usepackage{amsfonts}
\usepackage[disable]{todonotes}
\usepackage{textcomp} % Für Geburts- und Sterbezeichen

\usepackage{mathtools} %Text unter Klammern positionieren
\usepackage{enumerate}
\renewcommand{\qedsymbol}{$\blacksquare$}
\usepackage{graphicx}
\usepackage[
nonumberlist, %keine Seitenzahlen anzeigen
acronym,      %ein Abkürzungsverzeichnis erstellen
toc,          %Einträge im Inhaltsverzeichnis
section,      %im Inhaltsverzeichnis auf section-Ebene erscheinen
]%style=super
{glossaries}
\newglossaryentry{1}
{
	name={$(X,\mathfrak{A},\mu,T)$, $(Y,\mathfrak{B},\nu,S)$},
	description={Maßtheoretische dynamische Systeme}
}
\newglossaryentry{2}{name={$\mu,\nu,\sigma, \lambda_{\mathbb{T}}$},description={Wahrscheinlichkeitsmaße}}
\newglossaryentry{3}{name={$\mathfrak{A}$, $\mathfrak{B}$, $\mathfrak{D}$, $\mathfrak{E}$},description={$\sigma$"=Algebren}}
\newglossaryentry{4}{name={$\mathbbm{1}_A$},description={Charakteristische Funktion der Menge $A$}}
\newglossaryentry{5}{name={$\mu\otimes \nu$},description={Produktmaß von $\mu$ und $\nu$}}
\newglossaryentry{6}{name={$\mathcal{K}$},description={Kroneckerfaktor}}
\newglossaryentry{7}{name={$\mathcal{K}^\bot$},description={Orthogonales Komplement des Kroneckerfaktors $\mathcal{K}$}}
\newglossaryentry{8}{name={$\overline{X}$},description={Abschluss des metrischen oder normierten Raumes $X$}}
\newglossaryentry{10}{name={$\mathbb{T}$},description={Komplexer Einheitskreis}}
\newglossaryentry{11}{name={$\mathcal{N}$},description={Nullmenge}}
\newglossaryentry{12}{name={$\operatorname{span}(V)$},description={lineare Hülle von $V$}}
\newglossaryentry{13}{name={$M(\mathbb{T})$},description={Menge aller endlichen Borel"=Maße auf dem Einheitskreis $\mathbb{T}$}}
\newglossaryentry{14}{name={$\hat{\sigma}(n)$},description={$n$'ter Fourierkoeffizient des Maßes $\sigma$}}
\newglossaryentry{15}{name={$\langle \cdot, \cdot \rangle$},description={Skalarprodukt}}
\newglossaryentry{17}{name={$\mathbb{N}$},description={Menge der positiven ganzen Zahlen}}
\newglossaryentry{18}{name={$\mathcal{RI}$},description={Abbildung aus der Menge $\{\operatorname{Re}, \operatorname{Im}, -\operatorname{Re}, -\operatorname{Im}\}$}}
\newglossaryentry{19}{name={$A\triangle B$},description={Symmetrische Differenz von $A$ und $B$: $A\triangle B:= (A\setminus B)\cup (B\setminus A)$}}
\newglossaryentry{20}{name={$p$},description={Dichte einer Zahlenfolge in den natürlichen Zahlen}}
\newglossaryentry{21}{name={$\#A$},description={Anzahl der Element in der Menge $A$}}
\newglossaryentry{22}{name={$\operatorname{Re}$},description={Abbildung, die eine komplexe Zahl auf ihren Realteil abbildet}}
\newglossaryentry{23}{name={$\operatorname{Im}$},description={Abbildung, die eine komplexe Zahl auf ihren Imaginärteil abbildet}}
\newglossaryentry{24}{name={$\pi$},description={Faktorabbildung}}
\newglossaryentry{25}{name={$(\tilde{X}, \tilde{\mathfrak{A}}, \tilde{\mu}, \tilde{T})$},description={Invertierbare Erweiterung des Systems $(X,\mathfrak{A},\mu,T)$}}
\newglossaryentry{26}{name={$K$},description={Kompakte Menge beziehungsweise Konstante}}
\newglossaryentry{28}{name={$\overline{\nu}$},description={Vervollständigung des Wahrscheinlichkeitsmaßes $\nu$}}
\newglossaryentry{27}{name={$\overline{\mathfrak{B}}$},description={Vervollständigung der $\sigma$"=Algebra $\mathfrak{B}$}}
\newglossaryentry{30}{name={$\sigma(\mathcal{A})$},description={Von der Menge $\mathcal{A}$ erzeugte $\sigma$"=Algebra}}
\newglossaryentry{29}{name={$[x]_{\mathfrak{A}}$},description={Atom von $\mathfrak{A}$, welches $x$ enthält}}
\newglossaryentry{31}{name={$\left\lceil r \right \rceil$},description={$\left\lceil r \right \rceil :=\min\{k\in \mathbb{Z}|\, k\geq r\}$}}
\newglossaryentry{32}{name={$\mathbb{K}$},description={Der Körper $\mathbb{R}$ oder $\mathbb{C}$}}

\makenoidxglossaries

\begin{document}
	\todo{Danksagung!!!!!, siehe Quelltext}
	\todo{Fettgedruckten Text in Rohfassung abarbeiten}

\setstretch{1.15}
\begin{titlepage}

	{\centering
		{\scshape Universität Leipzig\\
			Fakultät für Mathematik und Informatik\\
			Mathematisches Institut\\
			Professur für Funktionalanalysis\\	 \par}
		\vspace{3cm}
		{\Huge\bfseries {"`Das Rückkehrzeitentheorem\\ von Bourgain"'}\par}
		\vspace{1cm}
		{\Huge Diplomarbeit
			\par}
		
		\vspace{4.5cm}}
	
	{{\Large
	
\noindent \begin{tabular}{p{6cm} l}
	\hspace{-0.25cm}Leipzig, April 2018 & vorgelegt von:\vspace{0.4cm}\\
	~ & Simon Fritzsch\\
	~ & Studiengang Wirtschaftsmathematik
\end{tabular}
}}
	
\vfill{\hspace{-0.4cm}\Large
\textbf{Betreuer:} Prof. Tatjana Eisner, Dr. Konrad Zimmermann}	

\end{titlepage}
 \setstretch{1.12}
\chapter*{Danksagung}
\addcontentsline{toc}{section}{Danksagung}
An dieser Stelle möchte ich all jenen danken, die durch ihre fachliche und persönliche Unterstützung zum Gelingen dieser Diplomarbeit beigetragen haben.\\
\noindent Ein ganz besonderer Dank gebührt Frau Professorin Tatjana Eisner, die mich von der Wahl des Diplomarbeitsthemas bis zum Abschluss dieser Arbeit unterstützte und bei Fragen jederzeit als Ansprechpartnerin zur Verfügung stand. Weiterhin herzlich bedanken möchte ich mich bei Herrn Dr. Zimmermann, für seine Hilfe bei offenen Fragen und die vielen Hinweise zur Verbesserung meiner Diplomarbeit.

	\pagenumbering{Roman}
	%\bibliographystyle{alphadin}	% Festlegung Art der Zitierung - Havardmethode: Abkuerzung Autor + Jahr
	%\addcontentsline{toc}{section}{Literatur} %bibliographie taucht im Inhaltsverzeichnis auf
	%\bibliography %Bibliographie erzeugen
	\setcounter{tocdepth}{2} % Gliederungstiefe erhöhen
	\setcounter{secnumdepth}{5} % Gliederungstiefe erhöhen
	\setstretch{1.03}
	\tableofcontents
		\addcontentsline{toc}{section}{Inhaltsverzeichnis}
	\setstretch{1.12}

	\clearpage
	\glsaddall
	\setglossarystyle{long}
	\setlength\LTleft{0pt}
\setlength\LTright{0pt}
\setlength\glsdescwidth{0.75\hsize}
%	\printnoidxglossaries
\printnoidxglossaries

	\chapter{Einleitung}
	\pagenumbering{arabic}
Im Jahr 1931 bewiesen von Neumann den Mittelergodensatz (\cite{Neumann1932}) und Birkhoff den punktweisen Ergodensatz (\cite{Birkhoff1931}) und begründeten damit die Ergodentheorie als eigenständiges mathematisches Forschungsfeld (vgl. \cite[S.4]{Eisner2014}). Während von Neumann Aussagen zur Konvergenz der Ergodenmittel in $L^2$ traf, konnte Birkhoff die folgende Aussage zur punktweisen Konvergenz dieser Mittel beweisen:\vspace{0.2cm}\\
\noindent \textbf{Punktweiser Ergodensatz}. \textit{Seien $(X,\mathfrak{A},\mu,T)$ ein maßtheoretisches dynamisches System und $f\in L^1(X,\mu)$. Dann gilt: $\lim_{N\to\infty}\frac{1}{N}\sum_{n=1}^N f(T^nx) \text{ existiert für }\mu\text{-fast alle }x\in X.$}\vspace{0.2cm}\\ 
\noindent Als Verallgemeinerung dieses Theorems wurde die Fragestellung nach universell guten Folgen für die punktweise Konvergenz aufgeworfen: Seien $\big(n_k\big)_{k=1}^\infty\subset \mathbb{N}$ eine streng ansteigende Folge natürlicher Zahlen sowie $1\leq p \leq \infty$. Dann bezeichnet man $\big(n_k\big)_{k=1}^\infty$ als eine \emph{universell gute Folge} für die punktweise Konvergenz von $L^p$"=Funktionen, falls für ein beliebiges maßtheoretisches dynamisches System $(X,\mathfrak{A},\mu,T)$ und eine beliebige Funktion $f\in L^p(X,\mu)$ der Grenzwert $\lim_{K\to\infty} \frac{1}{K}\sum_{k=1}^K f\big(T^{n_k}x\big)$ für $\mu$"=fast alle $x\in X$ existiert (vgl. \cite[S.309]{Bellow1985}). Der punktweise Ergodensatz besagt dann, dass die natürlichen Zahlen, verstanden als eine aufsteigende Zahlenfolge, universell gut für die punktweise Konvergenz von $L^1$"=Funktionen sind.\\ Als eines der klassischen Resultate in diesem Forschungsgebiet zeigte Jean Bourgain (\textborn 1954) in \cite{Bourgain1989b}, dass der ganzzahlige Teil einer durch ein beliebiges Polynom erzeugten Zahlenfolge universell gut für die punktweise Konvergenz von $L^p$"=Funktionen für $p>1$ ist. Wierdl bewies in \cite{Wierdl1988} eine analoge Aussage für die Folge der Primzahlen. Für eine geraume Zeit blieb die Frage,  welche Aussagen für $p=1$ zutreffend sind, ein offenes Problem. Aufbauend auf \cite{Buczolich2010} zeigte LaVictoire in \cite{LaVictoire2011}, dass die Folge der $d$'ten Potenzen und der Primzahlen universell schlecht für die punktweise Konvergenz von $L^1$"=Funktionen sind. Er konnte sogar zeigen, dass dies für jede Teilfolge dieser Folgen gilt. Dabei bezeichnet $\big(n_k\big)_{k=1}^\infty$ eine \emph{universell schlechte Folge} für die punktweise Konvergenz von $L^1$"=Funktionen, wenn für jedes nicht"=atomare, ergodische maßtheoretische dynamische System $(X,\mathfrak{A},\mu,T)$ eine Funktion $f\in L^1(X,\mu)$ derart existiert, dass die Folge $\big(\frac{1}{K}\sum_{k=1}^K T^{n_k}f\big)_{K=1}^\infty$ auf einer messbaren Teilmenge von $X$ mit positivem Maß divergiert (vgl. \cite[S.241]{LaVictoire2011}).\\
\noindent Eng verwandt mit der Erforschung universell guter Folgen ist die Untersuchung universell guter Gewichte (vgl. \cite[S.308-309]{Bellow1985}):\vspace{0.2cm}\\
\noindent {\textbf{Definition} (Universell gute Gewichte für die punktweise Konvergenz von $L^p$"=Funktionen).\nopagebreak\quad\textit{Sei $1\leq p \leq \infty$. Die Folge $(a_n)_{n=1}^\infty$ ist ein \emph{universell gutes Gewicht} für die punktweise Konvergenz von $L^p$"=Funktionen, wenn für ein beliebiges maßtheoretisches dynamisches System $(Y,\mathfrak{B},\nu,S)$ und eine beliebige Funktion $g\in L^p(Y,\nu)$ gilt: $$\lim_{N\to\infty}\frac{1}{N}\sum_{n=1}^N a_n g(S^ny) \text{ existiert für }\nu\text{-fast alle }y\in Y.$$}}
\noindent \hspace{-0.1cm}Für ein maßtheoretisches dynamisches System $(X,\mathfrak{A},\mu,T)$ lassen sich anstelle deterministischer Gewichte auch die von einer Menge $A\in \mathfrak{A}$\, "`zufällig"' erzeugten Gewichte $\big(\mathbbm{1}_A(T^nx)\big)_{n=1}^\infty,$ $x\in X$ betrachten. Da diese Folge genau dann den Wert $1$ annimmt, wenn $T^nx\in A$ erfüllt ist, bezeichnet man $\big(\mathbbm{1}_A(T^nx)\big)_{n=1}^\infty$ auch als \emph{Folge von Rückkehrzeiten}. Die Fragestellung, ob auf diese Weise erzeugte Gewichte universell gut für die punktweise Konvergenz von $L^p$"=Funktionen sind, wurde erstmalig in \cite{Brunel1966} untersucht (vgl.\cite[S.2]{Assani2017}). Brunel und Keane zeigten in \cite{Brunel1969}, dass die von ihnen eingeführten gleichmäßigen Folgen (\textit{uniform sequences}) universell gute Gewichte für die punktweise und die Norm-Konvergenz von $L^1$"=Funktionen sind. Ryll-Nardzewski verallgemeinerte dieses Resultat  und zeigte in \cite{Ryll1975}, dass die Menge beschränkter Besicovitch-Folgen, eine Menge universell guter Gewichte für die punktweise Konvergenz von $L^1$"=Funktionen ist.\\
\noindent Darauf aufbauend untersuchten Bellow und Losert in \cite{Bellow1985}, welche Eigenschaften von maßtheoretischen dynamischen Systemen sicherstellen, dass Folgen der Form $\big(f(T^nx)\big)_{n=1}^\infty$ universell gute Gewichte für die punktweise Konvergenz von $L^1$"=Funktionen sind. Unter anderem bewiesen sie, dass für einen Operator $T$ mit abzählbarem Lebesgue"=Spektrum und eine Funktion $f\in L^\infty(X,\mu)$ die Folge $\big(f(T^nx)\big)_{n=1}^\infty$ für $\mu$"=fast alle $x\in X$ ein universell gutes Gewicht für die punktweise Konvergenz von $L^1$"=Funktionen ist. Als zentrales Resultat in diesem Forschungsgebiet zeigte Bourgain   mithilfe fourieranalytischer Methoden in \cite{Bourgain1988} das folgende Theorem:\vspace{0.2cm}\\
\noindent \textbf{Rückkehrzeitentheorem.}\quad \textit{Seien $(X, \mathfrak{A}, \mu, T)$ ein ergodisches maßtheoretisches dynamisches System und $A\in \mathfrak{A}$ eine Menge mit $\mu(A)>0$. Dann ist für $\mu$-fast\ alle $x \in X$ die Folge 	 $\big(\mathbbm{1}_A(T^n x)\big)_{n=1}^{\infty}$ %oder n=0??
		ein universell gutes Gewicht für die punktweise Konvergenz von $L^1$"=Funktionen.}\vspace{0.2cm}\\
\noindent Das heißt, es existiert eine Menge $X'\in \mathfrak{A}$ mit $\mu(X')=1$, sodass für $x\in X'$, ein beliebiges weiteres maßtheoretisches dynamisches System $(Y,\mathfrak{B},\nu,S)$ und eine Funktion $g\in L^1$ gilt: $$\lim_{N\to\infty}\frac{1}{N}\sum_{n=1}^\infty \mathbbm{1}_A(T^nx)g(S^ny) \text{ existiert für }\nu\text{-fast alle }y\in Y.$$ \noindent Hervorzuheben ist, dass die Menge $X'\in \mathfrak{A}$ nur vom System $(X,\mathfrak{A},\mu,T)$ und der Funktion $\mathbbm{1}_A\in L^\infty(X,\mu)$, nicht jedoch vom System $(Y,\mathfrak{B},\nu,S)$  oder der Funktion $g\in L^1(Y,\nu)$ abhängt (vgl. \cite[S.23-25]{Assani2014}).  Die Aussagekraft des Rückkehrzeitentheorems ist aus diesem Grund stärker als die des  punktweisen Ergodensatzes angewendet auf das Produktsystem $({X\times Y}, {\mathfrak{A}\times \mathfrak{B}}, {\mu\otimes \nu}, {T\times S})$, da man mithilfe des punktweisen Ergodensatzes nur eine Menge $X'\in \mathfrak{A}$ erhält, welche von $(Y,\mathfrak{B},\nu,S)$ abhängt.\\
\noindent Zusammen mit Furstenberg, Katznelson und Ornstein veröffentlichte Bourgain 1989 einen weiteren Beweis des Rückkehrzeitentheorems (\cite{Bourgain1989}), auf welchen in der Literatur häufig als BFKO"=Beweis Bezug genommen wird. Wir werden diesen Beweis zunächst basierend auf \cite[S.72-99]{Assani2003} für den ergodischen Fall vorstellen und anschließend aufbauend auf eigenen Überlegungen das Rückkehrzeitentheorem in der folgenden, etwas allgemeineren Form zeigen:
	 \begin{mdframed}
	 	\textbf{Rückkehrzeitentheorem von Bourgain.}\quad \textit{Seien $(X, \mathfrak{A}, \mu, T)$ ein maßtheoretisches dynamisches System mit rein atomarer invarianter $\sigma$"=Algebra und $f \in L^\infty(X, \mu)$. 
	Dann ist die Folge $\big(f\left(T^n x\right)\big)_{n=1}^{\infty}$ für $\mu$-fast alle $x \in X$ 
	ein universell gutes Gewicht für die punktweise Konvergenz von $L^1$"=Funktionen.}
\end{mdframed}
Die wesentliche Idee des BFKO"=Beweises ist, Funktionen aus dem Kroneckerfaktor sowie aus dessen orthogonalem Komplement getrennt zu behandeln. Für den Beweis sind weitere Voraussetzungen an das System $(X,\mathfrak{A},\mu,T)$, wie zum Beispiel Ergodizität und Invertierbarkeit, sowie an die Funktion $f\in L^\infty(X,\mu)$ notwendig. Deshalb werden wir das Rückkehrzeitentheorem von Bourgain in den Kapiteln \ref{Kapitel2} und \ref{Kapitel3}  zunächst  in den Spezialfällen Satz \ref{Satz2.2b} und Satz \ref{Satz2.6} beweisen.  Eine Verallgemeinerung der Ergebnisse aus diesen beiden Kapiteln nehmen wir unter Verwendung der ergodischen Zerlegung sowie der  invertierbaren Erweiterung eines maßtheoretischen dynamischen Systems in Kapitel \ref{allgemein} vor. Die Beweise führen wir jeweils in mehreren aufeinander aufbauenden Schritten durch.

 \chapter{Vorbetrachtungen und Konventionen}
 \label{Kapitel2b}
\noindent In diesem Kapitel erläutern wir zunächst einige Konventionen und führen für den Beweis notwendige Resultate ein.\\
\noindent Wir bezeichnen mit $\mathbb{T}$ den Einheitskreis in der komplexen Zahlenebene sowie mit $\mathbbm{1}_A$ die charakteristische Funktion der Menge $A$. Weiterhin werden wir $\delta>0$ und $\epsilon>0$ beziehungsweise $\delta\in \mathbb{R}$ und $\epsilon \in \mathbb{R}$ notieren als $\delta,\epsilon>0$ beziehungsweise $\delta,\epsilon \in \mathbb{R}$. Für eine Abbildung $T$ und eine Menge $A$ bezeichnen wir mit $TA$ das Bild der Menge $A$ unter $T$ sowie mit $T^{-1}A$ das  Urbild von $A$. Weiterhin sei $T^0$ die identische Abbildung.\\
\noindent In dieser Arbeit verstehen wir maßtheoretische dynamische Systeme $(X,\mathfrak{A},\mu,T)$ als Systeme auf einem vollständigen Wahrscheinlichkeitsraum $(X,\mathfrak{A},\mu)$. Dabei heißt ein Maßraum vollständig, wenn jede Teilmenge einer $\mu$"=Nullmenge wieder ein Element der $\sigma$"=Algebra $\mathfrak{A}$ ist. Wir bezeichnen den zur Abbildung $T:X \to X$ gehörenden Koopmanoperator $L^1(X,\mu)\to L^1(X,\mu): \ f\mapsto f\circ T$ ebenfalls mit $T$. Ob es sich bei  $T$ um eine Abbildung $T:X\to X$ oder um den zugehörigen Koopmanoperator $T:L^1(X,\mu)\to L^1(X,\mu)$ handelt, wird jeweils zweifelsfrei aus dem Kontext hervorgehen. Wir weisen darauf hin, dass $L^\infty(X,\mu) \subseteq L^p(X,\mu)$, $1\leq p<\infty$ gilt, da $(X,\mathfrak{A},\mu,T)$ ein Wahrscheinlichkeitsraum ist.\\
\noindent Eine weitere Annahme bezieht sich auf das Verständnis einer Funktion $f$ als Funktion in $L^p$ oder in $\mathcal{L}^p$. Das Rückkehrzeitentheorem von Bourgain ist wie viele weitere punktweise Aussagen in dieser Arbeit für $L^p$"=Funktionen formuliert. Damit wir viele der Beweise sinnvoll führen können, ist es notwendig, für $L^p$"=Funktionen zu Beginn eines Beweises einen aus $\mathcal{L}^p$ stammenden Vertreter der Äquivalenzklasse zu fixieren. Eine explizite Wahl von Vertretern dieser Äquivalenzklassen nehmen wir in dieser Arbeit nicht vor. Dennoch sind punktweise Aussagen dahingehend zu verstehen, dass sie an einem zu Beginn des Beweises gedanklich fixierten Vertreter der jeweiligen Äquivalenzklasse ausgewertet werden. Für einfache Funktionen in $L^\infty(X,\mu)$ wählen wir einen Vertreter aus $\mathcal{L}^\infty(X,\mu)$, welcher nur endlich viele Werte annimmt.
\begin{Bemerkung}
In der Einleitung haben wir universell gute Gewichte für die punktweise Konvergenz von $L^p$"=Funktionen definiert. Gemeint ist damit die punktweise Konvergenz fast überall.
\end{Bemerkung}

\noindent Nachfolgend findet sich eine Zusammenstellung von Definitionen, Hilfsaussagen und Theoremen, welche wir zum Beweis des Rückkehrzeitentheorems benutzen werden. Damit diese Aussagen parallel zum Beweis einfach nachvollzogen werden können, sind sie in der Reihenfolge ihrer erstmaligen Verwendung angeordnet. Außerdem übernehmen wir aus diesem Grund  in der Formulierung vieler Aussagen die Bezeichnungen aus dem jeweiligen Kontext. Um den Umfang dieser Diplomarbeit nicht zu stark auszudehnen, geben wir zu vielen Resultaten lediglich Verweise auf die zugehörigen Beweise in der Literatur an.

\section*{Vorbetrachtungen zu Kapitel 1 und 3}

\begin{Definition}[Invariante $\sigma$"=Algebra]
	\label{invariant}
Sei $(X,\mathfrak{A},\mu,T)$ ein maßtheoretisches dynamisches System. Dann ist die invariante $\sigma$"=Algebra definiert als die $\sigma$"=Algebra $T$"=invarianter Mengen von $\mathfrak{A}$ gegeben durch$$\mathfrak{E}:=\big\{ A\in \mathfrak{A}|\ \mu\big((T^{-1}A) \triangle A\big)=0\big\}.$$
\end{Definition}

\begin{Definition}[Atome einer Maßalgebra, vgl. {\cite[S.65-66]{Elstrodt2011}}] \label{Atom} Sei $(X,\mathfrak{A},\mu)$ ein Maßraum. Eine Menge $A\in \mathfrak{A}$ mit $\mu(A)>0$ heißt ein \emph{$\mu$"=Atom}, wenn für jedes $B\in \mathfrak{A}$ mit $B\subseteq A$ gilt $\mu(B)=0$ oder $\mu(A\setminus B)=0$.\\
	\noindent $(X,\mathfrak{A},\mu)$ heißt \emph{rein  atomar}, wenn eine Menge $I\subseteq \mathbb{N}$ und eine Folge $(A_n)_{n\in I}$ von $\mu$"=Atomen existieren, sodass $\mu\big(\bigcup_{n\in I}A_n\big)=1$ erfüllt ist.
\end{Definition}
\begin{Bemerkung}
	Wir bezeichnen die invariante $\sigma$"=Algebra $\mathfrak{E}$ aus Definition \ref{invariant} als rein atomar, wenn $(X,\mathfrak{E},\mu_{\mid_{\mathfrak{E}}})$ rein atomar ist, wobei $\mu_{\mid_{\mathfrak{E}}}$ die Einschränkung von $\mu$ auf $\mathfrak{E}$ symbolisiere. \end{Bemerkung}

\begin{Definition}[Borel"=Wahrscheinlichkeitsraum]
	\label{BorelWRaum}
	Seien X eine Borel"=Teilmenge eines kompakten metrischen Raumes und $\mu$ ein Wahrscheinlichkeitsmaß, welches auf der Einschränkung der Borelschen $\sigma$"=Algebra auf $X$, bezeichnet durch $\mathfrak{A}$, definiert sei. Der resultierende Wahrscheinlichkeitsraum $(X,\mathfrak{A},\mu)$ ist ein \emph{Borel"=Wahrscheinlichkeitsraum}.
\end{Definition}
\begin{Bemerkung}
	\label{Bemerkung}
	Es ist üblich, die $\sigma$"=Algebra $\mathfrak{A}$ zur $\sigma$"=Algebra $\mathfrak{A}_\mu$ zu vervollständigen. Dies ist die kleinste $\sigma$"=Algebra, welche sowohl $\mathfrak{A}$ als auch alle Teilmengen von $\mu$"=Nullmengen enthält (vgl. \cite[S.411]{Einsiedler2010}). Da wir in dieser Arbeit von vollständigen Wahrscheinlichkeitsräumen ausgehen, ist bei der Annahme von Borel"=Wahrscheinlichkeitsräumen immer deren Vervollständigung gemeint.
\end{Bemerkung}

\begin{Satz}[Mittelergodensatz]
	\label{Mittelergodensatz}
	Seien $(X,\mathfrak{A},\mu,T)$ ein maßtheoretisches dynamisches System sowie $f\in L^2(X,\mu)$. $P_T$ bezeichne die orthogonale Projektion auf den abgeschlossenen linearen Teilraum $I=\big\{g\in L^2(X,\mu)|\ Tg=g\}\subseteq L^2(X,\mu)$. Dann gilt
	$$\lim_{N\to\infty}\Big|\Big|\frac{1}{N}\sum_{n=1}^N f(T^nx)-P_Tf\Big|\Big|_{L^2(X,\mu)}=0.$$
\end{Satz}
\begin{proof}[\textbf{Beweis:}]
Siehe \cite[S.32-34]{Einsiedler2010}.
\end{proof}
\begin{Bemerkung}
Es folgt insbesondere $\int_X f d\mu=\int_X P_Tfd\mu$.
\end{Bemerkung}

\begin{Satz}[Punktweiser Ergodensatz]
	\label{PunktweiserErgodensatz}
	Seien $(X,\mathfrak{A},\mu,T)$ ein maßtheoretisches dynamisches System sowie $f\in L^1(X,\mu)$. Dann konvergiert die Folge $\big( \frac{1}{N}\sum_{n=1}^N f(T^nx)\big)_{N=1}^\infty$ punktweise $\mu$"=fast überall sowie in $L^1(X,\mu)$ gegen eine Fixfunktion $f^*$ von $T$ mit $\int_X f^* d\mu=\int_X fd\mu.$
\end{Satz}

\begin{Bemerkung}
	\noindent Wenn $(X,\mathfrak{A},\mu,T)$ ergodisch ist, gilt $f^*(x)=\int_X fd\mu$ $\mu$"=fast überall. In der Literatur findet man den punktweisen Ergodensatz häufig mit den Mitteln  $\frac{1}{N}\sum_{n=0}^{N-1} f(T^nx)$. Es lässt sich jedoch leicht zeigen, dass die Folge $\big(\frac{1}{N}\sum_{n=0}^{N-1} f(T^nx)\big)_{N=1}^\infty$ genau dann konvergiert, wenn die Folge $\big(\frac{1}{N}\sum_{n=1}^N f(T^nx)\big)_{N=1}^\infty$ konvergiert. 
\end{Bemerkung}

\noindent Die folgenden Aussagen stammen aus \cite[S.25-26]{Assani2003}:

\begin{Definition}[Kroneckerfaktor]
	\label{DefKro}
	Sei $(X,\mathfrak{A}, \mu,T)$ ein maßtheoretisches dynamisches System. Für dieses System ist der \emph{Kroneckerfaktor $\mathcal{K}$} definiert als die in $L^2(X,\mu)$ abgeschlossene lineare Hülle der Eigenfunktionen von $T$.
\end{Definition}
\begin{Bemerkung}
\label{Bemerkung3}
Für eine beliebig gewählte Eigenfunktion $f\in L^2(X,\mu)$ von $T$ zum Eigenwert $\lambda$ gilt
$$T|f|=|Tf|=|\lambda f|=|\lambda||f|=|f|.$$ Also ist $|f|$ eine Fixfunktion von $T$. Wenn  $(X,\mathfrak{A},\mu,T)$ ergodisch ist, ist $|f|$ also $\mu$"=fast überall konstant (vgl. \cite[S.117]{Eisner2014}) und es gilt $f\in L^\infty(X,\mu).$
\end{Bemerkung}

\noindent Für die Projektion einer Funktion $f\in L^2(X,\mu)$ auf $\mathcal{K}$ lässt sich zeigen, dass diese als bedingte Erwartung von $f$ bezüglich einer Unter"=$\sigma$"=Algebra von $\mathfrak{A}$ darstellbar ist. Insbesondere gilt deshalb: Wenn $f$ eine Funktion aus $L^\infty(X,\mu)$ ist, dann sind auch deren Projektionen auf $\mathcal{K}$ beziehungsweise $\mathcal{K}^\bot$ Funktionen aus $L^\infty(X,\mu)$. Folglich existiert für eine beliebige Funktion $f\in L^\infty(X,\mu)$ eine Zerlegung 
\begin{equation}
\label{Kronecker}
f=f_1+f_2,\quad f_1\in\mathcal{K}\cap L^\infty(X,\mu),\ f_2\in \mathcal{K}^\bot \cap L^\infty(X,\mu).
\end{equation} 

\section*{Vorbetrachtungen zu Kapitel \ref{Kapitel3}}
\subsubsection*{Einige Aussagen zu Spektralmaßen}
%\noindent Die folgenden Aussagen über Spektralmaße benötigen wir, um in Kapitel \ref{Kapitel3} zu zeigen, dass die dort definierte Menge $X_1$ volles Maß hat.
\begin{Definition}
	\label{Mass}
	$M(\mathbb{T})$ bezeichne die Menge aller endlichen Borel"=Maße auf dem Einheitskreis $\mathbb{T}$.
\end{Definition}
\begin{Definition}[Fourierkoeffizienten eines Maßes, vgl. {\cite[S.35,38]{Katznelson2004}}]
	\label{Fourierkoeffizient}
	Sei $\sigma \in M(\mathbb{T})$. Für $n\in \mathbb{Z}$ ist der \emph{n'te Fourierkoeffizient} $\hat{\sigma}(n)$ von $\sigma$ definiert durch \vspace{-0.1cm}
	$$\hat{\sigma}(n):=\int\limits_{\mathbb{T}}z^{-n}d\sigma.\vspace{-0.1cm}$$
\end{Definition}

\begin{Definition}[Stetigkeit von Maßen, vgl. {\cite[S.44]{Katznelson2004}}]
	Sei $\sigma\in M(\mathbb{T})$. Dann heißt $\sigma$ \emph{stetig}, genau dann wenn $\sigma(\{z\})=0$ für alle $z\in \mathbb{T}$ erfüllt ist.\vspace{-0.3cm}
\end{Definition}

\begin{Satz}[Lemma von Wiener]
	\label{Satz2.3}
	Sei $\sigma \in M(\mathbb{T})$. Dann ist $\sigma$ genau dann stetig,  wenn $$\lim_{N \to \infty} \frac{1}{2N + 1} \sum_{-N}^N|\hat{\sigma}(n)|^2 = 0$$ erfüllt ist.
\end{Satz}
\begin{proof}[\textbf{Beweis:}]
	Siehe \cite[S.45]{Katznelson2004}.
\end{proof}
\noindent Die folgende Formulierung des Spektraltheorems ist aus \cite[S.420]{Einsiedler2010} entnommen:
\begin{Theorem}[Spektralmaß]
	\label{spectral}
	Sei $U$ ein unitärer Operator auf einem komplexen Hilbertraum $H$. Dann existiert für jedes Element $f\in H$ ein eindeutig bestimmtes endliches Borel"=Maß $\sigma_f \in M(\mathbb{T})$, sodass für alle $n\in \mathbb{Z}$
	$$\langle U^nf, f \rangle =\int\limits_{\mathbb{T}}z^n d\sigma_f(z)=\hat{\sigma}_f(-n)$$ erfüllt ist.
\end{Theorem}
\begin{proof}[\textbf{Beweis:}]
	Siehe \cite[S.39-40]{Katznelson2004}. \end{proof}
\noindent Weiterführende Informationen zu Spektralmaßen sind in \cite[S.17-19]{Nadkarni1998} zu finden.
\begin{Bemerkung}
	Wenn $(X,\mathfrak{A},\mu,T)$ ein invertierbares maßtheoretisches dynamisches System ist, dann ist der Koopmanoperator $T$ unitär (siehe \cite[S.29]{Einsiedler2010}) und wir können Theorem \ref{spectral} auf den Operator $T$ und eine beliebige Funktion $f$ aus dem Hilbertraum $L^2(X,\mu)$ anwenden.
\end{Bemerkung}
\noindent Dies motiviert die folgende Definition:
\begin{Definition}[Das zu $f$ gehörige Spektralmaß $\mu_f$]
	\label{DefSpek}
	Seien $(X,\mathfrak{A},\mu,T)$ ein invertierbares maßtheoretisches dynamisches System	und $f\in L^2(X,\mu)$. Dann ist das zu $f$ gehörige \emph{Spektralmaß} $\mu_f$ definiert als das eindeutig bestimmte Borel"=Wahrscheinlichkeitsmaß aus Theorem \ref{spectral} mit $\langle T^nf,f\rangle = \hat{\mu}_f(-n) \text{ für alle } n\in \mathbb{Z}$. 
\end{Definition}
\noindent Mithilfe dieser zu Funktionen aus $L^2(X,\mu)$ gehörigen Spektralmaße, können wir $\mathcal{K}^\bot$ spektral charakterisieren.
\begin{Satz}[Spektrale Charakterisierung von $\mathcal{K}^\bot$]
	\label{SpekChar}
	Sei $(X,\mathfrak{A},\mu,T)$ ein invertierbares, ergodisches maßtheoretisches dynamisches System. Dann gilt $f\in \mathcal{K}^\bot$ genau dann, wenn das zu $f$ gehörige Spektralmaß $\mu_f$ stetig ist.
\end{Satz}
\begin{proof}[\textbf{Beweis:}]
	Siehe \cite[S.27-28]{Assani2003}.
\end{proof}
\begin{Satz}
	\label{Satz2.2}
	Seien $(X, \mathfrak{A}, \mu, T)$ ein invertierbares maßtheoretisches dynamisches System sowie $f\in L^\infty(X, \mu)$. Dann sind die folgenden Aussagen äquivalent: 
	\begin{description}
		\item[(i)] Das zu $f$ gehörige Spektralmaß $\mu_f$ ist stetig. 
		\item[(ii)] $\lim\limits_{N\to \infty}\frac{1}{N}\sum\limits_{n=1}^N f(T^nx)\overline{f(T^n\xi)}=0\ \ \mu \otimes\mu$ - fast überall in $X\times X$.
	\end{description}
\end{Satz}

\begin{proof}[\textbf{Beweis:}] $(i) \Rightarrow (ii)$:\\
	Nach dem Mittelergodensatz angewendet auf das Produktsystem $(X \times X, \mathfrak{A}\times \mathfrak{A}, T\times T, \mu \otimes \mu)$ existiert eine Funktion $F \in L^2(X\times X, \mu \otimes \mu)$ derart, dass 
	$$\lim_{N\to \infty}\Big|\Big| \frac{1}{N}\sum_{n=1}^N f(T^nx)\overline{f(T^n\xi)} \ - \ F\Big|\Big|_{L^2(X\times X, \mu \otimes \mu)}=0$$ erfüllt ist.
	Folglich erhalten wir mithilfe des Satzes von Fubini	
	\begin{equation*}
	\label{Gleichung2.14}
	\begin{split}
	\big|\big|F\big|\big|_{L^2(X\times X, \mu \otimes \mu)}^2 &=\lim_{N\to \infty}\bigg|\bigg|\frac{1}{N}\sum_{n=1}^N f(T^nx)\overline{f(T^n\xi)}\bigg|\bigg|_{L^2(X\times X, \mu \otimes \mu)}^2\\
	&=\lim_{N\to\infty}\frac{1}{N^2}\sum_{n=1}^N\sum_{k=1}^N\bigg(\int\limits_X f(T^nx)\overline{f(T^kx)}
	d\mu(x)\bigg)\cdot\bigg(\int\limits_Xf(T^k\xi)\overline{f(T^n\xi})d\mu(y)\bigg)\\
	&=\lim_{N\to\infty}\frac{1}{N^2}\sum_{n=1}^N\sum_{k=1}^N\Big|\left<T^nf, T^kf\right>_{L^2(X,\mu)}\Big|^2\\
	&\hspace{-0.6cm}\overset{T\text{-Invarianz}}{\underset{\text{von }\mu }{=}}\lim_{N\to\infty}\frac{1}{N^2}\sum_{n=1}^N\sum_{k=1}^N\Big|\langle T^{n-k}f,f\rangle_{L^2(X,\mu)}\Big|^2\\
	&\leq\lim_{N\to\infty}\frac{1}{N^2}\sum_{n=1}^N\sum_{k=-N+1}^{N-1}\Big|\underbrace{\left<T^{k}f,f\right>_{L^2(X,\mu)}}_{=\hat{\mu}_f(k)}\Big|^2\\
	&\leq\lim_{N\to\infty}\frac{1}{N}\sum_{k=-N}^N\Big|\hat{\mu}_f(k)\Big|^2.\\
	\end{split}
	\end{equation*}
	\noindent Da $\mu_f$ nach Annahme stetig ist, ist nach Satz \ref{Satz2.3}\vspace{-0.2cm}
	$$\lim_{N\to\infty}\frac{1}{N}\sum_{k=-N}^N |\hat{\mu}_f(k)|^2=0\vspace{-0.1cm}$$ erfüllt.	Folglich gilt nach dem punktweisen Ergodensatz \vspace{-0.2cm}$$\lim_{N\to \infty}\frac{1}{N}\sum_{n=1}^N f(T^nx)\overline{f(T^n\xi)}= F(x,\xi) = 0 \ \mu\otimes\mu \text{-fast überall in }X\times X.$$
	\noindent $(ii) \Rightarrow (i)$: \\
	Sei nun $$\lim_{N\to \infty}\frac{1}{N}\sum_{n=1}^N f(T^nx)\overline{f(T^n\xi)}=0\  \mu \otimes\mu \text{-fast überall in } X\times X$$ erfüllt. Da nach Voraussetzung $f\in L^\infty(X,\mu)$ erfüllt ist, folgt mit dem Satz von der majorisierten Konvergenz

	\begin{equation}
	\begin{split}
	\label{Gleichung2.15}
	0&=\int\limits_{X\times X}\left(\lim_{N\to\infty}\frac{1}{N}\sum_{n=1}^N f(T^nx)\overline{f(T^n\xi)}\right)\cdot \overline{f(x)}f(\xi)d(\mu\otimes\mu)\\
	&\hspace{-0.23cm}\overset{\text{Satz v.}}{\underset{\text{Fubini}}{=}}\lim_{N\to\infty}\frac{1}{N}\sum_{n=1}^N\int\limits_X f(T^nx)\overline{f(x)}d\mu(x)\int\limits_Xf(\xi)\overline{f(T^n\xi)}d\mu(\xi)\\
	&=\lim_{N\to\infty}\frac{1}{N}\sum_{n=1}^N\underbrace{\left<T^nf, f\right>_{L^2(X, \mu)}}_{\hat{\mu}_f(-n)}\cdot\underbrace{\overline{\left<T^nf, f\right>_{L^2(X, \mu)}}}_{\overline{\hat{\mu}_f(-n)}}\\
	&=\lim_{N\to\infty}\frac{1}{N}\sum_{n=1}^N|\hat{\mu}_f(-n)|^2.
	\end{split}
	\end{equation}
	\noindent Nach Definition der Fourierkoeffizienten in Definition \ref{Fourierkoeffizient} gilt $\hat{\mu}_f(n)=\overline{\hat{\mu}_f(-n)}.$ Damit erhalten wir
	$$0=\lim_{N\to\infty}\frac{1}{N}\sum_{n=1}^N|\hat{\mu}_f(-n)|^2=\lim_{N\to\infty}\frac{1}{N}\sum_{n=1}^N|\overline{\hat{\mu}_f(n)}|^2=\lim_{N\to\infty}\frac{1}{2N}\sum_{k=-N}^N|\hat{\mu}_f(k)|^2.$$ Nach Satz \ref{Satz2.3} ist dann das  zu $f$ gehörige Spektralmaß $\mu_f$ stetig.
\end{proof}
%\noindent Mithilfe dieser Aussagen können wir nun Behauptung \ref{Behauptung3.1} beweisen. Den folgenden Satz benötigen wir zum Beweis der darauffolgenden beiden Behauptungen.
\subsubsection*{Der Satz von Jegorow}
\begin{Definition}[Fast gleichmäßige Konvergenz, vgl. {\cite[S.251]{Elstrodt2011}}]
Seien $(Y,\mathfrak{B},\nu)$ ein Maßraum und $(g_n)_{n=1}^\infty$ eine Folge messbarer Funktionen $g_n: Y\to \mathbb{K}$, welche punktweise $\nu$"=fast überall gegen die messbare Funktion $g$ konvergiert. Dann heißt die Folge $(g_n)_{n=1}^\infty$ \emph{fast gleichmäßig konvergent} gegen $g$, wenn zu jedem $\delta >0$ ein $B\in \mathfrak{B}$ mit $\nu(B)>1-\delta$ existiert, sodass $$\lim\limits_{\vphantom{y\in B} n\to\infty}\, \sup\limits_{y\in B} \big|g_n(y)-g(y)\big|=0$$ erfüllt ist. 
\end{Definition}

\begin{Satz}[Satz von Jegorow]
	\label{Jegorow}
Sei $(Y,\mathfrak{B},\nu)$ ein endlicher Maßraum. Weiterhin sei $(g_n)_{n=1}^\infty$ eine Folge messbarer Funktionen $g_n: Y\to \mathbb{K}$, welche punktweise $\nu$"=fast überall gegen die messbare Funktion $g\hspace{-0.05cm}:\hspace{-0.05cm} Y\hspace{-0.1cm}\to \mathbb{K}$ konvergiert. {Dann konvergiert $(g_n)_{n=1}^\infty$ fast gleichmäßig gegen $g$.} 
\end{Satz}
\begin{proof}[\textbf{Beweis:}]
Siehe \cite[S.252]{Elstrodt2011}.
\end{proof}
\begin{Behauptung}
	\label{Behauptung2.7} 
	Seien $(Y,\mathfrak{B},\nu,S)$ ein invertierbares, ergodisches maßtheoretisches dynamisches System. Seien weiterhin $0<\delta\in \mathbb{R}$ sowie $B\in \mathfrak{B}$ mit $\nu(B)>0$. Dann existiert ein $K=K(\delta)\in \mathbb{N}$ derart, dass $\nu\big(\bigcup_{j=1}^K S^jB\big)>1-\delta$ erfüllt ist.
\end{Behauptung} 
\begin{proof}[\textbf{Beweis:}]

	\noindent Wir führen den Beweis indirekt. Sei dazu angenommen,  dass ein $\delta>0$ existiert, sodass für alle $N\in \mathbb{N}$ \ $\nu\left(\bigcup_{j=1}^N S^jB \right)\leq 1-\delta$ erfüllt ist. \\
	\noindent Da die Abbildung $S$ nach Voraussetzung invertierbar ist, gilt $S^jB\in \mathfrak{B}$ für alle $j\in \mathbb{N}$. Weil für $N\leq M$ weiterhin $\bigcup_{j=1}^N S^jB\ \subseteq\ \bigcup_{j=1}^{M}S^jB$ erfüllt ist, gilt
	$$\nu\Big(\bigcup_{j=1}^\infty S^jB\Big)=\lim\limits_{N\to\infty}\,\nu\Big(\bigcup_{j=1}^N S^jB\Big)\leq 1-\delta$$ sowie 
	
	\begin{equation}
	\label{Gleichung2.24}
	\nu\bigg(\underbrace{Y\setminus \Big(\bigcup_{j=1}^\infty S^jB\Big)}_{=:C\in \mathfrak{B}}\bigg) \geq \delta.
	\end{equation}
	
	\noindent Aus der Definition von $C$ in \eqref{Gleichung2.24} folgt $S^{-j}C\cap B=\emptyset$ für alle $j\in \mathbb{N}$, da aus $y\in S^{-j}C$ folgt: $S^jy\in C\subseteq Y\setminus S^jB$. Für die Menge $$\mathfrak{B}\ni D:=\bigcup_{j=1}^\infty S^{-j}C$$ gilt deshalb
	\begin{equation}
	\label{Gleichung2.25}
	D\cap B=\bigcup_{j=1}^\infty S^{-j} C\ \cap B = \emptyset.
	\end{equation}
	
	\noindent Somit ist 
	\begin{equation} 
	\label{Gleichung3.19}
	0<\delta\overset{(\ref{Gleichung2.24})}{\leq}\nu(C)\leq \nu\Big(\bigcup_{j=1}^\infty S^{-j}C\Big)\overset{\text{Def.}}{=}\nu(D)\overset{\eqref{Gleichung2.25}}{\leq} 1-\underbrace{\nu(B)}_{>0}<1
	\end{equation}
	erfüllt. Wegen
	$S^{-1}D \subseteq D$ sowie 
	$\nu(S^{-1}D)=\nu(D)$ gilt
	$$\nu\big((S^{-1} D)\triangle D\big)=0.$$
	$D$ ist also eine $S$"=invariante Menge mod $\nu$, welche wegen \eqref{Gleichung3.19} weder eine Nullmenge ist noch volles Maß hat. Dies ist ein Widerspruch zur Ergodizität des Systems $(Y,\mathfrak{B},\nu,S)$. 
	Deshalb ist die zu Beginn des Beweises getroffene Annahme falsch. \end{proof}

\noindent Aufbauend auf Behauptung \ref{Behauptung2.7} sowie Satz \ref{Jegorow} beweisen wir die folgende Behauptung:
\begin{Behauptung}
	\label{Behauptung2.8}
	Seien $(Y,\mathfrak{B},\nu,S)$ ein invertierbares, ergodisches maßtheoretisches dynamisches System sowie $0<\delta, \delta'\in \mathbb{R}$ und $B\in \mathfrak{B}$ mit $\nu(B)>0$ beliebig. Dann existiert ein $K=K(\delta')\in \mathbb{N}$, sodass für die Funktion $\psi:= \mathbbm{1}_{\bigcup_{j=1}^K S^jB}$ gilt: Es existieren eine Menge $G\in \mathfrak{B}$ mit $\nu(G)>1-\delta$ sowie ein $M_0 \in \mathbb{N}$, sodass für alle $y\in G$ und alle $N>M_0$ 
	\begin{equation}
	\label{Gleichung2.26}
	\bigg|\frac{1}{N}\sum_{n=1}^N \psi(S^ny)-1\bigg|<\frac{\delta '}{4}
	\end{equation}
	erfüllt ist. 
\end{Behauptung}

\begin{proof}[\textbf{Beweis:}]
	Nach Behauptung \ref{Behauptung2.7} existiert ein $K\in \mathbb{N}$, sodass
	\begin{equation}
	\label{Gleichung3.20}
	\nu\big(\bigcup_{j=1}^K S^j B \big)>1-\frac{\delta'}{8}
	\end{equation}
	erfüllt ist. 
	\noindent Da $(Y, \mathfrak{B}, \nu, S)$ ergodisch sowie $\psi$ ein Element aus $ L^1(Y, \nu)$ ist, gilt nach dem punktweisen Ergodensatz für $\nu$-fast alle $y\in Y$ 
	$$\lim\limits_{N\to\infty} \frac{1}{N} \sum_{n=1}^N\psi(S^ny)=\int\limits_Y \psi\  d\nu=\int\limits_{Y}\mathbbm{1}_{\bigcup_{j=1}^K S^jB}\ d\nu = \nu\bigg(\bigcup_{j=1}^K S^jB\bigg)\overset{\eqref{Gleichung3.20}}{>}1-\frac{\delta'}{8}.$$
	\noindent Weil die Funktionenfolge $\Big(\frac{1}{N}\sum_{n=1}^N\psi(S^ny)\Big)_{n=1}^\infty$ $\nu$"=fast überall gegen $\nu\big(\bigcup_{j=1}^K S^j B\big)$ konvergiert, existiert nach Satz \ref{Jegorow} eine Menge $G\in \mathfrak{B}$ mit $\nu(G)>1-\delta$ derart, dass\vspace{-0.1cm} $$\lim\limits_{N\to\infty}\sup\limits_{y\in G} \bigg|\frac{1}{N}\sum_{n=1}^N \psi(S^ny)-\nu\bigg(\bigcup_{j=1}^K S^jB\bigg)\bigg|=0\vspace{-0.1cm}$$ erfüllt ist. Folglich existiert ein $M_0\in \mathbb{N}$, sodass für alle $N>M$ gilt:
	\begin{equation}
	\label{Gleichung2.27}
	\sup\limits_{y\in G}\bigg|\frac{1}{N}\sum_{n=1}^N \psi(S^ny)-\nu\bigg(\bigcup_{j=1}^K S^jB\bigg)\bigg|<\frac{\delta'}{8}.\vspace{-0.3cm}
	\end{equation}
	Deshalb ist für alle $y\in G$ und alle $N >M_0$ 
	\begin{equation*}
	\bigg|\frac{1}{N}\sum_{n=1}^N \psi(S^ny)-1\bigg| \leq \underbrace{\bigg|\frac{1}{N}\sum_{n=1}^N\psi(S^n y)-\nu\bigg(\bigcup_{j=1}^KS^j B\bigg)\bigg|}_{<\frac{\delta'}{8} \text{ wegen (\ref{Gleichung2.27})}} + \underbrace{\bigg|\nu\bigg(\bigcup_{j=1}^KS^j B\bigg)-1 \bigg|}_{<\frac{\delta'}{8} \text{ wegen} \eqref{Gleichung3.20}}<\frac{\delta'}{4}
	\end{equation*}
	erfüllt.
\end{proof}
\begin{Bemerkung}
	Die Behauptung ist für $\frac{\delta'}{4}$ formuliert, um die Wahl der Konstanten am Anfang von Abschnitt \ref{SchrittV} zu vereinfachen. Die Behauptung lässt sich natürlich auch für $\delta'$ anstelle von $\frac{\delta'}{4}$ beweisen.
\end{Bemerkung}
\subsubsection*{Approximation durch einfache Funktionen}
%\noindent Das folgende Hilfsresultat zur Approximierbarkeit in $L^\infty(X,\mu)$ von Funktionen $f\in L^\infty(X,\mu)$ durch  einfache Funktionen benötigen wir in Kapitel \ref{unendlich}.
\begin{Behauptung}
	\label{Behauptung2.10}
	Seien $(X,\mathfrak{A},\mu)$ ein Maßraum sowie $f\in L^\infty(X, \mu)$. Dann existiert eine Folge $\big(f_k\big)_{k=1}^\infty\subset L^\infty(X, \mu)$ von einfachen Funktionen, sodass $$\lim\limits_{k\to\infty}||f-f_k||_{L^\infty(X, \mu)}=0$$
	erfüllt ist.
\end{Behauptung}
\begin{proof}[\textbf{Beweis:}]
	Wir betrachten die Zerlegung $f=\operatorname{Re}(f) + i \operatorname{Im}(f)$ von $f$ in Real- und Imaginärteil. $\operatorname{Re}(f)$ und $\operatorname{Im}(f)$ sind messbare Funktionen, für welche $||\operatorname{Re}(f)||_{L^\infty(X, \mu)}\leq ||f||_{L^\infty(X, \mu)}$ sowie $||\operatorname{Im}(f)||_{L^\infty(X, \mu)}\leq||f||_{L^\infty(X, \mu)}$ erfüllt ist. Für $k\in \mathbb{N}$ definieren wir mithilfe der oberen Gaußklammer $N_k:=\left \lceil k \cdot ||f||_{L^\infty(X, \mu)}  \right\rceil$ 
	sowie $f_k^{Re}$ als die einfache Funktion $$f_{k}^{Re}:=\sum\limits_{j=-N_k}^{N_k} \frac{j}{k}\cdot \mathbbm{1}_{\left(\operatorname{Re}(f)\right)^{-1}[\frac{j}{k}, \frac{j+1}{k})}.$$
	Analog dazu definieren wir $f_{k}^{Im}$. Nach Konstruktion gilt $||\operatorname{Re}(f)-f_k^{Re}||_{L^\infty(X, \mu)}\leq \frac{1}{k}$ sowie $||\operatorname{Im}(f)-f_k^{Im}||_{L^\infty(X, \mu)}\leq \frac{1}{k}.$ Für die einfache Funktion $$f_k:=f_k^{Re}+i f_k^{Im} \in L^\infty(X, \mu)$$ ist folglich  
	$||f-f_k||_{L^\infty(X, \mu)} \leq \frac{2}{k}$
	erfüllt. Damit ist die Behauptung bewiesen.
\end{proof}

\section*{Vorbetrachtungen zu Kapitel \ref{allgemein}}
%\textbf{Die folgenden Überlegungen werden in den Schritten I und V in Kapitel  \ref{allgemein} genutzt.}
\subsubsection*{Invertierbare Erweiterung und ergodische Zerlegung}
\begin{Satz}[Invertierbare Erweiterung, vgl. {\cite[S.20]{Einsiedler2010}}]
	\label{Satz99}
	Sei $(Y,\mathfrak{B},\nu,S)$ ein maßtheoretisches dynamisches System. Definiere $(\tilde{Y},\tilde{\mathfrak{B}},\tilde{\nu}, \tilde{S})$ durch 	
	\begin{enumerate}
		\item[(1)] $\tilde{Y}=\big\{y\in Y^{\mathbb{Z}}\big| y_{k+1}=Sy_k \text{ für alle } k\in \mathbb{Z} \big\}$,
		\item[(2)] $\big(\tilde{S}y\big)_k=y_{k+1}$ für alle $k\in \mathbb{Z}$ und $y\in \tilde{Y}$ (Linksshift),
		\item[(3)] $\tilde{\nu}(\{y\in \tilde{Y}| y_k \in B \})=\nu(B)$ für alle $B\in \mathfrak{B}$ und $k\in \mathbb{Z}$,
		\item[(4)] $\tilde{\mathfrak{B}}$ die von den Zylindermengen $ \{y\in \tilde{Y}\big| y_k \in B\} \ k\in \mathbb{Z}, B\in \mathfrak{B}$ erzeugte $\sigma$"=Algebra.
	\end{enumerate}
	Dann ist $(\tilde{Y},\tilde{\mathfrak{B}},\tilde{\nu},\tilde{S})$ eine invertierbare Erweiterung des Systems $(Y,\mathfrak{B}, \nu,S).$	Weiterhin gilt: Wenn $(Y,\mathfrak{B},\nu,S)$ ergodisch ist, dann ist auch $(\tilde{Y},\tilde{\mathfrak{B}},\tilde{\nu},\tilde{S})$ ergodisch.
\end{Satz}
\begin{Bemerkung}
In \cite[S.249-254]{Eisner2014} wird ein anderer Zugang gewählt. Die invertierbare Erweiterung wird dort mittels induktiver Limites konstruiert.
\end{Bemerkung}
\noindent Auf den Beweis, dass es sich bei $(\tilde{Y}, \tilde{\mathfrak{B}}, \tilde{\nu}, \tilde{S})$ um ein invertierbares maßtheoretisches dynamisches System sowie bei der Abbildung $$\pi: \tilde{Y}\to Y: y \mapsto y_0$$ um eine Faktorabbildung handelt, werden wir aus Gründen des Umfangs verzichten. Den Zusammenhang veranschaulicht das folgende Diagramm:
\begin{equation*} \begin{tikzcd}
\tilde{Y} \arrow{d}[swap]{\pi} \arrow{r}{\tilde{S}} & \tilde{Y} \arrow{d}{\pi} \\
Y\arrow{r}{S} & Y 
\end{tikzcd}\end{equation*}
Für einen Beweis dafür, dass $(\tilde{Y},\tilde{\mathfrak{B}}, \tilde{\nu},\tilde{S})$ ergodisch ist, wenn $(Y,\mathfrak{B},\nu,S)$ ergodisch ist, verweisen wir auf \cite[S.252-253]{Eisner2014}. 

\begin{Theorem}[Ergodische Zerlegung]
	\label{Zerlegung}
	Sei $S: (Y, \mathfrak{B}, \nu) \to (Y, \mathfrak{B}, \nu)$ eine maßerhaltende Abbildung auf einem Borel"=Wahrscheinlichkeitsraum. Weiterhin bezeichne $M_1(Y,\mathfrak{B})$ die Menge aller Borel"=Wahrscheinlichkeitsmaße auf $(Y,\mathfrak{B})$. Dann gibt es einen Borel"=Wahrscheinlichkeitsraum $(Z, \mathfrak{B}_Z, \lambda)$ und eine messbare Abbildung $Z\to M_1(Y,\mathfrak{B})$, $z \mapsto \nu_z$, sodass  $\nu_z$ für $\lambda$-fast alle $z\in Z$ ein $S$-invariantes, ergodisches Wahrscheinlichkeitsmaß auf $Y$ ist und weiterhin \begin{equation} \label{Gleichung5.1}\int_Y h d\nu = \int_Z \int_Y h d\nu_zd\lambda(z) \text{ für alle } h\in L^\infty(Y,\nu)\end{equation} erfüllt ist.
\end{Theorem}
\begin{proof}[\textbf{Beweis:}]
	Siehe \cite[S.154-156]{Einsiedler2010}.
\end{proof}

%\subsubsection{Schritt III/ Metrisches Modell}
\subsubsection*{Metrisches Modell}
Zunächst benötigen wir die folgende Definition:
\begin{Definition}[Markov"=Isomorphismus, siehe {\cite[S.217]{Eisner2014}}]
	\label{Markov}
	Seien $(X_1,\mathfrak{B}_1,\mu_1)$ und $(X_2,\mathfrak{B}_2,\mu_2)$ zwei Wahrscheinlichkeitsräume. Ein linearer Operator $$\Phi: L^1(X_1,\mu_1) \to L^1(X_2,\mu_2)$$ ist ein \emph{Markov"=Operator}, falls 
	
	{$$\Phi\geq 0, \quad \Phi\mathbbm{1}_{X_1}=\mathbbm{1}_{X_2}\text{ und} $$
		\begin{equation}
		\label{Gleichung4.16}\int\limits_{X_2}\Phi h \,d\mu_2 = \int\limits_{X_1}h\, d\mu_1 \text{ für alle }h\in L^1(X_1,\mu_1)\end{equation}} \noindent  erfüllt sind. Eine \emph{Markov"=Einbettung} ist ein Markov"=Operator, für welchen zusätzlich \begin{equation}
	\label{Gleichung4.18}
	|\Phi h|=\Phi |h| \text{ für alle }h\in L^1(X_1,\mu_1)\end{equation} gilt. Ein \emph{\textbf{Markov"=Isomorphismus}} ist eine surjektive Markov"=Einbettung.
\end{Definition}	
\noindent Mithilfe dieser Definition können wir nun das folgende Theorem formulieren:
\begin{Theorem} 
	\label{MetrischesModell}
	\noindent Sei $(Y,\mathfrak{B},\nu,S)$ ein maßtheoretisches dynamisches System, sodass $L^1(Y,\nu)$ ein separabler Banachraum ist. Dann existiert ein maßtheoretisches dynamisches System $$(K,\mathfrak{B}_K,\lambda,\tilde{S})$$ mit der Eigenschaft, dass $(K,\mathfrak{B}_K,\lambda)$ ein Borel"=Wahrscheinlichkeitsraum ist. Weiterhin existiert ein Markov-Isomorphismus $\Phi: L^1(Y,\nu) \to L^1(K,\lambda)$, für welchen $$\Phi \circ S = \tilde{S}\circ \Phi$$ für die zu $S$ beziehungsweise $\tilde{S}$ gehörigen Koopmanoperatoren erfüllt ist.
\end{Theorem}
\noindent Aufgrund des Umfangs der diesem Theorem zugrundeliegenden Sätze und Definitionen in \cite[S.215-225]{Eisner2014} verzichten wir in dieser Arbeit auf deren ausführliche Darstellung und geben lediglich eine Beweisidee an:
\begin{proof}[\textbf{Beweisidee:}]
\noindent Nach \cite[S.224]{Eisner2014} existiert für $(Y,\mathfrak{B},\nu,S)$ ein metrisches Modell genau dann, wenn $L^1(Y,\nu)$ ein separabler Banachraum ist. Da $L^1(Y,\nu)$ nach Voraussetzung separabel ist, existiert nach der Definition eines metrischen Modells ein Wahrscheinlichkeitsraum $(K,\mathfrak{B}_K,\lambda)$ sowie ein Markov"=Isomorphismus $\Phi:L^1(Y,\nu)\to L^1(K,\lambda)$, wobei $K$ eine kompakte Menge, $\mathfrak{B}_K$ von einer Metrik erzeugt und $\lambda$ ein Baire"=Wahrscheinlichkeitsmaß ist. Nach \cite[S.75]{Eisner2014} fallen unter diesen Voraussetzungen die Begriffe Baire"= und Borelmengen sowie Baire"= und Borelmaße zusammen. Eine Vervollständigung des Wahrscheinlichkeitsraums $(K,\mathfrak{B}_K,\lambda)$ ergibt deshalb einen Borel"=Wahrscheinlichkeitsraum im Sinne der Definition \ref{BorelWRaum} in Verbindung mit Bemerkung \ref{Bemerkung}.
\end{proof}

\noindent Wir benötigen weiterhin das folgende Hilfsresultat:
\begin{Behauptung}
	\label{Behauptung17}
	Seien $(X_1,\mathfrak{B_1},\mu_1)$ und $(X_2, \mathfrak{B_2},\mu_2)$ zwei Wahrscheinlichkeitsräume sowie $\Phi: L^1(X_1,\mu_1)\to L^1(X_2,\mu_2)$ eine Markov"=Einbettung. Weiterhin sei $(f_n)_{n=1}^\infty\subseteq L^1(X_1,\mu_1)$ eine Folge von Funktionen, welche durch eine Funktion $g\in L^1(X_1,\mu_1)$ majorisiert wird. Dann gilt: $(f_n)_{n=1}^\infty$ konvergiert punktweise $\mu_1$"=fast überall gegen $0$ genau dann, wenn $(\Phi f_n)_{n=1}^\infty$ punktweise $\mu_2$"=fast überall gegen $0$ konvergiert.
\end{Behauptung}
\begin{proof}[\textbf{Beweis:}]
Wegen \eqref{Gleichung4.16} und \eqref{Gleichung4.18} ist $\Phi g$ eine $\mu_2$"=integrierbare Majorante für die Folge $(\Phi f_n)_{n=1}^\infty$. Falls $(f_n)_{n=1}^\infty$ punktweise $\mu_1$"=fast überall gegen $0$ konvergiert, erhalten wir durch zweimalige Anwendung des Satzes von der majorisierten Konvergenz
\begin{equation*}
\begin{split}
0=\lim_{n\to\infty}\int\limits_{X_1}|f_n|d\mu_1
\overset{\eqref{Gleichung4.16}}{=}\lim_{n\to\infty}\int\limits_{X_2}\Phi|f_n|d\mu_2
\overset{\eqref{Gleichung4.18}}{=}\lim_{n\to\infty}\int\limits_{X_2}|\Phi f_n|d\mu_2
=\int\limits_{X_2}\lim_{n\to\infty}|\Phi f_n|d\mu_2.
\end{split}
\end{equation*} Folglich konvergiert $(\Phi f_n)_{n=1}^\infty$ punktweise $\mu_2$"=fast überall gegen $0$.
Die Rückrichtung erhalten wir durch Umkehrung der Gleichungskette.
\end{proof}

%\subsubsection{Schritt IV: Separabilität von $L^1(Y,\nu)$}
\subsubsection*{Separabilität von $L^1$}
%In diesem Abschnitt zeigen wir, dass für einen beliebigen Wahrscheinlichkeitsraum $(Y,\mathfrak{B},\nu)$ mit abzählbar erzeugter $\sigma$"=Algebra $\mathfrak{B}$ der Raum $L^1(Y,\nu)$ ein separabler Banachraum ist. Genauer zeigen wir die folgende Behauptung gezeigt:
\begin{Behauptung}
	\label{Behauptung4}
	Sei $(Y,\mathfrak{B},\nu)$ ein Wahrscheinlichkeitsraum mit abzählbar erzeugter $\sigma$-Algebra $\mathfrak{B}$. Dann ist $L^1(Y,\nu)$ ein separabler Banachraum.
\end{Behauptung}
\begin{proof}[\textbf{Beweis:}]
	\noindent Nach dem Satz von Riesz"=Fischer ist $L^1(Y,\nu)$ ein Banachraum. Es genügt also, eine abzählbare dichte Teilmenge von $L^1(Y,\nu)$ zu finden. \\
	Nach Voraussetzung ist $\mathfrak{B}$ abzählbar erzeugt. Sei der Erzeuger gegeben durch $$C:=\big\{A_k|\ k\in \mathbb{N}\big\}.$$ Es bezeichne $\mathfrak{C}$ die von $C$ erzeugte Algebra. Da $C$ eine abzählbare Menge ist, ist auch $\mathfrak{C}$ abzählbar. Folglich ist auch die Menge  $$F:=\Big\{\sum_{k=1}^N a_k\mathbbm{1}_{C_k}\big| \ C_k\in \mathfrak{C}, a_k\in \mathbb{Q}\cup i \cdot \mathbb{Q}, N\in \mathbb{N}\Big\}$$ von rationalen Treppenfunktionen abzählbar. 
	Sei nun $\epsilon>0$ sowie $h\in L^1(Y,\nu)$ beliebig gewählt. Wir können ohne Beschränkung der Allgemeinheit annehmen, dass $h$ eine nichtnegative reelle Funktion ist. (Sonst können wir Positiv"= und Negativteil von Real"= und Imaginärteil separat betrachten.) Dann existiert eine Treppenfunktion\vspace{-0.1cm} $$h_\epsilon := \sum_{k=1}^N b_k\mathbbm{1}_{B_k}$$\vspace{-0.1cm} mit $0<b_k\in \mathbb{R}, B_k\in \mathfrak{B}$ für $k=1, \dots, N$, sodass \begin{equation}
	\label{Gleichung4.18b}
	||h-h_\epsilon||_{L^1(Y,\nu)}=\int\limits_Y |h-h_\epsilon| d\nu < \frac{\epsilon}{2}
	\end{equation}
	\noindent erfüllt ist. Wähle nun für $k=1, \dots, N$\ \  $0<a_k\in \mathbb{Q}$ so, dass 
	\begin{equation}
	\label{Gleichung4.19}
	|a_k-b_k|<\frac{\epsilon}{4N}
	\end{equation}
	gilt. Da die Algebra $\mathfrak{C}$ den gleichen Erzeuger wie die $\sigma$"=Algebra $\mathfrak{B}$ hat, können wir für $k\in \{1, \dots, N\}$ Mengen $C_k\in \mathfrak{C}$ so wählen, dass \begin{equation}
	\label{Gleichung4.20}
	\nu(C_k\triangle B_k)<\frac{\epsilon}{4a_k\cdot N}
	\end{equation}
	erfüllt ist (siehe \cite[S.452]{Durrett1996}).  Definiere
$\tilde{h}_\epsilon:=\sum_{k=1}^N a_k \mathbbm{1}_{C_k}\in F.$
	Dann gilt 
	\begin{equation*}
	\begin{split}
	||h_\epsilon-\tilde{h}_\epsilon||_{L^1(Y,\nu)}&=\int\limits_Y \Big|\sum_{k=1}^N b_k \mathbbm{1}_{B_k}-\sum_{k=1}^N a_k\mathbbm{1}_{C_k}\Big|d\nu\\
	&\leq \int\limits_Y \sum_{k=1}^N |b_k-a_k|\cdot\mathbbm{1}_{B_k}d\nu + \int\limits_Y\sum_{k=1}^N |a_k|\cdot |\mathbbm{1}_{B_k}-\mathbbm{1}_{C_k}|d\nu\\
	&\leq \sum_{k=1}^N |b_k-a_k| +\sum_{k=1}^N |a_k|\cdot \nu\big(C_k\triangle B_k\big)\\
	&\hspace{-0.2cm}\overset{\eqref{Gleichung4.19}}{\underset{\eqref{Gleichung4.20}}{\leq}}\frac{\epsilon}{2}.
	\end{split}
	\end{equation*}
	Zusammen mit \eqref{Gleichung4.18b} folgt
	$$||h-\tilde{h}_\epsilon||_{L^1(Y,\nu)}\leq ||h-h_\epsilon||_{L^1(Y,\nu)}+||h_\epsilon -\tilde{h}_{\epsilon}||_{L^1(Y,\nu)}\leq\epsilon.$$
	Da wir $\epsilon>0$ beliebig gewählt haben, ist die Behauptung damit bewiesen.
\end{proof}
\begin{Bemerkung}
	\label{Bemerkung5.5}
	Die Behauptung bleibt wahr, wenn wir zur Vervollständigung $(Y,\overline{\mathfrak{B}}, \overline{\nu})$ eines Wahrscheinlichkeitsraumes $(Y,\mathfrak{B},\nu)$ mit abzählbar erzeugter $\sigma$"=Algebra $\mathfrak{B}$ übergehen, da $L^1(Y,\mathfrak{B},\nu)$ isomorph zu $L^1(Y,\overline{\mathfrak{B}},\overline{\nu})$ ist. Der Isomorphismus $\phi: L^1(Y,\mathfrak{B},\nu)\to L^1(Y,\overline{\mathfrak{B}},\overline{\nu})$ ist gegeben durch $L^1(Y,\mathfrak{B},\nu)\ni[h]_{\mathfrak{B}}\mapsto [h]_{\overline{\mathfrak{B}}}\in L^1(Y,\overline{\mathfrak{B}},\overline{\nu}).$ Das heißt, die Äquivalenzklasse $[h]_{\mathfrak{B}}$ wird abgebildet auf die Äquivalenzklasse $[h]_{\overline{\mathfrak{B}}}$ in $L^1(Y,\overline{\mathfrak{B}},\overline{\nu})$. Die Wohldefiniertheit, Injektivität, Linearität sowie Stetigkeit von $\phi$ lassen sich leicht zeigen, ein Beweis für die Surjektivität der Abbildung sowie für die Stetigkeit der Umkehrabbildung entfällt aus Gründen des Umfangs. \end{Bemerkung}
\subsubsection*{Zerlegung in Atome}
\begin{Behauptung}
	\label{Behauptung9}
	Sei $(X,\mathfrak{A},\mu)$ ein Maßraum, sodass $L^\infty(X,\mu)$ separabel ist. Dann ist $(X,\mathfrak{A},\mu)$ rein atomar.
\end{Behauptung}
\begin{proof}[\textbf{Beweis:}]
Wir nehmen an, $(X,\mathfrak{A},\mu)$ wäre nicht rein atomar. Dann existiert eine Menge $A\in \mathfrak{A}$ mit $\mu(A)>0$, welche keine $\mu$"=Atome enthält.\\ 
\noindent Nach \cite[S.56]{Bogacev2007} existiert  für jedes $\alpha\in [0,\mu(A)]$ eine Menge $A_\alpha\in \mathfrak{A}$, für welche $\mu(A_\alpha)=\alpha$ erfüllt ist. Wir wählen für jedes $\alpha\in [0, \mu(A)]$ eine solche Menge $A_\alpha$ und definieren $$f_\alpha:=\mathbbm{1}_{A_\alpha}\in L^\infty(X,\mu).$$ Für beliebige $\alpha_1, \alpha_2\in [0,\mu(A)],\, \alpha_1\neq \alpha_2$ unterscheiden sich die Funktionen $f_{\alpha_1}$ und $f_{\alpha_2}$ auf einer Menge mit positivem Maß. Folglich gilt $||f_{\alpha_1}-f_{\alpha_2}||_{L^\infty(X,\mu)}=1.$ Da die Menge $[0, \mu(A)]$ überabzählbar ist, ist  $L^\infty(X,\mu) \supseteq \{f_\alpha| \ \alpha\in [0,\mu(A)]\} $ nicht separabel. Dies steht im Widerspruch zur Voraussetzung. Folglich ist die Annahme, $(X,\mathfrak{A},\mu)$ ist nicht  rein atomar, falsch und die Behauptung damit bewiesen.
\end{proof}

\begin{Bemerkung}
Die umgekehrte Beweisrichtung gilt im Allgemeinen nicht.
\begin{proof}[\textbf{Beweis:}]
Sei $(X,\mathfrak{A},\mu)$ ein rein atomarer Maßraum mit abzählbar unendlich vielen paarweise disjunkten $\mu$"=Atomen $A_n\in \mathfrak{A}$, $n\in \mathbb{N}$. Seien $I,J\subseteq \mathbb{N}$, $I\neq J$ zwei Teilmengen der natürlichen Zahlen. Dann gilt für die Funktionen $f:=\sum_{n\in J} \mathbbm{1}_{A_n}\in L^\infty(X,\mu)$ und $g:=\sum_{n\in I} \mathbbm{1}_{A_n}\in L^\infty(X,\mu)$\vspace{-0.2cm} $$||f-g||_{L^\infty(X,\mu)}=1.$$ Da die natürlichen Zahlen überabzählbar viele Teilmengen enthalten, ist $L^\infty(X,\mu)$ folglich nicht separabel.
\end{proof} 
\end{Bemerkung}

\noindent Wir betrachten nun wieder maßtheoretische dynamische Systeme $(X,\mathfrak{A},\mu,T)$ und untersuchen Maße, welche wir durch Einschränkung  von $\mu$ auf bestimmte Teilmengen von $X$ erhalten.
\begin{Behauptung}
	\label{Behauptung16}
	Seien $(X,\mathfrak{A},\mu,T)$ ein maßtheoretisches dynamisches System und $A\in \mathfrak{A}$ ein $T$-invariantes $\mu$"=Atom. Dann ist das Maß $\mu_A$ gegeben durch \vspace{-0.15cm}$$\mu_A(B):=\frac{\mu(B\cap A)}{\mu(A)},\quad B\in \mathfrak{A}$$ ein $T$-invariantes, ergodisches Wahrscheinlichkeitsmaß.
\end{Behauptung}
\begin{proof}[\textbf{Beweis:}]
Da die Menge $A$ nach der Definition von $\mu$-Atomen positives Maß hat, ist die Abbildung $\mu_A: \mathfrak{A}\to [0,1]$ wohldefiniert und offensichtlich ein Wahrscheinlichkeitsmaß. Wir zeigen zunächst, dass $\mu_A$ $T$-invariant ist. Sei dazu eine beliebige Menge $B\in \mathfrak{A}$ gegeben. Da $A$ $T$-invariant ist, gilt $\mu\big(A \triangle (T^{-1}A)\big)=0$. Damit erhalten wir\vspace{-0.15cm}
$$\mu_A\big(T^{-1}B\big)\hspace{-0.09cm}=\hspace{-0.09cm}\frac{\mu\big((T^{-1}B) \cap A\big)}{\mu(A)}\hspace{-0.09cm}=\hspace{-0.09cm}\frac{\mu\big((T^{-1}B)\cap (T^{-1}A)\big)}{\mu(A)}\hspace{-0.09cm}=\hspace{-0.09cm}\frac{\mu\big(T^{-1}(B\cap A)\big)}{\mu(A)}\hspace{-0.09cm}=\hspace{-0.09cm}\frac{\mu\big(B\cap A\big)}{\mu(A)}\hspace{-0.09cm}=\hspace{-0.09cm}\mu_A(B).$$ \noindent Folglich ist $\mu_A$ $T$-invariant. \noindent Da $A$ ein $\mu$"=Atom ist, nimmt  $\mu_A$ nur die Werte $0$ und $1$ an. Das Maß $\mu_A$ ist also trivialerweise ergodisch.
\end{proof}

\subsubsection*{Approximation in $L^\infty$}
%\subsubsection{Schritt VII}
%Um den letzten Verallgemeinerungsschritt vornehmen zu können, benötigen wir das folgende Hilfsresultat:
\begin{Behauptung}
	\label{Lemma9}
	Seien $(Y,\mathfrak{B},\nu,S)$ ein maßtheoretisches dynamisches System und $g\in L^1(Y,\nu)$. Für $k\in \mathbb{N}$ sei $g_k \in L^\infty(Y,\nu)$ gegeben durch $g_k:=g\cdot \mathbbm{1}_{|g|\leq k}.$ Dann gilt
	\begin{equation*}
	\label{Gleichung4.32}
	\lim\limits_{k\to\infty}\lim\limits_{N\to\infty}\frac{1}{N}\sum_{n=1}^N |g-g_k|(S^ny)=0 \text{ für }\nu-\text{fast alle } y\in Y.
	\end{equation*}
\end{Behauptung}
\begin{proof}[\textbf{Beweis:}]
	Nach Definition der Folge $(g_k)_{k=1}^\infty $ konvergiert diese punktweise $\nu$"=fast überall gegen die Funktion $g$ und nach dem Satz von der majorisierten Konvergenz ist \begin{equation}
	\label{Gleichung2.16d}\vspace{-0.2cm}
	0=\int\limits_Y \lim\limits_{k\to\infty} |g-g_k|d\nu=\lim\limits_{k\to\infty}\int\limits_Y |g-g_k|d\nu=\lim\limits_{k\to\infty}||g-g_k||_{L^1(Y, \nu)} \end{equation} erfüllt. Nach dem punktweisen Ergodensatz existiert für beliebiges $k\in \mathbb{N}$ eine Funktion $g_k^*\in L^1(Y, \nu)$, sodass \begin{equation} \label{Gleichung2.17c}\lim_{N\to\infty}\frac{1}{N}\sum_{n=1}^N |g-g_k|(S^ny)=g_k^*(y) \text{ für }\nu\text{-fast alle } y\in Y \end{equation} sowie 
	\begin{equation}
	\label{Gleichung2.17b}
	\int\limits_Y |g-g_k|d\nu=\int\limits_Y g_k^*d\nu
	\end{equation}
	gelten. Nach Definition der Folge $\big(g_k\big)_{k=1}^\infty$ gilt für alle $k,l\in \mathbb{N}$ mit $k\geq l$ und $\nu$-fast alle $y\in Y$ $$0\leq |g-g_k|(y)\leq|g-g_l|(y)$$ und folglich wegen \eqref{Gleichung2.17c}  $0\leq g_k^*(y)\leq g_l^*(y).$ Deshalb ist die Funktion 
	$L^1(Y, \nu)\ni g^*:=\lim\limits_{k\to\infty}g_k^*$ wohldefiniert und erfüllt nach dem Satz über die monotone Konvergenz 
	\begin{equation}
	\label{Gleichung4.36}
	0\leq \int\limits_Y g^*d\nu=\int\limits_Y \lim\limits_{k\to\infty} g_k^*\, d\nu=\lim\limits_{k\to\infty}\int\limits_Yg_k^*\,d\nu\overset{\eqref{Gleichung2.17b}}{=}\lim_{k\to\infty}\int\limits_Y|g-g_k|d\nu\overset{\eqref{Gleichung2.16d}}{=}0.
	\end{equation}
	Nach Definition der Funktionen $g_k^*$, $k\in \mathbb{N}$ in \eqref{Gleichung2.17c} sind diese  $\nu$"=fast überall nicht"=negativ. Zusammen mit \eqref{Gleichung4.36}  folgt, dass $g^*(y)=0$ für $\nu$"=fast alle $y\in Y$ erfüllt ist. 
	Damit erhalten wir für $\nu$"=fast alle $y\in Y$\vspace{-0.cm} $$0=g^*(y)=\lim\limits_{k\to\infty} g_k^*(y)\overset{\eqref{Gleichung2.17c}}{=}\lim\limits_{k\to\infty}\lim\limits_{N\to\infty}\frac{1}{N}\sum\limits_{n=1}^N\big|g-g_k|(S^ny).\vspace{-1cm}$$
\end{proof}
\noindent Wir verfügen nun über die Hilfsmittel, um in den folgenden drei Kapiteln den Beweis des Rückkehrzeitentheorems von Bourgain zu führen.

	\chapter{Beweis für den Kroneckerfaktor}
	\label{Kapitel2}
\noindent In diesem Kapitel beweisen wir in drei Schritten den folgenden Satz:
\begin{Satz}
	\label{Satz2.2b}
	Seien $(X, \mathfrak{A}, \mu, T)$ ein ergodisches maßtheoretisches dynamisches System und $f \in \mathcal{K}$. 	Dann ist für $\mu$-fast alle $x \in X$ die Folge 	 $\big( f(T^n x)\big)_{n=1}^{\infty}$ 
	ein universell gutes Gewicht für die punktweise Konvergenz von $L^\infty$"=Funktionen.
\end{Satz}

\noindent Dazu beweisen wir den Satz in Schritt I zunächst für Eigenfunktionen von $T$ und in Schritt II für Linearkombinationen dieser Eigenfunktionen. Mithilfe eines Approximationsargument führen wir in Schritt III schließlich den Beweis für Funktionen aus dem Kroneckerfaktor $\mathcal{K}$.

\section{Schritt I: Eigenfunktionen}
	\label{Eigenfunktion}
	\begin{Proposition}
		\label{Lemma2.1}
		
		Seien $(X, \mathfrak{A}, \mu, T)$ ein  maßtheoretisches dynamisches System und $f \in L^1(X, \mu)$ eine Eigenfunktion von $T$.	Dann ist für $\mu$-fast alle $x \in X$ die Folge 	 $\left( f\left(T^n x\right)\right)_{n=1}^{\infty}$ 
		ein universell gutes Gewicht für die punktweise Konvergenz von $L^1$"=Funktionen.
		\end{Proposition}

	\begin{proof}[\textbf{Beweis:}]
		  Da $f$ nach Voraussetzung eine Eigenfunktion von $T$ ist, existiert ein $\lambda\in \mathbb{T}$, sodass $Tf=f\circ T=\lambda f$ erfüllt ist. Folglich existiert eine Menge		 $\mathcal{N}_{f} \in \mathfrak{A}$ mit $\mu(\mathcal{N}_{f})=0$, sodass für alle $x\in X\setminus\mathcal{N}_f$ \vspace{-0.15cm}$$Tf(x)=\lambda f(x)$$ gilt.
		 Seien nun ein beliebiges maßtheoretisches dynamisches System $(Y,\mathfrak{B}, \nu, S)$ und eine Funktion $ g \in L^1(Y,\nu) $ gegeben.
		Wir betrachten das Produktsystem \vspace{-0.15cm}$$\big(\mathbb{T}\times Y, \mathbb{B}_{\mathbb{T}}\times \mathfrak{B}, \lambda_{\mathbb{T}}\otimes \nu, \phi_{\lambda}\big),$$ wobei $\mathbb{B}_{{\mathbb{T}}}$ die Borelsche $\sigma$-Algebra auf dem Einheitskreis ${\mathbb{T}}$, $\lambda_{{\mathbb{T}}}$ das Lebesgue-Maß auf ${\mathbb{T}}$ (Bildmaß des Lebesgue-Maßes auf $[0,1]\subset \mathbb{R}$ unter $t \mapsto e^{2\pi it}$) sowie  $\phi_{\lambda }$ die messbare Abbildung $\phi_{\lambda }: {\mathbb{T}}\times Y \to {\mathbb{T}}\times Y$ mit \vspace{-0.15cm}$$\phi_{\lambda }(z,y):= (\lambda z, Sy)$$ bezeichne. 
		Weiterhin sei $\tilde{g} \in L^1({\mathbb{T}}\times Y, \lambda_{{\mathbb{T}}}\otimes \nu )$ definiert als \vspace{-0.15cm}$$\tilde{g}(z,y):= z \cdot g(y).$$
		\noindent Aufgrund des punktweisen Ergodensatzes existiert für $(\lambda_{{\mathbb{T}}}\otimes\nu)$-fast alle $(z,y)\in {\mathbb{T}}\times Y$ der Grenzwert 	
		\begin{equation*}
		\label{Gleichung2.1}
		\lim_{N\to \infty}{\frac{1}{N} \sum_{n=1}^{N}\tilde{g}\big({\phi_{\lambda }}^n(z,y)\big)}= \lim_{N\to\infty}\frac{1}{N}\sum_{n=1}^{N}\lambda ^nzg(S^ny)=z\cdot \lim_{N\to\infty}\frac{1}{N}\sum_{n=1}^{N} \lambda ^n g(S^n y).	\end{equation*}
		Da die Konvergenz nicht  von $z\in{\mathbb{T}}$ abhängt, konvergiert die Folge $\big(\frac{1}{N}\sum_{n=1}^{N} \lambda^ng(S^n y)\big)_{N=1}^\infty$ auf einer Produktmenge $\mathbb{T}\times Y'$ mit $Y'\in \mathfrak{B}$. Da die Menge $\mathbb{T}\times Y'$ Maß $1$ hat, folgt nach dem Satz von Fubini:
		\begin{equation} \label{Gleichung1}  \lim_{N\to\infty}\frac{1}{N}\sum_{n=1}^{N} \lambda^ng(S^n y) \text{ existiert für } \nu\text{-fast alle } y\in Y.\end{equation} 		
		 Es folgt für alle $x\in X\setminus \mathcal{N}_f$:
		\begin{equation*}
		\lim_{N\to\infty}\frac{1}{N}\sum_{n=1}^{N}f(T^nx)g(S^ny) 
		=\lim_{N\to\infty}f(x)\frac{1}{N}\sum_{n=1}^{N}{\lambda}^n g(S^ny) \text{ existiert für } \nu"=\text{fast alle }y\in Y.
		\end{equation*}
	Da die Menge $\mathcal{N}_f$ nicht vom System $(Y,\mathfrak{B},\nu,S)$ und der Funktion $g\in L^1(Y,\nu)$ abhängt, ist die Proposition damit bewiesen.
	\end{proof}

\section{Schritt II: Linearkombinationen von Eigenfunktionen}
\begin{Proposition}
	\label{Behauptung3.3}
	Seien $(X, \mathfrak{A}, \mu, T)$ ein maßtheoretisches dynamisches System und $f\in L^1(X,\mu)$ eine Linearkombination von Eigenfunktionen von $T$. Dann ist für $\mu$-fast alle $x \in X$ die Folge 	 $\big( f(T^n x)\big)_{n=1}^{\infty}$ 
ein universell gutes Gewicht für die punktweise Konvergenz von $L^\infty$"=Funktionen.	
\end{Proposition}
\noindent \begin{proof}[\textbf{Beweis:}]
Da $f$ nach Voraussetzung eine Linearkombination von Eigenfunktionen von $T$ ist, existieren ein $\kappa \in \mathbb{N}$, Koeffizienten $\mathbb{C}\ni c_k,\ k=1, \dots, \kappa$ und Eigenfunktionen von $T$ $L^1(X,\mu)\ni h_k,\ k=1, \dots, \kappa$, sodass sich $f$ darstellen lässt als 
\begin{equation}
\label{Gleichung3.5}
f=\sum_{k=1}^\kappa c_k h_k.
\end{equation} 
Nach Proposition \ref{Lemma2.1} ist für $k\in \{1, \dots, \kappa \}$ die Folge $\big(h_k(T^nx)\big)_{n=1}^\infty$ für $\mu$"=fast alle $x\in X$ ein universell gutes Gewicht für die punktweise Konvergenz von $L^1$"=Funktionen. Da eine endliche Vereinigung von $\mu$"=Nullmengen wieder eine $\mu$"=Nullmenge ist, existiert folglich eine Menge $X'\in \mathfrak{A}$ mit $\mu(X')=1$, sodass die Folge $\big(h_k(T^nx)\big)_{n=1}^\infty$ für alle $x\in X'$ und alle $k\in \{1, \dots, \kappa \}$ ein universell gutes Gewicht für die punktweise Konvergenz von $L^1$"=Funktionen ist.\\
Seien nun ein weiteres maßtheoretisches dynamisches System $(Y,\mathfrak{B},\nu,S)$ und eine Funktion $g\in L^1(Y,\nu)$ beliebig gewählt sowie ein $x\in X'$ fixiert. Nach Definition der Menge $X'$ gilt für ein beliebiges $k\in \{1, \dots, \kappa \}$:

\begin{equation}
\label{Gleichung2.5}
\lim\limits_{N\to\infty}\frac{1}{N}\sum_{n=1}^{N}h_k(T^nx)g(S^ny) \text{ existiert für }\nu\text{-fast alle }y\in Y. 
\end{equation}
Für jedes $k \in \{1, \dots, \kappa\}$ existiert eine $\nu$"=Nullmenge, sodass für alle $y$ außerhalb dieser Nullmenge der Grenzwert in \eqref{Gleichung2.5}  existiert. Da die endliche Vereinigung von $\nu$"=Nullmengen eine $\nu$"=Nullmenge ist, existiert also eine Menge 
$Y'\in \mathfrak{B}$ mit $\nu(Y')=1$, sodass der Grenzwert $\lim_{N\to\infty}\frac{1}{N}\sum_{n=1}^Nh_k(T^nx)g(S^ny)$ für alle $y\in Y'$ und alle $k\in \{1, \dots, \kappa \}$ existiert. 

\noindent Damit existiert für alle  $y\in Y'$ der Grenzwert 
 \begin{equation*}
 \lim_{N\to\infty}\frac{1}{N}\sum_{n=1}^{N}{f}(T^nx)g(S^ny).
 \end{equation*}
 Da wir $x\in X'$ beliebig gewählt haben, ist die Proposition damit bewiesen.\end{proof}
 
\section{Schritt III: Funktionen aus dem Kroneckerfaktor}
\label{SchrittIIIb}
\todo{Den neuen Teil nochmal auf Richtigkeit überprüfen.}
Um den Satz \ref{Satz2.2b}  beweisen zu können, benötigen wir das folgende Hilfsresultat:
\begin{Lemma}
	\label{Behauptung42}
Seien $(X,\mathfrak{A},\mu,T)$ ein maßtheoretisches dynamisches System und $\big(f_k\big)_{k=1}^\infty \subset L^2(X,\mu)$ eine Folge von Funktionen, für welche $\lim_{k\to\infty}||f-f_k||_{L^2(X,\mu)}=0$ für ein ${f\in L^2(X,\mu)}$ gilt. Weiterhin existiere ein $x\in X$, sodass \vspace{-0.15cm} $$\lim_{N\to\infty} \frac{1}{N}\sum_{n=1}^N|f-f_k|(T^nx)=\int\limits_X|f-f_k|d\mu \text{ für alle } k\in \mathbb{N}$$  erfüllt ist. Sei $\epsilon>0$. Dann existieren Zahlen $k_0, N_{k_0}\in \mathbb{N}$, sodass für alle $N\geq N_{k_0}$\vspace{-0.15cm}
$$0\leq \frac{1}{N}\sum_{n=1}^N|f-f_{k_0}|(T^nx)<\epsilon$$ gilt.
\end{Lemma}
\begin{proof}[\textbf{Beweis:}]
Da $(X,\mathfrak{A},\mu)$ ein Wahrscheinlichkeitsraum ist, gilt nach der Cauchy-Schwarzschen Ungleichung $$\lim_{k\to\infty}\int\limits_X |f-f_k|d\mu\leq \lim_{k\to\infty}||f-f_k||_{L^2(X,\mu)}=0.$$ Also existieren ein $k_0\in \mathbb{N}$ mit $$\int\limits_X |f-f_{k_0}|d\mu<\frac{\epsilon}{2}$$ sowie nach Voraussetzung ein $N_{k_0}\in \mathbb{N}$, sodass für alle $N\geq N_{k_0}$ $$\bigg|\frac{1}{N}\sum_{n=1}^N |f-f_{k_0}|(T^nx)-\int\limits_X|f-f_{k_0}|d\mu\,\bigg|<\frac{\epsilon}{2}$$ erfüllt ist. Folglich gilt für alle $N\geq N_{k_0}$: $$0\leq \frac{1}{N}\sum_{n=1}^N |f-f_{k_0}|(T^nx)\leq \bigg|\frac{1}{N}\sum_{n=1}^N |f-f_{k_0}|(T^nx)-\int\limits_X |f-f_{k_0}|d\mu\,\bigg|+\int\limits_X|f-f_{k_0}|d\mu<\epsilon.\vspace{-0.7cm}$$
\end{proof}

\noindent Wir zeigen nun:\\
\noindent \textbf{Satz \ref{Satz2.2b}.}\ \textit{	Seien $(X, \mathfrak{A}, \mu, T)$ ein ergodisches maßtheoretisches dynamisches System und $f \in \mathcal{K}$. 	Dann ist für $\mu$-fast alle $x \in X$ die Folge 	 $\big( f(T^n x)\big)_{n=1}^{\infty}$ 
	ein universell gutes Gewicht für die punktweise Konvergenz von $L^\infty$"=Funktionen.}

\noindent \begin{proof}[\textbf{Beweis:}]
Es bezeichne $$\mathcal{E}:= \big\{h \in L^{2}(X, \mu)\big|\ Th = \lambda h  \text{ für ein}\ \lambda \in \mathbb{T} %\ mit\ |\lambda|=1
\big\}\overset{\text{Bem.}}{\underset{\ref{Bemerkung3}}{=}}\big\{h \in L^{\infty}(X, \mu)\big|\ Th = \lambda h  \text{ für ein}\ \lambda \in \mathbb{T} %\ mit\ |\lambda|=1
\big\}$$ die Menge der Eigenfunktionen von $T$.
Nach Definition des Kroneckerfaktors (Definition \ref{DefKro}) gilt
$$\mathcal{K} = \overline{\text{span}\big(\{h|h \in \mathcal{E}\}\big)}^{L^2(X,\mu)}. $$  Da  nach Voraussetzung ${f} \in \mathcal{K}$ erfüllt ist, existiert nach Bemerkung \ref{Bemerkung3} eine Folge $\big({f_k}\big)_{k=1}^\infty\subset L^\infty(X,\mu)$ von Linearkombinationen von Eigenfunktionen von $T$
mit 	
\begin{equation}
\label{Gleichung_Konvergenz}
\lim_{k\to\infty} ||f-f_k||_{L^2(X,\mu)}=0.
\end{equation}

\noindent Nach Proposition $\ref{Behauptung3.3}$ existiert für jede Funktion $f_k, \ k\in \mathbb{N}$ eine Menge $\mathcal{N}_{k} \in \mathfrak{A}$ mit $\mu(\mathcal{N}_{k})=0,$ sodass die Folge $\big(f_k(T^nx)\big)_{n=1}^\infty$ für alle $x\in X\setminus \mathcal{N}_{k}$  ein universell gutes Gewicht für die punktweise Konvergenz von $L^1$"=Funktionen und damit auch für die punktweise Konvergenz von $L^\infty$"=Funktionen ist. Da das System $(X,\mathfrak{A},\mu,T)$ ergodisch ist, existieren für $k\in \mathbb{N}$ nach dem punktweisen Ergodensatz weiterhin Mengen $ \hat{\mathcal{N}}_{k}\in \mathfrak{A}$ mit $\mu(\hat{\mathcal{N}}_{k})=0,$ sodass für alle $x\in X\setminus \hat{\mathcal{N}}_{k}$\vspace{-0.08cm} 
\begin{equation}
\label{Gleichung3.9}
\lim_{N\to\infty}\frac{1}{N}\sum_{n=1}^N \big|f-f_k\big|(T^nx)=\int\limits_X\big|f-f_k\big|d\mu
\end{equation}
erfüllt ist. 
Definiere nun $X'\in \mathfrak{A}$ durch \vspace{-0.08cm}
\begin{equation} \label{Gleichung2.2b}X':=\Big(\bigcap_{k=1}^\infty(X\setminus \mathcal{N}_{k})\Big)\cap \Big(\bigcap_{k=1}^\infty (X\setminus\hat{\mathcal{N}}_{k})\Big).\end{equation}
Dann gilt $\mu(X')=1.$\\
\noindent Seien nun ein beliebiges maßtheoretisches dynamisches System $(Y,\mathfrak{B}, \nu, S)$, eine Funktion $ g \in L^\infty(Y,\nu) $ sowie ein $x\in X'$ gegeben.\\
Nach Definition von $X'$ ist die Folge $\big(f_k(T^nx)\big)_{n=1}^\infty$ für alle $k\in \mathbb{N}$ ein universell gutes Gewicht für die punktweise Konvergenz von $L^\infty$"=Funktionen. Deshalb existiert für alle $k\in \mathbb{N}$ der Grenzwert $$\lim_{N\to\infty} \frac{1}{N}\sum_{n=1}^N f_k(T^nx)g(S^ny)$$ für $\nu$"=fast alle $y\in Y$. Da die abzählbare Vereinigung von $\nu$"=Nullmengen eine $\nu$"=Nullmenge ist, existiert folglich eine Menge $Y'\in \mathfrak{B}$ mit $\nu(Y')=1$, sodass für alle $y\in Y'$ gilt:
\begin{equation} \label{Gleichung3.11} \lim_{N\to\infty}\frac{1}{N}\sum_{n=1}^N f_k(T^nx)g(S^ny) \text{ existiert für alle }k\in \mathbb{N}\end{equation} und 
$$|g(S^ny)|\leq ||g||_{L^\infty(Y,\nu)} \text{ für alle }n\in \mathbb{N}.$$
Seien nun ein $y\in Y'$ sowie ein $\epsilon>0$ fixiert. Für $N,M,k\in \mathbb{N}$ gilt
\begin{equation}
\label{Gleichung3}
\begin{split}
0 &\leq \left|\frac{1}{N}\sum_{n=1}^{N}{f}(T^nx)g(S^ny)-\frac{1}{M}\sum_{n=1}^{M}{f}(T^n x)g(S^n y)\right| \\
&\leq \left| \frac{1}{N}\sum_{n=1}^{N}\left({f}-{f_k}\right)(T^nx)g(S^ny)\right|
+ \left|\frac{1}{M}\sum_{n=1}^{M}\left({f}-{f_k}\right)(T^nx)g(S^ny)\right|\\
&+ \left|\frac{1}{N}\sum_{n=1}^{N}{f_k}(T^nx)g(S^ny) - \frac{1}{M}\sum_{n=1}^{M}{f_k}(T^nx)g(S^ny)\right|\\
&\leq\underbrace{\frac{1}{N}\sum_{n=1}^{N}\left|{f}-{f_k}\right|(T^nx)||g||_{L^\infty(Y, \nu)}}_{(I)}+\underbrace{\frac{1}{M}\sum_{n=1}^{M}\left|{f}-{f_k}\right|(T^nx)||g||_{L^\infty(Y, \nu)}}_{(II)}\\
&+\underbrace{\left|\frac{1}{N}\sum_{n=1}^N{f_k}(T^nx)g(S^ny) - \frac{1}{M}\sum_{n=1}^{M}{f_k}(T^nx)g(S^ny)\right|}_{(III)}.
\end{split}
\end{equation}
\noindent \textbf{Zu I und II:} Nach Wahl der Folge $\big(f_k\big)_{k=1}^\infty$ sowie von $x\in X'$ sind die Voraussetzungen von Lemma \ref{Behauptung42} erfüllt. Folglich existieren Zahlen $k_0, N_{k_0}\in \mathbb{N}$, sodass für alle $N,M\geq N_{k_0}$ 
\begin{equation}
\label{Gleichung3.17b}
\frac{1}{N}\sum_{n=1}^N |f-f_{k_0}|(T^nx)||g||_{L^\infty(Y,\nu)}<\frac{\epsilon}{3}
\end{equation}
sowie 
\begin{equation}
\label{Gleichung3.18b}
\frac{1}{M}\sum_{n=1}^M |f-f_{k_0}|(T^nx)||g||_{L^\infty(Y,\nu)}<\frac{\epsilon}{3}
\end{equation}
erfüllt sind.\\

\noindent \textbf{Zu III:}
Nach Wahl von $y\in Y'$ und \eqref{Gleichung3.11} sowie nach Wahl von $x\in X'$ in \eqref{Gleichung2.2b}  ist die  Folge  $\big(\frac{1}{N}\sum_{n=1}^{N}{f}_{k_0}(T^nx)g(S^ny)\big)_{N=1}^\infty$ konvergent und bildet deshalb eine Cauchy-Folge. Folglich existiert ein $N_0$ sodass für alle $N,M \geq N_0$ gilt:

\begin{equation}
\label{Gleichung2.11}
\left|\frac{1}{N}\sum_{n=1}^{N}{f}_{k_0}(T^nx)g(S^ny)\ - \ \frac{1}{M}\sum_{n=1}^{M}{f}_{k_0}(T^nx)g(S^ny)\right|\
<\frac{\epsilon}{3}.
\end{equation}

\noindent Zusammengenommen gilt also in \eqref{Gleichung3}  für alle $N,M \geq \max\{N_{k_0}, N_0\}$ und $k=k_0$:

\begin{equation*}
\label{Gleichung2.12}
0 \leq \left|\frac{1}{N}\sum_{n=1}^{N}{f}(T^nx)g(S^ny)-\frac{1}{M}\sum_{n=1}^{M}{f}(T^n x)g(S^n y)\right|< \frac{\epsilon}{3}+\frac{\epsilon}{3}+\frac{\epsilon}{3}=\epsilon.
\end{equation*}
\noindent Da wir $\epsilon>0$ beliebig gewählt haben, ist $$\Big(\frac{1}{N}\sum_{n=1}^{N}{f}(T^nx)g(S^ny)\Big)_{N=1}^\infty$$ eine Cauchyfolge und der Grenzwert
\begin{equation*}
\lim_{N\to \infty}\frac{1}{N}\sum_{n=1}^{N}f(T^nx)g(S^ny)
\end{equation*} existiert. Damit ist der Satz \ref{Satz2.2b} bewiesen. \end{proof}

\noindent 
\begin{Bemerkung} Analog zu Abschnitt \ref{SchrittVII} könnten wir unter der zusätzlichen Annahme $f\in \mathcal{K}\cap L^\infty(X,\mu)$ bereits an dieser Stelle den Beweis dafür führen, dass die in Satz \ref{Satz2.2b} definierten Gewichte universell gute Gewichte für die punktweise Konvergenz von $L^1$"=Funktionen sind. \end{Bemerkung}

\chapter{Beweis für das orthogonale Komplement des Kroneckerfaktors}	
\label{Kapitel3}
\noindent In diesem Kapitel führen wir den Beweis des Rückkehrzeitentheorems von Bourgain für $f\in \mathcal{K}^\bot\cap L^\infty(X,\mu)$ für den Fall, dass $(X,\mathfrak{A},\mu,T)$  und $(Y,\mathfrak{B},\nu,S)$ invertierbare, ergodische maßtheoretische dynamische Systeme sind sowie $f\in L^\infty(X,\mu)$ und $g\in L^\infty(Y,\nu)$ gewählt sind. Wie in \cite[S.43]{Bourgain1989} treffen wir zunächst die Annahme, dass die Funktion $f\in \mathcal{K}^\bot \cap L^\infty(X,\mu)$ einfach sei, das heißt ein endliches Bild habe. Eine Verallgemeinerung dieses Ergebnisses auf beliebige Funktionen $f\in \mathcal{K}^\bot \cap L^\infty(X,\mu)$ nehmen wir im Abschnitt \ref{unendlich} vor. \vspace{-0.15cm}

\section{Beweis für einfache Funktionen $f\in  \mathcal{K}^\bot\cap L^\infty(X,\mu)$}
\label{endlich}
\begin{Satz}
	\label{Satz2.5}
	Seien $(X, \mathfrak{A}, \mu, T)$  und $(Y, \mathfrak{B}, \nu, S)$ invertierbare, ergodische maßtheoretische dynamische Systeme.  Des Weiteren seien $f\in \mathcal{K}^\bot\cap L^\infty(X, \mu)$ eine einfache Funktion und $g \in L^\infty(Y, \nu)$. Dann existiert eine Menge $X'\in \mathfrak{A}$ mit $\mu(X')=1$, welche unabhängig vom System $(Y,\mathfrak{B},\nu, S)$ und $g\in L^\infty(Y,\nu)$ gewählt werden kann, sodass für alle $x\in X'$\vspace{-0.15cm}
	$$\lim_{N\to\infty}\frac{1}{N}\sum_{n=1}^N f(T^nx)g(S^ny)= 0 \text{ für }\nu-\text{fast alle } y\in Y$$ erfüllt ist.
\end{Satz}

\noindent Der hier vorgestellte Beweis beruht auf der Darstellung in \cite[S.73-92]{Assani2003}. Zunächst konstruieren wir die Menge $X'$ und beweisen, dass diese Maß $1$ hat. Anschließend zeigen wir in drei Schritten, dass die Gewichte $\big(f(T^nx)\big)_{n=1}^\infty,\ x\in X'$ die in Satz \ref{Satz2.5} formulierte Eigenschaft erfüllen.

\subsection{Vorbetrachtungen}
\label{Vorbetrachtungen2}
Um die Menge $X'$ aus Satz \ref{Satz2.5} konstruieren zu können, benötigen wir die folgenden beiden Lemmata:
\begin{Lemma}
\label{Behauptung3.1}
Seien $(X,\mathfrak{A},\mu,T)$ ein invertierbares, ergodisches maßtheoretisches dynamisches System. Darüber hinaus seien eine Funktion $f\in \mathcal{K}^\bot\cap L^\infty(X,\mu)$ sowie die Menge $X_1$ definiert als 
\begin{equation}
\label{X1}
X_1:=\bigg\{x\in X\Big| \lim_{N\to \infty}\frac{1}{N} \sum_{n=1}^N f(T^nx)\overline{f(T^n\xi)}=0 \ \text{für}\ \mu \text{-fast alle } \xi\in X \bigg\}
\end{equation}
gegeben. Dann gilt $X_1\in \mathfrak{A}$ sowie
$$\mu(X_1)=1.$$
\end{Lemma} 

\begin{proof}[\textbf{Beweis:}]
Nach Voraussetzung gilt $f\in \mathcal{K}^\bot\cap L^\infty(X,\mu)$. Wegen Satz \ref{SpekChar} ist das zu $f$ gehörige Spektralmaß $\mu_f$ stetig. Zusammen mit Satz \ref{Satz2.2} erhalten wir deshalb\vspace{-0.15cm} 
$$\lim\limits_{N\to \infty}\frac{1}{N}\sum\limits_{n=1}^N f(T^nx)\overline{f(T^n\xi)}=0\  \mu \otimes\mu \text{-fast überall in } X\times X.$$
Die Anwendung des Satzes von Fubini liefert\vspace{-0.15cm}
$$0=\hspace{-0.16cm}\int\limits_{X\times X}\hspace{-0.13cm} \Big|\hspace{-0.099cm}\lim\limits_{N\to\infty} \frac{1}{N}\hspace{-0.099cm}\sum_{n=1}^N\hspace{-0.099cm} f(T^nx)\overline{f(T^n\xi)}\Big|d\big(\mu\otimes \mu\big)(x,\xi)=\hspace{-0.099cm}\int\limits_X\hspace{-0.099cm}\int\limits_X\hspace{-0.099cm} \Big|\hspace{-0.099cm}\lim\limits_{N\to\infty}\frac{1}{N}\hspace{-0.099cm}\sum_{n=1}^N\hspace{-0.099cm}f(T^nx)\overline{f(T^n\xi)}\Big|d\mu(\xi)d\mu(x).$$
Deshalb gilt für $\mu$"=fast alle $x\in X$:\vspace{-0.15cm}
$$\int\limits_X \Big|\lim\limits_{N\to\infty}\frac{1}{N}\sum_{n=1}^Nf(T^nx)\overline{f(T^n\xi)}\Big|d\mu(\xi)=0$$ und folglich\vspace{-0.15cm}
$$\lim\limits_{N\to\infty}\frac{1}{N}\sum_{n=1}^Nf(T^nx)\overline{f(T^n\xi)}=0 \text{ für }\mu"=\text{fast alle }\xi\in X.$$ \noindent Aufgrund der Vollständigkeit von $(X,\mathfrak{A},\mu)$ ist deshalb $X_1\in \mathfrak{A}$ erfüllt.
Das Lemma ist damit bewiesen.	
\end{proof}

\begin{Lemma}
	\label{X2}
	Seien $(X,\mathfrak{A},\mu,T)$ ein ergodisches maßtheoretisches dynamisches System und $f\in L^\infty(X,\mu)$ eine einfache Funktion  mit dem Bild $f(X)$. Für $n\in \mathbb{N}$ sei die Abbildung $\Phi_n: X\to f(X)^n$ definiert durch \vspace{-0.15cm}$$\Phi_n: x\mapsto \big(f(x), f(Tx), \dots, f(T^{n-1}x)\big).$$ Weiterhin sei für $n\in \mathbb{N}$ und $A\subseteq f(X)^n$ die Menge $X_A^n\in \mathfrak{A}$ gegeben durch \vspace{-0.15cm}$$X_A^n:=\bigg\{x\in X\Big|\lim\limits_{N\to\infty}\frac{1}{N}\sum_{k=1}^N\big(\mathbbm{1}_A \circ\Phi_n\big)(T^kx)=\mu\big(\big\{\xi\in X|\ \big(f(T\xi), \dots, f(T^{n}\xi)\big)\in A \big\}\big)\bigg\}$$ sowie die Menge $X_2\in \mathfrak{A}$ definiert als \vspace{-0.3cm}
	\begin{equation}
	\label{MengeX2}
	X_2:=\bigcap\limits_{n\in \mathbb{N}}\bigcap\limits_{A\subseteq f(X)^n} X_A^n.
	\end{equation}
	Dann gilt $$\mu(X_2)=1.$$	
\end{Lemma}

\begin{proof}[\textbf{Beweis:}]
Die Abbildung $\lim_{N\to\infty} \frac{1}{N}\sum_{k=1}^N \big(\mathbbm{1}_A \circ\Phi_n\big)(T^kx)$ ist als punktweiser Grenzwert von Kompositionen messbarer Abbildungen messbar, deshalb gilt $X_A^n\in \mathfrak{A}$.   
Da $(X, \mathfrak{A}, \mu, T)$ ein ergodisches System sowie $f$ ein Element aus $L^\infty(X,\mu)$ ist, gilt nach dem punktweisen Ergodensatz für jedes $n\in \mathbb{N}$ und $A \subseteq f(X)^n$ sowie für $\mu$- fast alle $x\in X$:
\begin{equation*}
\begin{split}
\lim_{N\to\infty}\frac{1}{N}\sum_{k=1}^N \big(\mathbbm{1}_A \circ \Phi_n\big)(T^kx)
&=\lim_{N\to\infty}\frac{1}{N}\sum_{k=2}^{N+1} \big(\mathbbm{1}_A \circ \Phi_n\big)(T^kx)\\
&=\lim_{N\to\infty}\frac{1}{N}\sum_{k=1}^N \mathbbm{1}_A \circ \Phi_n\circ T(T^kx)\allowdisplaybreaks[1]\\
&= \int\limits_X \mathbbm{1}_A\circ \Phi_n\circ T(x)d\mu(x)\\
&=\mu\left(\left\{\xi\in X|\ \big(f(T\xi), f(T^2\xi), \dots f(T^{n}\xi)\big)\in A \right\}\right).
\end{split}\end{equation*}

\noindent Folglich ist 
\begin{equation}
\label{Gleichung2.16}
\mu(X_A^n)=1
\end{equation}
\noindent erfüllt. Definiere nun für $n\in \mathbb{N}$ die Menge $X_2^n\in \mathfrak{A}$ durch
$$X_2^n:=\bigcap\limits_{A\subseteq f(X)^n} X_A^n.$$
Weil $f$ nach Voraussetzung einfach ist, hat die Menge $X_2^n$ als endlicher Schnitt von Mengen mit Maß 1 selbst volles Maß und wir erhalten $$\mu(X_2)=1.\vspace{-0.99cm}$$
\end{proof}

\begin{Bemerkung}
	Wir haben die Mengen $X_1$ und $X_2$ im Hinblick auf Schritt II (Abschnitt \ref{SchrittIV}) konstruiert. 
\end{Bemerkung}

\noindent Im Folgenden beweisen wir in drei Schritten Satz \ref{Satz2.5} und zeigen, dass die Menge $X'$ bis auf eine Nullmenge als $X_1\cap X_2$ gewählt werden kann.

\begin{proof}[\textbf{Beweis von Satz \ref{Satz2.5}:}]
\noindent Seien $X_1, X_2\in \mathfrak{A}$ die Mengen, welche wir durch Anwendung der Lemmata \ref{Behauptung3.1} und \ref{X2} auf das System $(X,\mathfrak{A},\mu,T)$ und die Funktion $f\in \mathcal{K}^\bot \cap L^\infty(X,\mu)$ erhalten. Definiere
\begin{equation*}
\label{Gleichung3.13}
\mathfrak{A}\ni X':=X_1\cap X_2.
\end{equation*}
Nach Definition der Mengen $X_1$ und $X_2$ sind diese unabhängig vom System $(Y,\mathfrak{B},\nu,S)$ und der Funktion $g\in L^\infty(Y,\nu)$ und es gilt $\mu(X')=1.$ Da $T$ maßerhaltend ist und der abzählbare Schnitt von Mengen mit Maß $1$ wieder volles Maß hat, können wir zudem ohne Beschränkung der Allgemeinheit annehmen, dass 
\begin{equation}
\label{Gleichung1000}
|f(T^nx)| \leq ||f||_{L^\infty(X,\mu)} \text{ für alle }x\in X' \text{ und  }n\in \mathbb{N}
\end{equation} gilt. Mithilfe eines indirekten Beweises zeigen wir nun, dass für alle $x\in X'$ 

$$\lim_{N\to\infty}\frac{1}{N}\sum_{n=1}^N f(T^nx)g(S^ny)= 0 \text{ für }\nu-\text{fast alle } y\in Y$$ 
erfüllt ist.\\
\noindent  Sei also angenommen, dass dies nicht der Fall ist. 
Dann existiert ein $x\in X'$, sodass für die Menge 
\begin{equation}
\label{Gleichung3.14}
\mathfrak{B}\ni B^*:= \Big\{y\in Y\Big|\limsup\limits_{N\to\infty}\Big|\frac{1}{N}\sum_{n=1}^N f(T^nx)g(S^ny) \Big|>0 \Big.\Big\}
\end{equation}
gilt:
$$\nu(B^*)>0.$$
Wir fixieren für den gesamten Beweis von Satz \ref{Satz2.5} dieses $x\in X'$.

\subsection{Schritt I: Konstruktion von "`schlechten"' Intervallen}
\noindent   Es bezeichne Re: $\mathbb{C}\to \mathbb{R}$ diejenige Funktion, welche eine komplexe Zahl auf ihren Realteil abbildet und  Im: $\mathbb{C}\to \mathbb{R}$ diejenige Funktion, welche eine komplexe Zahl auf ihren Imaginärteil abbildet. Wir zeigen nun die folgende Behauptung:
\begin{Behauptung}
\label{Behauptung2.6} 
Es existieren Konstanten $0<a,c\in \mathbb{R}$, sodass für ein beliebiges $J\in \mathbb{N}$ sowie beliebige $\mathfrak{k}_1, \mathfrak{k}_2, \dots, \mathfrak{k}_J\in \mathbb{N}$ eine Menge $B\in \mathfrak{B}$ mit $B\subseteq B^*$ und $\nu(B)>c$, eine messbare Abbildung $\mathcal{RI}\in\{ \operatorname{Re}, \operatorname{Im}, -\operatorname{Re}, -\operatorname{Im}\}$  und eine Folge von Intervallen $\big((L_j, M_j)\big)_{j=1}^J$ aus $\mathbb{N}$ mit $L_j>\mathfrak{k}_j,\ j=1, \dots, J$ existieren, sodass gilt: Für jedes $y\in B$ und $j\in \{1, \dots, J\}$ existiert ein $n_j=n_j(y)\in (L_j, M_j)$ derart, dass
\begin{equation}
\label{Gleichung2.18}
\mathcal{RI}\Big( \sum_{n=1}^{n_j} f(T^nx)g(S^ny)\Big) >an_j
\end{equation} erfüllt ist. Dabei kann  für $j\in \{1, \dots, J-1\}$ das Intervall $(L_j, M_j)$ unabhängig von den Zahlen $\mathfrak{k}_{j+1}, \dots, \mathfrak{k}_J$ konstruiert werden.
\end{Behauptung}

\begin{proof}[\textbf{Beweis:}]

\noindent Für $N\in \mathbb{N}$ sei $F_N:Y\to \mathbb{C}$ definiert als die messbare Funktion $$F_N(y) := \frac{1}{N}\sum_{n=1}^N f(T^nx)g(S^ny).$$ 

\noindent Da für $B^*\in \mathfrak{B}$ aus $(\ref{Gleichung3.14})$\vspace{-0.1cm}
$$\nu(B^*)>0$$
erfüllt ist und sich $F_N$ darstellen lässt als $F_N=\operatorname{Re}(F_N)+i\cdot \operatorname{Im}(F_N)$, können wir eine Abbildung $\mathcal{RI}\in \{\operatorname{Re}, \operatorname{Im}, -\operatorname{Re}, -\operatorname{Im}\}$ wählen, sodass gilt:
 \begin{equation}
\label{Gleichung2.20}
\nu\Big(\big\{y\in Y\bigg| \limsup\limits_{N\to\infty}\mathcal{RI}\big(F_N(y)\big)>0\big\}\Big)>0.
\end{equation}
Da weiterhin\vspace{-0.1cm}
\begin{equation*}
\Big\{y\in Y\big|\limsup\limits_{N\to\infty}\mathcal{RI}\big(F_N(y)\big) >0\Big\}=\bigcup_{k\in \mathbb{N}}\Big\{y\in Y\big|\limsup\limits_{N\to\infty}\mathcal{RI}\big(F_N(y)\big)>\frac{1}{k}\Big\}
\end{equation*}
erfüllt ist, existiert ein $k_0\in \mathbb{N}$ mit 
\begin{equation} 
\label{Gleichung3.17}
\nu\Big(\Big\{\underbrace{y\in Y\big|\limsup\limits_{N\to\infty}\mathcal{RI}\big(F_N(y)\big)>\frac{1}{k_0} }_{=:B'\in \mathfrak{B}}\Big\}\Big)>0.
\end{equation}
\noindent Definiere\vspace{-0.1cm} \begin{equation}
\label{Gleichung2.22b}a:=\frac{1}{k_0}.
\end{equation}
\noindent Seien nun $J\in \mathbb{N}$ sowie $\mathfrak{k}_1, \mathfrak{k}_2, \dots, \mathfrak{k}_J\in \mathbb{N}$ beliebig gewählt. Im Folgenden konstruieren wir die in der Behauptung \ref{Behauptung2.6} gesuchte Menge $B$. Dazu führen wir eine Folge von Mengen $B_1, \dots, B_J \in \mathfrak{B}$ mit $B_J\subseteq B_{J-1}\subseteq \dots \subseteq B_2\subseteq B_1\subseteq B'\subseteq B^*$ ein, für welche gilt: Es existiert eine Folge von Intervallen $\big((L_j, M_j)\big)_{j=1}^J$ aus $\mathbb{N}$ derart, dass für $j\in \{1, \dots, J\}$ und alle $y\in B_j$ \vspace{-0.1cm}
$$\sup_{L_j<N<M_j} \mathcal{RI}\big(F_N(y)\big)>a$$ sowie $L_j>\mathfrak{k}_j$ erfüllt ist. Die gesuchte Menge $B$ ist dann gerade durch $B_J$ gegeben. Im Folgenden definieren wir die Mengen $B_1,\dots, B_J$:\\
\noindent Da für jedes Element $y$ aus der in \eqref{Gleichung3.17} definierten Menge $B'$ und jedes $L\in \mathbb{N}$ nach Definition des Limes superior\vspace{-0.1cm}
$$\sup\limits_{N>L}\mathcal{RI}\big(F_N(y)\big)\geq \limsup\limits_{N\to\infty}\mathcal{RI}\big(F_N(y)\big)>a$$ erfüllt ist, gilt für beliebiges $L\in \mathbb{N}$ die folgende Identität:\vspace{-0.1cm}
\begin{equation}
\label{Gleichung2.22}
B'=\Big\{y\in B'\Big| \sup\limits_{N>L}\mathcal{RI}\big(F_N(y)\big)>a\Big\}=\bigcup_{M>L}\Big\{\underbrace{y\in B'\Big| \sup\limits_{L<N<M}\mathcal{RI}\big(F_N(y)\big)>a}_{=:B_M \in \mathfrak{B}}\Big\}.
\end{equation}
Da alle  $B_M$ messbare Mengen sind sowie $B_M\subseteq B_{N}$ für $N\geq M$ erfüllt ist, gilt \vspace{-0.15cm}
\begin{equation}
\label{Gleichung2.23}
\lim_{M\to\infty}\nu(B_M)=\nu(B')>0.
\end{equation}
Wähle $\mathbb{N}\ni L_1>\mathfrak{k}_1$ beliebig. Wegen $(\ref{Gleichung2.23})$ existiert ein $\mathbb{N}\ni M_1 > L_1$, sodass für die Menge $\mathfrak{B}\ni B_{M_1}\subseteq B'$  \vspace{-0.05cm}$$\nu(B_{M_1})>\left(1-\frac{1}{2J}\right) \nu(B')$$ erfüllt ist. 
Definiere \vspace{-0.05cm}$$B_1:=B_{M_1}$$ sowie $$R_1:=(L_1, M_1).$$
\noindent Die weiteren Intervalle definieren wir induktiv. 
Seien für $k\in \mathbb{N}$ mit $2\leq k\leq J$ bereits Intervalle $(L_j, M_j), \ j=1, \dots, k-1$ mit $L_j>\mathfrak{k}_j$ sowie  Mengen $B_{1}\supseteq \dots \supseteq B_{k-1} %\text{ aus } \mathfrak{B}
$ mit der Eigenschaft gefunden, dass für alle $y\in B_{k-1}$ und alle $j\in \{1, \dots, k-1\}$ gilt: Es existiert ein $n_j=n_j(y)\in (L_j, M_j)$ derart, dass $\mathcal{RI}\big(\sum_{n=1}^{n_j}f(T^nx)g(S^ny)\big)>a{n_j}$ und $\nu(B_{k-1})>\big(1-\frac{k-1}{2J}\big)\nu(B')$ erfüllt sind.\\
\noindent Wähle ein beliebiges $\mathbb{N}\ni L_{k}>\max\{\mathfrak{k}_k, M_{k-1}\}$. Analog zu $(\ref{Gleichung2.22})$ gilt\vspace{-0.05cm}
$$B_{k-1}=\bigcup_{M>L_k}\bigg\{\underbrace{y\in B_{k-1}\bigg| \sup\limits_{L_K<N<M}\mathcal{RI}\big(F_N(y)\big)>a}_{=:B_{M}^{(k-1)}\in \mathfrak{B}}\bigg\}.$$
\noindent Folglich existieren ein $M_k \in \mathbb{N}$, sodass für die Menge $\mathfrak{B}\ni B_{M_k}^{(k-1)}\subseteq B_{{k-1}}$\vspace{-0.05cm}
$$\nu\big(B_{M_k}^{({k-1})}\big)>\nu(B_{{k-1}})-\frac{1}{2J}\nu(B')>\left(1-\frac{k}{2J}\right)\nu(B')$$ erfüllt ist.
Definiere \vspace{-0.05cm}$$B_k:=B_{M_k}^{(k-1)}$$ sowie \vspace{-0.05cm}$$R_k:=(L_k, M_k).$$
Die Definition von $R_k$ ist insbesondere unabhängig von den Zahlen $\mathfrak{k}_{k+1}, \dots, \mathfrak{k}_J$.\\
\noindent Seien nun die Intervalle $(L_j, M_j)\ \ j=1, \dots, J$ sowie die Mengen $B_1, \dots, B_J \in \mathfrak{B}$ mit $B_J\subseteq B_{J-1}\subseteq \dots \subseteq B_2\subseteq B_1\subseteq B'\subseteq B^*$  gefunden. Definiere\vspace{-0.05cm} 
$$\mathfrak{B}\ni B:=B_J.$$ Dann gilt\vspace{-0.05cm} $$\nu(B)>\left(1-\frac{J}{2J}\right)\nu(B')=\frac{1}{2}\nu(B')>0.$$ Definiere \vspace{-0.05cm}$$c:=\frac{1}{2}\nu(B')>0.$$ Weiterhin ist nach Konstruktion für alle $j\in \{1, \dots, J\}$ und alle $y\in B$ $$\sup\limits_{N\in (L_j ,M_j)}\mathcal{RI}\big(F_N(y)\big)>a$$  erfüllt. Also existiert für jedes $j\in \{1, \dots, J\} $ und jedes $y\in B$ ein $n_j=n_j(y)\in (L_j, M_j)$ mit 
$$\mathcal{RI}\Big(\sum_{n=1}^{n_j}f(T^nx)g(S^ny)\Big)>n_j a. $$
Damit ist die Behauptung bewiesen. 
\end{proof}
\begin{Bemerkung}
\label{Bemerkung4.1}\begin{itemize}
	\item[] 
	\item[(1)] Die Intervalle $(L_j, M_j), \ j=1, \dots, J$ sind gerade so gewählt, dass sich die Folge $\big(\frac{1}{n_j}\sum_{n=1}^{n_j} f(T^nx)g(S^ny)\big)_{j=1}^J$ in Bezug auf das Rückkehrzeitentheorem "`schlecht"' verhält und sich für beliebig großes $J$ gerade nicht $0$ annähert. 
 \item[(2)] Wir benötigen die Aussage, dass das Intervall $(L_j, M_j)$ für $j\in \{1, \dots, J-1\}$ unabhängig von den Zahlen $\mathfrak{k}_{j+1}, \dots, \mathfrak{k}_J$ konstruiert werden kann, um die Voraussetzungen in Schritt III zu erfüllen. Aus diesem Grund haben wir auch die Konstante $c>0$ als von $J$ und $\mathfrak{k}_1, \dots, \mathfrak{k}_J$ unabhängige untere Schranke für das Maß von $B$ eingeführt. \item[(3)]  Nach Wahl der Intervalle $\big((L_j, M_j)\big)_{j=1}^J$ ist insbesondere $L_1<M_1<L_2<M_2<\dots < L_J<M_J$ erfüllt. Auf diese Ungleichung werden wir in späteren Beweisen wiederholt Bezug nehmen.
\item[(4)]Die Annahme in Satz \ref{Satz2.5}, $f\in \mathcal{K}^\bot\cap L^\infty(X,\mu)$ sei eine einfache Funktion, ist zum Beweis dieser Behauptung nicht erforderlich.\end{itemize}
	\end{Bemerkung}

\subsection{Schritt II: "`Gute"' Intervalle}
\label{SchrittIV}
Das Ziel dieses Abschnittes ist, die  nachfolgend formulierte Behauptung \ref{Behauptung3.21} zu beweisen. Es sei nochmals darauf hingewiesen, dass wir zu Beginn des Beweises ein $x\in X'=X_1\cap X_2$ fixiert haben, sodass wir nachfolgend die Definitionen der Mengen $X_1$ und $X_2$ im Abschnitt \ref{Vorbetrachtungen2} in Verbindung mit den Lemmata \ref{Behauptung3.1} und \ref{X2} nutzen können.\\
\noindent Aufgrund der Wahl $x\in X_1$ gilt
\begin{equation}
\label{Gleichung2.28}
\lim\limits_{N\to\infty}\frac{1}{N}\sum_{n=1}^N f(T^nx)\overline{f(T^n\xi)}=0 \text{ für }\mu\text{"=fast alle }\xi \in X.
\end{equation}
Für beliebiges $\delta >0$ und $\delta''>0$ existiert nach Satz \ref{Jegorow} eine Menge 
$C \in \mathfrak{A}$ mit $\mu(C)>1-\delta''$, sodass  der Grenzwert in $(\ref{Gleichung2.28})$ auf $C$ gleichmäßig in $\xi$ ist. Folglich existiert ein $N_{\delta} \in \mathbb{N}$ derart, dass für alle $N>N_\delta$ und alle $\xi\in C$ 
\begin{equation}
\label{Gleichung2.29b}
\left|\frac{1}{N}\sum_{n=1}^N f(T^nx)\overline{f(T^n\xi)}\right|<\delta
\end{equation} 
erfüllt ist. 
\noindent Für $n\in \mathbb{N}$ bezeichne $B_n^*\subseteq f(X)^n$ die Menge
\begin{equation} 
	\label{Gleichung3.22}
B_n^*:=\left\{\big(f(T\xi), f(T^2\xi), \dots, f(T^{n}\xi)\big)\left| \xi \in C\right.\big)\right\}.
\end{equation}
Mithilfe dieser Bezeichnungen, welche wir auch in Schritt III nutzen werden, können wir die folgende Behauptung formulieren:
\begin{Behauptung}
\label{Behauptung3.21}
Sei $(L,M)\subset \mathbb{N}$ ein nichtleeres Intervall. Dann existiert eine Zahl $N(M)\in \mathbb{N}$, sodass für alle $N>N(M)$ 
$$\frac{1}{N}\cdot \# \left\{k\in \{1, \dots, N\}\Big| \ \big(f(T^kx),\dots, f(T^{k+n-1}x)\big)\in B_n^*\text{ für alle }n\in (L,M)\right\}>1-\delta''$$ erfüllt ist.
\end{Behauptung}
\begin{Bemerkung}
Die Intervalle $[k, k+n-1]$, für welche $\big(f(T^kx),\dots, f(T^{k+n-1}x)\big)\in B_n^*\text{ für alle }n\in (L,M)$ gilt, verhalten sich "`gut"'  in Bezug auf den Grenzwert in \eqref{Gleichung2.28}. 
\end{Bemerkung}
\begin{proof}[\textbf{Beweis:}]
Weil $x$ ein Element der Menge $X'=X_1\cap X_2$ ist, gilt nach Definition der Menge $X_2$ in Lemma \ref{X2} für  $n\in \mathbb{N}$ und $B_n^*\subseteq f(X)^n$:
\begin{equation}
\label{Gleichung3.21}
\begin{split} 
\lim\limits_{N\to\infty}\frac{1}{N}\sum_{k=1}^N\big(\mathbbm{1}_{B_n^*} \circ\Phi_n\big)(T^kx)
&=\lim\limits_{N\to\infty}\frac{1}{N}\sum_{k=1}^N\mathbbm{1}_{B_n^*} \big(f(T^kx), f(T^{k+1}x), \dots, f(T^{k+n-1}x)\big)\\
&\hspace{-0.1cm}\overset{\text{Def.}}{\underset{X_2}{=}}\mu\big(\underbrace{\big\{\xi\in X\big|\ \big(f(T\xi), f(T^2\xi), \dots, f(T^{n}\xi)\big)\in B_n^{*} \big\}}_{\supseteq C\text{ nach Definition von }B_n^* \text{ in } \eqref{Gleichung3.22}}\big)\\
&\geq \mu(C) > 1-\delta''.
\end{split}
\end{equation}

\noindent Deshalb gilt
\begin{equation*}
%\label{Gleichung2.29}
\begin{split}
&\hspace{0.45cm}\lim\limits_{N\to\infty}\frac{1}{N}\cdot\#\left\{k\in \{1, \dots, N\}\big| \ \big(f(T^kx), f(T^{k+1}x), \dots, f(T^{k+n-1}x)\big)\in B_n^*\right\}\\
&=\lim\limits_{N\to\infty}\frac{1}{N}\cdot\#\left\{k\in \{1, \dots, N\}\big| \Phi_n(T^kx)\in B_n^*\right\}\\
&=\lim\limits_{N\to\infty}\frac{1}{N}\sum_{k=1}^N \big(\mathbbm{1}_{B_n^*}\circ \Phi_n\big)(T^kx)\\
&\hspace{-0.2cm} \overset{\eqref{Gleichung3.21}}{>}1-\delta''.
\end{split}
\end{equation*}
Folglich existiert  ein $N(n)\in \mathbb{N}$, sodass für alle $N>N(n)$
\begin{equation} 
\label{Gleichung3.25}
\frac{1}{N}\cdot\#\left\{k\in \{1, \dots, N\}\big| \ \big(f(T^kx), f(T^{k+1}x), \dots, f(T^{k+n-1}x)\big)\in B_n^*\right\}>1-\delta''
\end{equation}
 erfüllt ist. Insbesondere existiert im Fall $n=M$ also ein $N(M)$, sodass für alle $N>N(M)$ 
 \begin{equation} 
 \label{Gleichung3.24}
 \frac{1}{N}\cdot \# \left\{k\in \{1, \dots, N\}\big| \ \big(f(T^kx), f(T^{k+1}x), \dots, f(T^{k+M-1}x)\big)\in B_M^*\right\}>1-\delta''
 \end{equation}erfüllt ist.\\ \noindent Weiterhin gilt nach Definition der Menge $B_n^*$:\\ \noindent Falls die Folge $\big(f(T^kx), f(T^{k+1}x), \dots, f(T^{k+n-1}x)\big)$ ein Element von $B_n^*$ ist, dann ist für $m\leq n$ die Folge $\big(f(T^kx), f(T^{k+1}x), \dots, f(T^{k+m-1}x)\big)$ ein Element von $B_m^*$. Deshalb ist für alle $N\in \mathbb{N}$
\begin{align}
\label{Gleichung2.31}
\nonumber
&\quad \hspace{0.05cm} \left\{k\in \{1, \dots, N\}\big| \ \big(f(T^kx), f(T^{k+1}x), \dots, f(T^{k+n-1}x)\big)\in B_n^*\right\}\\
\nonumber
&\subseteq \left\{k\in \{1, \dots, N\}\big| \ \big(f(T^kx), f(T^{k+1}x), \dots, f(T^{k+m-1}x)\big)\in B_m^*\right\} 
\end{align}
erfüllt.
\noindent Folglich gilt für alle  $N>N(M)$ \vspace{-0.15cm}
\begin{align}
\nonumber
&\quad \ \frac{1}{N}\cdot \# \left\{k\in \{1, \dots, N\}\big| \ \big(f(T^kx), \dots, f(T^{k+n-1}x)\big)\in B_n^* \text{ für alle }n\in (L,M)\right\}\\
\nonumber
&\geq \frac{1}{N}\cdot \# \left\{k\in \{1, \dots, N\}\big| \ \big(f(T^kx), f(T^{k+1}x), \dots, f(T^{k+M-1}x)\big)\in B_M^*\right\}\\
\nonumber
&\hspace{-0.2cm}\overset{\eqref{Gleichung3.24}}{>}1-\delta''. 
\end{align}\vspace{-1.5cm}\\
\end{proof}
\subsection{Schritt III: Der Widerspruch}
\label{SchrittV}
In diesem Abschnitt werden wir aufbauend auf den Behauptungen \ref{Behauptung2.8}, \ref{Behauptung2.6} und \ref{Behauptung3.21}  die Behauptung \ref{Behauptung43} beweisen und zu einem Widerspruch führen. Dazu verwenden wir die in den Behauptungen \ref{Behauptung2.8}, \ref{Behauptung2.6} und \ref{Behauptung3.21} bereits genutzte Notation. Bevor wir die Behauptung \ref{Behauptung43} formulierten können, müssen wir zunächst einige Konstanten passend wählen:\\ 
Wir wählen für $a$ und $c$ aus Behauptung \ref{Behauptung2.6} ein beliebiges $J\in \mathbb{N}$ mit \vspace{-0.15cm}
\begin{equation}
\label{Gleichung2.34d}
J >\frac{4||f||^2_{L^\infty(X,\mu)}||g||^2_{L^\infty(Y, \nu)}}{a^2}
\end{equation}
sowie ein beliebiges\vspace{-0.15cm}
\begin{equation} \label{Gleichung2.31b}0<\delta<\min(c,a).  \end{equation}
Wir wählen nun in Behauptung \ref{Behauptung2.8} $\delta'\in (0,1)$ so, dass \vspace{-0.15cm}
\begin{equation}
\label{Gleichung2.34c}
\frac{J(J+1)}{2}\delta'<\frac{\delta}{a}
\end{equation}
gilt. Wir erhalten aus Behauptung \ref{Behauptung2.8} $K,M_0\in \mathbb{N}$ und wählen in Schritt II 
\begin{equation}
\label{Gleichung2.32b}
\delta''<\frac{\delta'}{2(K+1)}.
\end{equation}
Damit erhalten wir in Schritt II ein $N_\delta\in \mathbb{N}$ (siehe \eqref{Gleichung2.29b}) und wählen $L_1$ in Schritt I so, dass 
\begin{equation}
\label{Gleichung2.32}
\frac{K+1}{L_1}<\frac{\delta'}{4} \text{ und } L_1>M_0
\end{equation}

\noindent sowie
\begin{equation}
\label{Gleichung2.33}
L_1>N_\delta
\end{equation}

\noindent erfüllt sind. Weiterhin wählen wir die Folge von Intervallen $\big( (L_j, M_j)\big)_{j=1}^J$ aus Schritt I sowie die Folge $\big(N(M_j)\big)_{j=1}^J$ aus Schritt II so, dass für alle $j\in \{1, \dots, J-1\}$ \begin{equation}
\label{Gleichung2.34b}
\frac{M_j}{L_{j+1}}<\delta'', \ L_{j+1}>N(M_j), \ \frac{M_j}{N(M_{j+1})}<\delta' \text{ und }N(M_{j+1})>{N(M_j)}
\end{equation} erfüllt sind (siehe nachfolgende Bemerkung \ref{Bemerkung42}). Weiterhin wählen wir ein $N\in \mathbb{N}$ mit 
\begin{equation}
\label{Gleichung2.39}
N> N(M_J) \text{ sowie } \frac{M_J}{N}<\frac{\delta'}{4}.
\end{equation}

\begin{Bemerkung}
	\label{Bemerkung42}
\begin{itemize}\item[]
	\item[(1)] Um die Bedingungen \eqref{Gleichung2.32} und \eqref{Gleichung2.33} zu erfüllen, wählen wir zunächst $\mathfrak{k}_1\in \mathbb{N}$ so, dass die entsprechenden Ungleichungen mit $\mathfrak{k}_1$ anstelle von $L_1$ Gültigkeit besitzen. Nach Behauptung \ref{Behauptung2.6} können wir das Intervall $(L_1,M_1)$ so wählen, dass $L_1>\mathfrak{k}_1$ erfüllt ist. $N(M_1)$ erhalten wir aus Schritt II.\\
	\noindent Nun wählen wir $\mathfrak{k}_2\in \mathbb{N}$ so, dass die Bedingungen in \eqref{Gleichung2.34b} mit $\mathfrak{k}_2$ anstelle von $L_2$ erfüllt sind. Nach Behauptung \ref{Behauptung2.6} ist die Konstruktion des Intervalls $(L_1, M_1)$ unabhängig von $\mathfrak{k}_2$ (sowie allen übrigen $\mathfrak{k}_j$, $j=3, \dots, J$). Wir können nun das Intervall $(L_2,M_2)$ so wählen, dass $L_2>\mathfrak{k}_2$ gilt. Nun wählen wir $N(M_2)$ in Behauptung \ref{Behauptung3.21} so, dass $\frac{M_1}{N(M_2)}<\delta'$ und $N(M_2)>N(M_1)$ gilt.\\
	Analog dazu können wir  für $j>2$ induktiv die Konstanten $\mathfrak{k}_j, L_j, M_j$ sowie $N(M_j)$ so wählen, dass die Bedingungen in \eqref{Gleichung2.34b} erfüllt sind.
	\item[(2)] Für $G$ gilt nach Behauptung \ref{Behauptung2.8} $\nu(G)>1-\delta$ und nach Behauptung \ref{Behauptung2.6} sowie nach Wahl von $\delta$ in \eqref{Gleichung2.31b} $\nu(B)>c>\delta$. Deshalb ist $\nu(B\cap G)>0$ erfüllt. Da $g$ nach Voraussetzung in Satz \ref{Satz2.5} ein Element aus $L^\infty(Y,\nu)$ ist, gilt $|g(y)|\leq ||g||_{L^\infty(Y,\nu)}$ für $\nu$"=fast alle $y\in Y$. Weil die Abbildung $S$ weiterhin maßerhaltend ist und der Schnitt abzählbar vieler Mengen mit Maß 1 wieder volles Maß hat, existiert deshalb ein $y\in B\cap G$ mit $|g(S^ny)|\leq ||g||_{L^\infty(Y,\nu)},\ n=1, \dots, N$.
\end{itemize}
\end{Bemerkung}

 \noindent In den nächsten Abschnitten werden wir Folgen $\big(c_n^j\big)_{n=1}^N$, $j=1, \dots, J$ induktiv definieren. Dazu gehen wir in umgekehrter Reihenfolge in $j$ vor, das heißt wir beginnen mit $j=J$. Für gegebenes $j$ konstruieren wir zunächst sogenannte Basisintervalle $(l_1^j, m_1^j], (l_2^j, m_2^j], \dots, (l_{\kappa_j}^j, m_{\kappa_j}^j]$. Die Folge $\big(c_n^j\big)_{n=1}^N$ definieren wir anschließend in Abhängigkeit von diesen Basisintervallen. Im nächsten Induktionsschritt definieren wir die Basisintervalle für ${j-1}$ innerhalb der gerade definierten Basisintervalle $(l_1^j, m_1^j],\dots, (l_{\kappa_j}^j, m_{\kappa_j}^j]$ und in Abhängigkeit von der Folge $\big(c_n^j\big)_{n=1}^N$. Aufgrund dieser schrittweisen Konstruktion der Basisintervalle für $j=J, J-1, J-2, \dots, 1$, werden wir von nun an von Basisintervallen der $j$'ten Schicht sprechen. Abbildung \ref{Abbildung} stellt die Konstruktion dieser Basisintervalle beispielhaft dar.\\ 
 \begin{figure}
 	\includegraphics[trim=0.5cm 0.4cm 0.3cm 1cm, clip=true, width=1.0\textwidth]{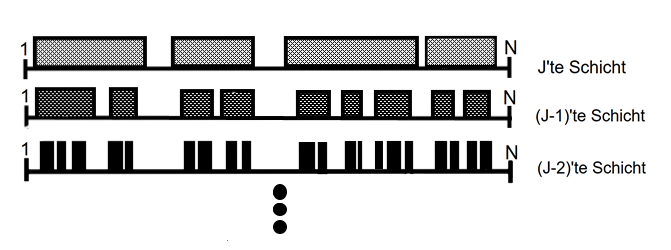}
 	\caption{Beispielhafte Darstellung der Konstruktion der Basisintervalle der J'ten, (J-1)'ten und (J-2)'ten Schicht. Die Abbildung ist angelehnt an die Darstellung in \cite[S.87]{Assani2003}.}
 	\label{Abbildung}
 \end{figure}
\noindent \todo{eventuell händisch einrücken} Wir fixieren nun für den gesamten Beweisschritt ein $y\in B\cap G$ mit \todo{schauen, dass die Grafik nicht komisch dazwischenrutscht} \begin{equation} \label{Gleichung411}|g(S^ny)|\leq ||g||_{L^\infty(Y,\nu)}, \, n=1, \dots, N.\end{equation}  (vgl. Bemerkung \ref{Bemerkung42}(2)) und zeigen die folgende Behauptung:
 \begin{Behauptung}
 	\label{Behauptung43}
 Seien die Voraussetzungen \eqref{Gleichung2.34d} bis \eqref{Gleichung2.39} erfüllt. Dann existieren Folgen $\big(c_n^j\big)_{n=1}^N\in \mathbb{C}^N, \ j=1, \dots, J$ und Basisintervalle $(l_1^j,m_1^j], \dots, (l_{\kappa_j}^j, m_{\kappa_j}^j]\subseteq \mathbb{N},\ j=1, \dots, J$, sodass für alle $j\in \{1, \dots, J\}$  die folgenden Eigenschaften erfüllt sind:
 \begin{itemize}
 	\item[(1)] Die Basisintervalle $(l_1^j,m_1^j], \dots, (l_{\kappa_j}^j, m_{\kappa_j}^j]$ der $j$'ten Schicht sind paarweise disjunkt und mit $n_j(\cdot)$ aus Behauptung \ref{Behauptung2.6} gilt $$S^{l_k^j}y\in B \text{ sowie } m_k^j-l_k^j=n_j(S^{l_k^j}y)\in (L_j, M_j), \ \ k=1, \dots, \kappa_j.$$
 	\item[(2)] Für $j<J$ liegen die Basisintervalle der $j$'ten Schicht innerhalb der Basisintervalle der ($j$+1)'ten Schicht.
 	\item[(3)] Für ein beliebiges Basisintervall $(l,m]$ der $j$'ten Schicht sowie $B_{m-l}^*$ aus Schritt II/\eqref{Gleichung3.22} gilt $$\big(c_{l+1}^k, c_{l+2}^k, \dots, c_{m}^k\big)\in B_{m-l}^*\text{ für alle }k\in \{j+1, \dots, J\}.$$
 	\item[(4)] Es gilt \begin{equation*}
 	c_n^j:=
 	\begin{cases}
 	f(T^{n-l_k^j}x), &n\in (l_k^j, m_k^j] \text{ für ein } k\in \{1, \dots, \kappa_j\},  \\
 	0, &\text{sonst.} 
 	\end{cases}
 	\end{equation*}
 	\item[(5)] Für die Dichte $p_j$ der Basisintervalle der $j$'ten Schicht in $[1, \dots, N]$, gegeben durch 
 		\begin{equation}
 	\label{Gleichung4.111}
 	p_j:=\frac{\sum_{k=1}^{\kappa_j} \#(l_k^j, m_k^j]}{N},
 	\end{equation} gilt
 	$$p_{j-1}>p_j-(J-j+2)\delta', \ \ \ j=2, \dots, J-1$$ sowie
 $$p_1>1-\frac{J\cdot(J+1)}{2}\delta'.$$
  \end{itemize}
 \end{Behauptung}
\noindent Bevor wir mit dem Beweis der Behauptung beginnen, definieren wir für $j\in \{1, \dots, J\}$ die Menge $\mathcal{C}_j$ komplementärer Indizes der $j$'ten Schicht als 
\begin{equation}
\label{Gleichung3.31}
\mathcal{C}_j:=[1, \dots, N]\setminus \Big(\bigcup_{k=1}^{\kappa_j} (l_k^j, m_k^j]\Big)
\end{equation}
%Für $j>1$ ist also insbesondere $\mathcal{C}_{j}\subset \mathcal{C}_{j-1}$ erfüllt. Außerdem ist jedes Element aus $\mathcal{C}_{j-1}\setminus \mathcal{C}_j$ ein Element aus einem Basisintervall der $j$'ten Schicht.
sowie die Menge $E\subseteq [1, \dots, N]$ durch 
\begin{equation}
\label{Gleichung3.33}
E:=\Big\{n\in[1, \dots, N]\big| \ S^ny\in \bigcup\limits_{j=1}^K S^jB \Big\}.
\end{equation}
Darüber hinaus definieren wir für $\tilde{J}\in \{1, \dots, J-1\}$ und ein beliebiges Basisintervall $(a,b]$ der ($\tilde{J}$+1)'ten Schicht die Mengen 
$$O_j^{\tilde{J}}\subset (a,b],\ j=\tilde{J}+1, \tilde{J}+2, \dots, J$$ durch

\begin{equation} 
\label{Gleichung4.112}
O_j^{\tilde{J}}:=\Big\{k\in (a,b]\big| \ \big(c_k^j, c_{k+1}^j,\dots, c_{k+n-1}^j\big)\in B_n^* \text{ für alle } n\in (L_{\tilde{J}}, M_{\tilde{J}})\Big\}.
\end{equation}
Aus dem Kontext wird jeweils genau hervorgehen, welches Intervall $(a,b]$ wir zur Definition der Mengen $O_j^{\tilde{J}},\ j=\tilde{J}+1, \dots, J$ nutzen. 
\begin{Bemerkung}
	Wir werden die gerade eingeführten Definitionen für festes $j$ beziehungsweise $\tilde{J}$ im nachfolgenden Beweis erst verwenden, nachdem wir die zur Definition jeweils notwendigen Basisintervalle beziehungsweise Folgenglieder konstruiert haben.
\end{Bemerkung}

\begin{proof}[\textbf{Beweis von Behauptung \ref{Behauptung43}:}]
	 In den folgenden Abschnitten werden wir induktiv, beginnend mit der $J$'ten Schicht, die Basisintervalle sowie die Folgen $\big(c_n^j\big)_{n=1}^N, \ j=J, J-1, \dots, 1$ konstruieren.

\subsubsection{Konstruktion der J'ten Schicht}
\label{Abschnitt2.3.5.1}
\noindent Es bezeichne $l_1^J$ den ersten Index aus $\{1, \dots, N\}$,   für welchen $S^{l_1^J}y\in B$, für das in \eqref{Gleichung411} fixierte $y$, erfüllt ist. Die Existenz eines solchen Index zeigen wir mithilfe des folgenden Widerspruchsbeweises:\vspace{-0.2cm}
\begin{proof}[\textbf{Beweis der Existenz von $l_1^J$:}]
	Angenommen,  für alle $k\in \{1, \dots, N\}$ gilt $S^ky\not\in B.$ Wir zeigen zunächst für $K\overset{\eqref{Gleichung2.32}}{\underset{\eqref{Gleichung2.39}}{<}}N$ aus Behauptung \ref{Behauptung2.8}, dass für alle $k\in \mathbb{N}$ mit $K<k\leq N$ 
\begin{equation}
\label{Gleichung3.26}
\psi(S^ky)\overset{\text{Def.}}{=}\mathbbm{1}_{\cup_{j=1}^K S^j B}(S^ky)=0
\end{equation} erfüllt ist:\\ \noindent Angenommen, es existiert ein $k\in \{K+1, \dots, N\}$ mit $S^ky\in \bigcup_{j=1}^K S^jB$. Dann existiert ein $j\in \{1, \dots, K\}$ mit $S^ky \in S^jB$. Da das System $(Y,\mathfrak{B},\nu,S)$ invertierbar ist, folgt $S^{k-j}y\in B.$ Dies ist ein Widerspruch zur Annahme $S^ky \not\in B$ für alle $k\in \{1, \dots, N\}$.\\
\noindent  Folglich gilt wegen $N\overset{\eqref{Gleichung2.39}}{>} M_J > L_J>L_1$ 

\begin{equation}
\label{Gleichung2.36}
\frac{1}{N}\sum_{k=1}^N \psi(S^ky)\overset{\eqref{Gleichung3.26}}{=} \frac{1}{N}\sum_{n=1}^K \psi(S^ky) \leq\frac{K}{N}<\frac{K}{L_1} \overset{\eqref{Gleichung2.32}}{<}\frac{\delta'}{4}.
\end{equation}
\noindent Da wir $N>M_0$ sowie $y\in B\cap G\subseteq G$ gewählt haben, gilt nach \ref{Behauptung2.8}
$\big|\frac{1}{N} \sum_{k=1}^N \psi(S^ky)-1\big|<\frac{\delta'}{4}$ und folglich $$\frac{1}{N}\sum_{k=1}^N\psi(S^ky)>1-\frac{\delta'}{4}.$$
Weil nach Voraussetzung $\delta'\in (0,1)$ erfüllt ist, ist dies ein Widerspruch zu $(\ref{Gleichung2.36})$. Also existiert der Index $l_1^J$. \end{proof}\vspace{0.1cm}
\noindent Nach Wahl des Index $l_1^J$ gilt $S^{l_1^J}y\in B$. Für den Punkt $S^{l_1^J}y$ existiert folglich die Zahl $n_J(S^{l_1^J}y)\in (L_J,M_J)$ aus Schritt I/Behauptung \ref{Behauptung2.6}. Das erste Basisintervall der J'ten Schicht definieren wir nun als $$\big(l_1^J, m_1^J\big]:=\big(l_1^J, l_1^J+n_J(S^{l_1^J}y)\big]\subset \mathbb{N},$$ das heißt als linksoffenes Intervall mit Startpunkt $l_1^J$ und Länge $n_J(S^{l_1^J}y)$.
\noindent Die weiteren Basisintervalle der J'ten Schicht definieren wir induktiv:\\
\noindent Seien für $k>1$ die Basisintervalle 
$\big(l_1^J, m_1^J\big], \dots, \big(l_{k-1}^J, m_{k-1}^J\big]$ bereits konstruiert. Wir bezeichnen mit $l_{k}^J$ den ersten Index nach $m_{k-1}^J$, für welchen $S^{l_{k}^J}y\in B$  sowie $l_{k}^J+n_J(S^{l_{k}^J}y)\leq N$ erfüllt sind. (Falls es keinen solchen Index gibt, brechen wir die Konstruktion ab und setzen $\kappa_J=k-1$.) Das nächste Basisintervall der J'ten Schicht definieren wir als $$\big(l_k^J, m_k^J\big]:=\big(l_{k}^J, l_{k}^J + n_J(S^{l_{k}^J}y)\big].$$ Falls $l_{k}^J+n_J(S^{l_{k}^J}y)>N$ gilt, sind bereits alle Basisintervalle der J'ten Schicht gefunden und wir setzen ebenfalls $\kappa_J:=k-1$.\\ Innerhalb eines beliebigen Basisintervalls $\big(l_k^J, m_k^J\big]$  der J'ten Schicht setzen wir\vspace{-0.15cm} 
$$c_n^J:=f(T^{n-l_k^J}x), \ n\in \big(l_k^J, m_k^J\big].$$ Außerhalb der Basisintervall der J'ten Schicht definieren wir\vspace{-0.15cm}
$$c_n^J:=0.$$ 

\subsubsection{Dichte der Basisintervalle der J'ten Schicht in [1, \dots, N]}
Um die in Behauptung \ref{Behauptung43}(5) formulierte untere Abschätzung für die Dichte $p_1$ der Basisintervalle der ersten Schicht in $[1, \dots, N]$ beweisen zu können, benötigen wir eine untere Abschätzungen für die Dichte der Basisintervalle aller anderen Schichten. Wir beginnen mit der Suche nach einer unteren Abschätzung für die Dichte $p_J$ der Basisintervalle der J'ten Schicht in $[1, \dots, N]$. Bevor wir die zugehörige Behauptung \ref{Behauptung10} beweisen können, sind einige Vorbetrachtungen erforderlich:
\begin{Behauptung}
	\label{Behauptung12}
	
	Für ein beliebiges $\omega\in \bigcup_{j=1}^K S^jB$ existiert ein $k\in \{1, \dots, K+1\}$ derart, dass $S^k\omega\in B$ oder $S^k\omega\in Y\setminus \bigcup_{j=1}^K S^jB$ erfüllt ist. 
\end{Behauptung}

\begin{proof}[\textbf{Beweis:}] 
\noindent Betrachte $S^{K+1}\omega.$ Falls $S^{K+1}\omega\in Y\setminus \bigcup_{j=1}^K S^jB$ gilt, ist nichts zu zeigen. \\
\noindent Falls $S^{K+1}\omega\in \bigcup_{j=1}^K S^jB$ erfüllt ist, existiert ein $j\in \{1, \dots, K\}$ mit $S^{K+1}\omega\in S^jB.$ Aufgrund der Invertierbarkeit von $(Y,\mathfrak{B},\nu,S)$  folgt $S^{K+1-j}\omega\in B$. Für $k:=K+1-j$ gilt $k\in \{1, \dots, K\}$ sowie $S^k\omega\in B.$ \vspace{-0.15cm}
\end{proof}	
\noindent Analog dazu lässt sich auch die folgende Behauptung zeigen:
\begin{Behauptung}
	\label{Behauptung13}
Seien ein $k\in \mathbb{N}$ und $\omega\in Y\setminus \bigcup_{j=1}^K S^jB$ mit $S^k\omega\in \bigcup_{j=1}^K S^jB$ gegeben. Dann existiert ein $\mathbb{N}\ni k_0<k$ mit \vspace{-0.15cm}$$S^{k_0}\omega\in B.$$	
\noindent Das bedeutet, dass sich der $S$"=Orbit des Punktes $\omega$ so lange in $Y\setminus \bigcup_{j=1}^K S^jB$ befindet, bis ein Punkt aus $B$ erreicht wird.\end{Behauptung}
\noindent Mithilfe dieser beiden Behauptungen können wir Aussagen über den Abstand zwischen zwei beliebigen aufeinanderfolgenden Basisintervallen der J'ten Schicht gewinnen. Seien zwei beliebige benachbarte Basisintervalle der J'ten Schicht gegeben durch $(l_k,m_k]$ und $(l_{k+1},m_{k+1}]$. Es sei nochmals daran erinnert, dass wir in Behauptung \ref{Behauptung43} ein $y\in B\cap G$ fixiert haben. Wir können zwei Fälle unterscheiden: $S^{m_k}y\in \bigcup_{j=1}^K S^jB $ oder $S^{m_k}y\in Y\setminus \bigcup_{j=1}^K S^jB.$\vspace{0.25cm}\\
\noindent \textbf{Fall 1:} Es gelte $S^{m_k}y\in \bigcup_{j=1}^K S^jB:$\\
\noindent Wegen Behauptung \ref{Behauptung12} angewendet auf $S^{m_k}y\in B$ existiert ein $n\in \{1, \dots, K+1\}$ derart, dass $S^{m_k+n}y\in B$ oder $S^{m_k+n}y\in Y\setminus \bigcup_{j=1}^K S^jB$ erfüllt ist. Da $l_{k+1}$ nach Konstruktion der Basisintervalle der J'ten Schicht  der erste Index nach $m_k$ ist, für welchen $S^{l_{k+1}}y\in B$ erfüllt ist, bedeutet dies anschaulich: Der Abstand $l_{k+1}-m_k$ zwischen den Basisintervallen $(l_k,m_k]$ und $(l_{k+1},m_{k+1}]$ beträgt höchstens $K+1$ (wenn $S^{m_k+n}y\in B$ für ein $n\in \{1, \dots, K+1\}$ erfüllt ist) \textbf{oder} der $S$"=Orbit des Punktes $S^{m_k}y$ befindet sich nach höchstens $K+1$ Schritten in $Y\setminus \bigcup_{j=1}^K S^jB$.\vspace{0.1cm}\\
\noindent \textbf{Fall 2:} Es gelte $S^{m_k}y\in Y\setminus \bigcup_{j=1}^K S^jB$:\\
\noindent Weil $l_{k+1}$ der erste Index nach $m_k$ ist, für welchen $S^{l_{k+1}}y\in B$ erfüllt ist, gilt wegen Behauptung \ref{Behauptung13} für alle $n\in \mathbb{N}$ mit $m_k<n<l_{k+1}$
$$S^{m_k+n}\in Y\setminus \bigcup_{j=1}^K S^jB.$$ 
\noindent Dies bedeutet anschaulich, dass der $S$"=Orbit des Punktes $S^{m_k}y$ so lange die Menge $Y\setminus \bigcup_{j=1}^K S^jB$ durchläuft, bis das Basisintervall $(l_{k+1},m_{k+1}]$ beginnt.\vspace{0.35cm}\\ 
\noindent Aus Fall 1 in Verbindung mit Fall 2 folgt für den Abstand $l_{k+1}-m_k$ zwischen den beiden aufeinanderfolgenden Basisintervallen $(l_k,m_k]$ und $(l_{k+1},m_{k+1}]$
\begin{equation} \label{Gleichung4.46} l_{k+1}-m_k\leq K+1 \quad\ \ \textbf{ oder }\ \ \quad S^ny\in Y\setminus \bigcup_{j=1}^K S^jB, \ \ m_k+K+1<n<l_{k+1}.\end{equation}
\noindent Indem wir in der Fallunterscheidung den Punkt $S^{m_k}y$ durch den Punkt $y$ ersetzen, erhalten wir eine analoge Aussage für den Abstand bis zum ersten Basisintervall $(l_1,m_1]$. Es gilt 
\begin{equation} \label{Gleichung4.47}l_1\leq K+1\ \ \quad \textbf{ oder }\ \ \quad S^ny\in Y\setminus \bigcup_{j=1}^K S^jB,\  \    K+1<n<l_1.\end{equation} 
\noindent Es ist weiterhin notwendig, eine obere Abschätzung für den Abstand zwischen dem letzten Basisintervall und $N$ zu finden: Sei das letzte Basisintervall der J'ten Schicht gegeben durch $(l_\kappa, m_\kappa]$; der Abstand zwischen dem Intervall $(l_\kappa, m_\kappa]$ und $N$ ist also gegeben durch $N-m_\kappa$. Es bezeichne $l_{\kappa+1}$ den ersten Index nach $m_\kappa$, für welchen $S^{l_{\kappa+1}}y\in B$ erfüllt ist. Analog zu  \eqref{Gleichung4.46} können wir folgern, dass \vspace{-0.4cm}\begin{equation} \label{Gleichung4.48} l_{\kappa+1}-m_\kappa \leq K+1\ \ \quad \textbf{ oder }\ \ \quad S^ny \in Y\setminus \bigcup_{j=1}^K S^jB, \ \ m_\kappa + K+1 < n<l_{\kappa +1}\end{equation}\vspace{-0.6cm}\\ \noindent gilt. Da $(l_\kappa, m_\kappa]$ das letzte Basisintervall der J'ten Schicht ist, gilt nach Konstruktion der Basisintervalle $l_{\kappa+1}+n_J\big(S^{l_{\kappa+1}}y\big)>N.$ Wegen $n_J\big(S^{l_{\kappa+1}}y\big)\in (L_J,M_J)$ gilt also insbesondere $l_{\kappa+1}+M_J>N$. Zusammen mit \eqref{Gleichung4.48} folgt für den Abstand $N-m_\kappa$ zwischen dem letzten Basisintervall $(l_\kappa, m_{\kappa}]$ und $N$:\vspace{-0.2cm}
\begin{align} 
\label{Gleichung4.49}
&N-m_\kappa=(N-l_{\kappa +1})+(l_{\kappa+1}-m_\kappa)<M_J+ (K+1)\\
\nonumber \textbf{oder}\qquad\ & N-l_{\kappa+1}<M_J\ \text{ und } \ S^ny\in Y\setminus \bigcup_{j=1}^K S^jB,\ m_\kappa + (K+1)<n<l_{\kappa +1}.
\end{align}\vspace{-0.6cm}\\
\noindent Mithilfe dieser Vorbetrachtungen können wir eine untere Abschätzung für die Dichte der Basisintervalle der J'ten Schicht in $[1, \dots, N]$ finden:
\begin{Behauptung}
	\label{Behauptung10}
Für die Dichte $p_J$ dieser Basisintervalle in $[1, \dots, N]$ gilt \vspace{-0.15cm}$$p_J>1-\delta'.$$
\end{Behauptung}
\begin{proof}[\textbf{Beweis:}]
Aus der Definition der Menge $\mathcal{C}_J$ in \eqref{Gleichung3.31} sowie der Dichte $p_J$ in \eqref{Gleichung4.111} folgt
\vspace{-0.15cm}$$p_J=1-\frac{\# \mathcal{C}_J}{N}.$$ Um eine untere Abschätzung für $p_J$ zu gewinnen, genügt es folglich, eine obere Abschätzung für $\frac{\# \mathcal{C}_J}{N}$ zu finden. Dazu nutzen wir die in \eqref{Gleichung3.33} definierte Menge $E$ und erhalten\vspace{-0.15cm}
\begin{equation}
\label{Gleichung3.32}
\begin{split}
\frac{\# \mathcal{C}_J}{N}&=\frac{\# \bigg([1, l_1^J]\cup\Big( \bigcup\limits_{k=1}^{\kappa_J -1} (m_k^J, l_{k+1}^J]\Big)\cup (m_{\kappa_J}^J, N]\bigg)}{N}\\
&=\frac{\#\bigg(\Big([1, l_1^J]\cup \Big(\bigcup\limits_{k=1}^{\kappa_J -1} (m_k^J, l_{k+1}^J]\Big)\cup (m_{\kappa_J}^J, N]\Big)\cap \Big(E\cup \big([1, \dots, N]\setminus E\big)\Big)\bigg)}{N}\\
&\leq \underbrace{\frac{\#\Big([1, l_1^J]\cap E\Big)+\#\Big(\bigcup\limits_{k=1}^{\kappa_J-1} (m_k^J, l_{k+1}^J]\cap E\Big)+\#\Big((m_{\kappa_J}^J, N]\cap E\Big)}{N}}_{(I)}+\underbrace{\frac{\#\Big([1, \dots, N]\setminus E\Big)}{N}}_{(II)}. 
\end{split}
\end{equation}\vspace{-0.4cm}\\
\noindent Zu (I): Wegen \eqref{Gleichung4.46}, \eqref{Gleichung4.47} und \eqref{Gleichung4.49} gilt\vspace{-0.15cm}
\begin{equation} 
\label{Gleichung4.50}
(I)\leq \frac{(K+1)\ +\ \sum\limits_{k=1}^{\kappa_J-1} (K+1)\ + \ \big(M_J+ K+1\big)}{N}=\frac{(\kappa_J+1) \cdot (K+1)}{N}+\frac{M_J}{N}.
\end{equation}
Zu (II): Wegen $y\in B\cap G \subseteq G$ sowie $N\overset{\eqref{Gleichung2.39}}{>}M_J>L_1\overset{\eqref{Gleichung2.32}}{>}M_0$, gilt nach Behauptung \ref{Behauptung2.8}\vspace{-0.15cm}
\begin{equation} 
\label{Gleichung4.51}
\begin{split}
\frac{\#\Big([1, \dots, N]\setminus E\Big)}{N}
&=\frac{\sum\limits_{n=1}^N \mathbbm{1}_{Y\setminus \bigcup_{j=1}^K S^jB} (S^ny)}{N}
=\frac{N-\sum\limits_{n=1}^N \mathbbm{1}_{\bigcup_{j=1}^K S^jB}(S^ny)}{N}\\
&\hspace{-0.2cm}\overset{\text{Def.}}{\underset{\text{von }\psi}{=}}1-\frac{1}{N}\sum\limits_{n=1}^N \psi(S^ny) <\frac{\delta'}{4}.
\end{split}
\end{equation}
\noindent Weil die Länge der Basisintervalle der J'ten Schicht eine Zahl aus $(L_J,M_J)$ ist, ist für die Anzahl $\kappa_J$ dieser Basisintervalle\vspace{-0.1cm}
\begin{equation} 
\label{Gleichung4.52}
\kappa_J < \frac{N}{L_J}
\end{equation}
erfüllt. Zusammen mit  \eqref{Gleichung4.50} und \eqref{Gleichung4.51} folgt in \eqref{Gleichung3.32}\vspace{-0.2cm}
\begin{equation}
\label{Gleichung4.53}
\begin{split}
\frac{\# \mathcal{C}_J}{N}&<\frac{(\kappa_J+1) \cdot (K+1)}{N}+\frac{M_J}{N}+\frac{\delta'}{4}\\
&\hspace{-0.2cm}\overset{\eqref{Gleichung4.52}}{\leq} \frac{(\frac{N}{L_J}+1)\cdot (K+1)}{N}+\frac{M_J}{N}+\frac{\delta'}{4}\\
&=\underbrace{\frac{K+1}{L_J}}_{<\frac{K+1}{L_1}}+\underbrace{\frac{K+1}{N}}_{<\frac{K+1}{L_1}}+\underbrace{\frac{M_J}{N}}_{\overset{\eqref{Gleichung2.39}}{<}\frac{\delta'}{4}}+\frac{\delta'}{4}\\
&\hspace{-0.2cm}\overset{\eqref{Gleichung2.32}}{<}\delta'.
\end{split}
\end{equation}
Damit erhalten wir für die Dichte $p_J=1-\frac{\#\mathcal{C}_J}{N}$ der Basisintervalle der J'ten Schicht in $[1, \dots, N]$\vspace{-0.15cm} $$p_J>1-\delta'.$$ \noindent Die Behauptung \ref{Behauptung10} ist damit bewiesen.
\end{proof}

\subsubsection{Die (J-1)'te Schicht}
Die Basisintervalle der (J-1)'ten Schicht konstruieren wir  innerhalb der Basisintervalle der J'ten Schicht.  Wir definieren sie ähnlich wie die Basisintervalle der J'ten Schicht, jedoch mit dem Unterschied, dass die Basisintervalle der (J-1)'ten Schicht zusätzlich eine sogenannte Orthogonalitätsbedingung erfüllen müssen.\\
\noindent Sei ein beliebiges Basisintervall der J'ten Schicht gegeben durch $(l,m]$. Die Basisintervalle der (J-1)'ten Schicht innerhalb dieses Intervalls konstruieren wir nun folgendermaßen: Es bezeichne $l_1^{J-1}$ den ersten Index aus $(l,m]$, für welchen $S^{l_1^{J-1}}y\in B$ erfüllt ist. Weiterhin gelte für die Folgenglieder $c_n^J, n\hspace{-0.1cm}\in\hspace{-0.1cm} \big(l_1^{J-1}, l_1^{J-1}\hspace{-0.1cm}+\hspace{-0.1cm}n_{J-1}(S^{l_1^{J-1}}y)\big]$ der bereits konstruierten Folge $\big(c_n^J\big)_{n=1}^N$
\vspace{-0.15cm}
\begin{equation}
\label{Gleichung3.35}
\big(c_{l_1^{J-1}+1}^J, c_{l_1^{J-1}+2}^J, \dots, c_{l_1^{J-1}+n_{J-1}(S^{l_1^{J-1}}y)}^J\big)\in B_{n_{J-1}(S^{l_1^{J-1}}y)}^*.
\end{equation}
\noindent  Dann ist das erste Basisintervall der (J-1)'ten Schicht in $(l,m]$ gegeben durch $$\big(l_1^{J-1},m_1^{J-1}\big]:=\big(l_1^{J-1}, l_1^{J-1}+n_{J-1}(S^{l_1^{J-1}}y)\big].$$ Die Anforderung \eqref{Gleichung3.35} bezeichnen wir nachfolgend als Orthogonalitätsbedingung, die Motivation für diese Begriffsbildung ergibt sich aus dem sich anschließenden Abschnitt. Die Existenz des Index folgt analog zu \eqref{Gleichung2.36}, indem man $N$ durch $b-a>L_J>L_1$ ersetzt, in Verbindung mit dem Beweis  zu der später formulierten Behauptung \ref{Behauptung10b}.\\
Seien für $k>2$ die Basisintervalle $\big(l_1^{J-1}, m_1^{J-1}\big], \big( l_2^{J-1}, m_2^{J-1}\big], \dots, \big(l_{k-1}^{J-1}, m_{k-1}^{J-1}\big]$ bereits konstruiert. Wir definieren nun $l_{k}^{J-1}$ als den ersten Index nach $m_{k-1}^{J-1}$, für welchen $S^{l_{k}^{J-1}}y\in B$ sowie die Orthogonalitätsbedingung $$\big(c_{l_{k}^{J-1}+1}^J, c_{l_{k}^{J-1}+2}^J, \dots, c_{l_{k}^{J-1}+n_{J-1}(S^{l_{k}^{J-1}}y)}^J\big)\in B_{n_{J-1}(S^{l_{k}^{J-1}}y)}^*$$ erfüllt ist. (Existiert kein solcher Index, brechen wir die Konstruktion im Basisintervall $(l,m]$ ab.) Falls weiterhin $\big(
l_k^{J-1}, l_k^{J-1}+n_{J-1}(S^{l_k^{J-1}}y)\big] \subseteq \big(l,m\big]$ gilt, ist das $k$'te Basisintervall der (J-1)'ten Schicht in $\big(l,m\big]$ gegeben durch $$\big(l_k^{J-1}, m_k^{J-1}\big]:=\big(l_k^{J-1},l_k^{J-1} + n_{J-1}(S^{l_{k}^{J-1}}y)\big].$$ \noindent (Andernfalls bricht die Konstruktion im Basisintervall $(l,m]$ ebenfalls ab.)
\noindent Wir wiederholen diese Konstruktion der Basisintervalle der (J-1)'ten Schicht im Basisintervall $(l,m]$  nun für alle weiteren Basisintervalle der J'ten Schicht und definieren $\kappa_{J-1}$ als die Anzahl der Basisintervalle der (J-1)'ten Schicht in $[1, \dots, N]$.\\
\noindent Für ein beliebiges Basisintervall $\big(
l_k^{J-1}, m_k^{J-1}\big]$ der (J-1)'ten Schicht definieren wir \begin{equation} \label{Gleichung10.1} c_n^{J-1}:=f(T^{n-l_k^{J-1}}x), \ n\in \big(l_k^{J-1}, m_k^{J-1}\big].\end{equation} Außerhalb der Basisintervalle der (J-1)'ten Schicht setzen wir $$c_n^{J-1}:=0.$$
\noindent Der nachfolgende Abschnitt motiviert die Begriffsbildung "`Orthogonalitätsbedingung"'. 
\subsubsection*{Orthogonalität}
In Schritt II haben wir die Menge $C\in \mathfrak{A}$ mit $\mu(C)>1-\delta''$ sowie ein $N_\delta\in \mathbb{N}$ eingeführt, sodass für alle $\tilde{N}>N_\delta$ und alle $\xi \in C$\vspace{-0.15cm}
\begin{equation}
\label{Gleichung3.27}
\Big|\frac{1}{\tilde{N}}\sum_{n=1}^{\tilde{N}} f(T^nx)\overline{f(T^n\xi)}\Big|<\delta
\end{equation}
erfüllt ist.\\
\noindent Für ein beliebiges Basisintervall $\big(l, l+n_{J-1}(S^ly)\big]$ der (J-1)'ten Schicht ist nach Konstruktion  $\big(c_{l+1}^J,$ $c_{l+2}^J, \dots, c_{l+n_{J-1}(S^ly)}^J\big)\in B_{n_{J-1}(S^ly)}^*$ erfüllt. 
 Weiterhin gilt nach Behauptung \ref{Behauptung2.6} $$(L_{J-1}, M_{J-1})\ni n_{J-1}(S^ly)>L_{J-1}>L_1\overset{(\ref{Gleichung2.33})}{>}N_\delta.$$ Aufgrund der Definition von $B^*_{n_{J-1}(S^ly)}$ in Schritt II/\eqref{Gleichung3.22} sind deshalb
\begin{equation} 
\label{Gleichung3.28}
\big(c_{l+1}^J, c_{l+2}^J, \dots, c_{l+n_{J-1}(S^ly)}^J\big)=\big(f(T\xi), f(T^2\xi), \dots, f(T^{n_{J-1}(S^ly)}\xi\big) \text{ für ein }\xi\in C
\end{equation} sowie

\begin{equation}
\label{Gleichung4.61}
\begin{split}
\left|\frac{1}{n_{J-1}(S^ly)}\sum\limits_{n=l+1}^{l+n_{J-1}(S^ly)}c_n^{J-1}\overline{c_n^J}\right|&\overset{\eqref{Gleichung3.28}}{=}
\left|\frac{1}{n_{J-1}(S^ly)}\sum\limits_{n=1}^{n_{J-1}(S^ly)}c_{n+l}^{J-1}\overline{f(T^n\xi)}\right|\\
&\overset{\eqref{Gleichung10.1}}{=}\left|\frac{1}{n_{J-1}(S^ly)}\sum\limits_{n=1}^{n_{J-1}(S^ly)}f(T^nx)\overline{f(T^n\xi)}\right|\\
&\overset{\eqref{Gleichung3.27}}{<}\delta 
\end{split}
\end{equation}
erfüllt. Aus diesem Grund bezeichnen wir $(l, l+n_{J-1}(S^ly)]$ als zur J'ten Schicht orthogonales Intervall und $l$ als orthogonalen Index der (J-1)'ten Schicht.\vspace{-0.3cm}\\ 

\noindent Mithilfe von Schritt II/Behauptung \ref{Behauptung3.21}
	können wir eine untere Abschätzung für die Dichte der orthogonalen Indizes $l$ in einem beliebigen Basisintervall $(a,b]$ der J'ten Schicht finden:\\ 
\noindent Nach Konstruktion der Basisintervalle der J'ten Schicht gilt $b-a \in (L_J,M_J)$. Insbesondere ist also $b-a > L_J\overset{\eqref{Gleichung2.34b}}{>}N(M_{J-1})$ erfüllt. Nach Schritt II/Behauptung \ref{Behauptung3.21} gilt folglich
	\begin{equation} 
	\label{Gleichung3.30}	
	\frac{1}{b-a}\cdot \# \left\{k\in \{1, \dots, b-a\}\Big| \ \big(f(T^kx),\dots, f(T^{k+n-1}x)\big)\in B_n^*\ \forall n\in (L_{J-1},M_{J-1})\right\}>1-\delta''.
	\end{equation}
	
\noindent Nach Konstruktion der Folge $\big(c_k^J\big)_{k=1}^N$ im Abschnitt \ref{Abschnitt2.3.5.1} ist für alle $k\in (a, b]$
$$c_k^J=f(T^{k-a}x)$$ erfüllt. Da für alle $n\in (L_{J-1}, M_{J-1})$ insbesondere $n<M_{J-1}$ erfüllt ist, gilt für alle $k\in \big(a, \dots, b-M_{J-1}\big]$ und alle $n\in (L_{J-1}, M_{J-1})$:
\begin{equation} 
\label{Gleichung3.29}\big(c_k^J, c_{k+1}^J, \dots, c_{k+n-1}^J\big)=\big(f(T^{k-a}x), f(T^{k+1-a}x),\dots, f(T^{k+n-1 -a}x)\big).
\end{equation}
%		\nonumber 
%	& \quad \ \frac{1}{b-a}\cdot \#\left\{k\in \{a+1, \dots, b\}\Big| \ \big(c_k^J, c_{k+1}^J,\dots, c_{k+n-1}^J\big)\in B_n^*\ \forall n\in R_{J-1}\right\}\\

\noindent Damit erhalten wir
\begin{equation}
\label{Gleichung2.48}
\begin{split}
 \frac{\#(O_J^{J-1})}{b-a}&\hspace{-0.15cm}\overset{\eqref{Gleichung4.112}}{=} \frac{\#\big\{k\in (a,b]\big| \ \big(c_k^J, c_{k+1}^J,\dots, c_{k+n-1}^J\big)\in B_n^*\ \forall n\in (L_{J-1}, M_{J-1})\big\}}{b-a}\\
&\hspace{0.1cm}\geq \frac{\#\big\{k\in (a,b-M_{J-1}]\big| \ \big(c_k^J, c_{k+1}^J,\dots, c_{k+n-1}^J\big)\in B_n^*\ \forall n\in (L_{J-1}, M_{J-1})\big\}}{b-a}\\
&\hspace{-0.1cm}\overset{\eqref{Gleichung3.29}}{=}\hspace{-0.2cm} \frac{\big\{k\in \{1, \dots, b\hspace{-0.05cm}-\hspace{-0.05cm}a\hspace{-0.05cm}-\hspace{-0.05cm}M_{J-1}\}\big|\hspace{-0.05cm}\big(f(T^kx),\dots, f(T^{k+n-1}x)\big)\hspace{-0.05cm}\in \hspace{-0.05cm} B_n^*\ \forall n \hspace{-0.05cm} \in \hspace{-0.05cm} (L_{J-1}, M_{J-1})\big\}}{b-a}\\
&\hspace{-0.1cm}\overset{\eqref{Gleichung3.30}}{\geq}1-\delta'' - \frac{M_{J-1}}{\underbrace{b-a}_{\in (L_J, M_J)}}\\
&\hspace{0.1cm}> 1-\delta''-\frac{M_{J-1}}{L_J}\\
&\hspace{-0cm}\hspace{-0.1cm}\overset{\eqref{Gleichung2.34b}}{>}1-\delta''-\delta''=1-2\delta''.
\end{split}
\end{equation}
Das bedeutet, dass die Menge aller Indizes $k\in (a,b]$, für welche $$\big(c_k^J, c_{k+1}^J, \dots, c_{k+n-1}^J\big)\in B_n^*\text{ für alle } n\in (L_{J-1},M_{J-1})$$ erfüllt ist, eine Dichte $>1-2\delta''$ im Intervall $(a,b]$ hat. Diese untere Abschätzung für die Dichte der orthogonalen Indizes im beliebig gewählten Basisintervall $(a,b]$ der J'ten Schicht benutzen wir nun, um eine untere Abschätzung für die Dichte der Basisintervalle der (J-1)'ten Schicht in $[1, \dots, N]$ zu finden.

\subsubsection*{Dichte der Basisintervalle der (J-1)'ten Schicht in [1, \dots, N]}
Um eine untere Abschätzung für die Dichte $p_{J-1}$ der Basisintervalle der (J-1)'ten Schicht in $[1, \dots, N]$ zeigen zu können, beweisen wir zunächst das folgende Hilfsresultat:
 \begin{Behauptung}
	\label{Behauptung11}
	Sei $(a,b] \subset [1, \dots, N]$ ein beliebiges Basisintervall der J'ten Schicht. Dann gilt für die in  \eqref{Gleichung3.31} und \eqref{Gleichung3.33} eingeführten Mengen $\mathcal{C}_{J-1}, \mathcal{C}_{J}$ und $E$
	$$\frac{\#\Big(\big(\mathcal{C}_{J-1}\setminus \mathcal{C}_{J}\big)\cap E\cap (a,b]\Big)}{b-a}<\frac{7}{4}\delta'.$$ 
\end{Behauptung}
\begin{proof}[\textbf{Beweis:}]
	Da wir die Basisintervalle der (J-1)'ten Schicht innerhalb der Basisintervalle der J'ten Schicht konstruiert haben, ist insbesondere $\mathcal{C}_J\subset \mathcal{C}_{J-1}$ erfüllt. Zunächst suchen wir eine obere Abschätzung für die Dichte der Menge $\#\big((\mathcal{C}_{J-1}\setminus \mathcal{C}_J)\cap E\big)$ in $(a,b]$ unter der Annahme, die Basisintervalle der (J-1)'ten Schicht wären \textbf{ohne} Beachtung der Orthogonalitätsbedingung konstruiert worden. Analog zum Beweis der oberen Abschätzung für $\frac{\#(\mathcal{C}_J\cap E)}{N}$ in \eqref{Gleichung3.32} bis  \eqref{Gleichung4.53} können wir unter dieser Annahme zeigen, dass
	\begin{equation}
	\label{Gleichung4.65}
	\frac{\#\Big(\big(\mathcal{C}_{J-1}\setminus \mathcal{C}_{J}\big)\cap E\cap (a,b]\Big)}{b-a}< \frac{\frac{b-a}{L_{J-1}}(K+1)}{b-a}+\frac{M_{J-1}}{b-a}+\frac{K+1}{b-a}
	<\frac{3}{4}\delta'
	\end{equation}
	gilt.\\ \noindent Nach Konstruktion erfüllen jedoch alle Basisintervalle der (J-1)'ten Schicht eine Orthogonalitätsbedingung. Deshalb erhöht sich die obere Abschätzung in \eqref{Gleichung4.65}:\\
	\noindent Seien $(l,m] \subset (a,b]$ ein beliebiges Basisintervall der (J-1)'ten Schicht in $(a,b]$ und $l'>m$ der erste Index nach $m$, für welchen $S^{l'}y\in B$ erfüllt ist. Definiere $m':=l'+n_{J-1}(S^{l'}y)$. Das Intervall $(l',m']$ ist  ein Basisintervall der (J-1)'ten Schicht in $(a,b]$, wenn das  Intervall $(l',m']$ die Orthogonalitätsbedingung  $$\big(c_{l'+1}^J,c_{l'+2}^J, \dots, c_{m'}^J\big)\in B_{m'-l'}^*$$ sowie $(l',m']\subset (a,b]$ erfüllt. Gilt die Orthogonalitätsbedingung für das Intervall $(l',m']$ nicht, erhalten wir mit der in \eqref{Gleichung4.112} definierten Menge $O_J^{J-1}$ \begin{equation} \label{Gleichung4.114} l'\in (a,b]\setminus O_J^{J-1}\subset (a,b]\end{equation} 
	\noindent  und wir fahren mit der Suche nach dem nächsten Basisintervall der (J-1)'ten Schicht  in $(a,b]$ fort: Sei $l''>l'$ der erste Index nach $l'$ für welchen $S^{l''}y\in B$ erfüllt ist. Analog zum Beweis von $\eqref{Gleichung4.46}$ schließen wir, dass 
	\begin{equation} 
	\label{Gleichung3.37}
	l''-l'\leq K+1 \text{\textbf{ oder }} (l'+K+1, l'')\subset\ (a,b]\setminus E
	\end{equation} 
	\noindent erfüllt ist.	Deshalb erhöht sich  die obere Abschätzung für $\frac{\#\big((\mathcal{C}_{J-1}\setminus \mathcal{C}_{J})\cap E\cap (a,b]\big)}{b-a}$ in \eqref{Gleichung4.65} für einen Index $l'\in (a,b]\setminus O_J^{J-1}$  um höchstens $\frac{K+1}{b-a}$. Da die Anzahl dieser Indizes durch $\#\big( (a,b]\setminus O_J^{J-1}\big)$ begrenzt ist, erhöht sich die obere Abschätzung in \eqref{Gleichung4.65} folglich um höchstens
	\begin{equation}
	\label{Gleichung4.69}
	%\begin{split}
	\frac{\#\Big((a, b]\setminus O_{J}^{J-1}\Big)\cdot (K+1)}{b-a}
	=\Big(1-\frac{\#O_J^{J-1}}{b-a}\Big)\cdot (K+1)
	\overset{\eqref{Gleichung2.48}}{<}2\delta''\cdot (K+1)
	\overset{\eqref{Gleichung2.32b}}{<}\delta'. 
	%\end{split}
	\end{equation} 
	  Damit erhalten wir\vspace{-0.15cm}
	\begin{equation*}
	\frac{\#\Big(\big(\mathcal{C}_{J-1}\setminus \mathcal{C}_{J}\big)\cap E\cap (a,b]\Big)}{b-a}\overset{\eqref{Gleichung4.65}}{\underset{\eqref{Gleichung4.69}}{<}}\frac{3}{4}\delta'+\delta'=\frac{7}{4}\delta'.\vspace{-1cm}
	\end{equation*} 
\end{proof}

\noindent Aufbauend auf diesem Hilfsresultat beweisen wir die folgende Behauptung:
\begin{Behauptung}
	\label{Behauptung10b} \samepage{
Für die Dichte $p_{J-1}$ der Basisintervalle der (J-1)'ten Schicht in $[1, \dots, N]$ gilt $$p_{J-1}>1-3\delta'.$$}
\end{Behauptung}
\begin{proof}[\textbf{Beweis:}]
 \noindent Aufgrund der Definition von $\mathcal{C}_{J-1}$ gilt
\begin{equation}
\label{Gleichung3.34}
%\begin{split}
1-p_{J-1}=\frac{\# \mathcal{C}_{J-1}}{N}=\frac{\#\Big(\mathcal{C}_J\cup \big(\mathcal{C}_{J-1}\setminus \mathcal{C}_J\big)\Big)}{N}
=\frac{\#\mathcal{C}_J}{N}+\frac{\#\big(\mathcal{C}_{J-1}\setminus \mathcal{C}_J\big)}{N}
\overset{\eqref{Gleichung4.53}}{<}\delta'+\frac{\#\big(\mathcal{C}_{J-1}\setminus \mathcal{C}_J\big)}{N}.
%\end{split}
\end{equation}
\noindent Es folgt
 \begin{equation}
 \label{Gleichung3.36}
\begin{split}
\frac{\# \big(\mathcal{C}_{J-1}\setminus \mathcal{C}_J\big)}{N}&
=\frac{\#\Big(\big(\mathcal{C}_{J-1}\setminus \mathcal{C}_J\big)\cap E\Big)}{N}+\frac{\# \Big(\big(\mathcal{C}_{J-1}\setminus \mathcal{C}_J\big)\cap \big([1, \dots, N]\setminus E\big)\Big)}{N}\\
&\leq \frac{\#\Big(\big(\mathcal{C}_{J-1}\setminus \mathcal{C}_J\big)\cap E\Big)}{N} + \frac{\#\big([1, \dots, N]\setminus E\big)}{N}.
\end{split}
 \end{equation}
 Zusammen mit \eqref{Gleichung4.51} erhalten wir
 \begin{equation} 
 \label{Gleichung3.39}
 \frac{\# \big(\mathcal{C}_{J-1}\setminus \mathcal{C}_J\big)}{N}<\frac{\#\Big(\big(\mathcal{C}_{J-1}\setminus \mathcal{C}_J\big)\cap E\Big)}{N}+\frac{\delta'}{4}.
 \end{equation} 
  Um eine untere Abschätzung für $p_{J-1}$ zu finden, reicht es also aus,  eine obere Abschätzung für $\frac{\#\big((\mathcal{C}_{J-1}\setminus \mathcal{C}_J)\cap E\big)}{N}$ zu gewinnen.
\noindent Nach Definition der Mengen $\mathcal{C}_J$ und $\mathcal{C}_{J-1}$ in \eqref{Gleichung3.31} ist jedes Element aus $\mathcal{C}_{J-1}\setminus \mathcal{C}_{J}$  ein Element aus einem Basisintervall der J'ten Schicht. Da die Abschätzung in Behauptung \ref{Behauptung11} für ein beliebiges Basisintervall der J'ten Schicht gilt, folgt
\begin{equation} \label{Gleichung4.60}\frac{\#\Big(\big(\mathcal{C}_{J-1}\setminus \mathcal{C}_J\big)\cap E\Big)}{N}< \frac{7}{4}\delta'.\end{equation}
 Damit erhalten wir
\begin{equation}
\label{Gleichung4.73}
\begin{split}
1-p_{J-1}&\overset{\eqref{Gleichung3.34}}{\underset{\eqref{Gleichung3.39}}{<}}\delta' + \frac{\#\Big(\big(\mathcal{C}_{J-1}\setminus \mathcal{C}_J\big)\cap E\Big)}{N}+\frac{\delta'}{4}\\
&\overset{\eqref{Gleichung4.60}}{<}3\delta'
\end{split}
\end{equation}
und folglich
\begin{equation*} 
p_{J-1}>1-3\delta'.\qedhere
\end{equation*}
\end{proof}

\subsubsection{Konstruktion und Dichte der weiteren Schichten}
\label{Abschnitt2.3.5.6}
Die weiteren Basisintervalle sowie die Folgen $\big(c_n^j\big)_{n=1}^N$ für $j<J-1$ definieren wir induktiv: \\
\noindent Seien für $1<\tilde{J}\leq J-1
$ bereits die Basisintervalle der $\tilde{J}$'ten, ($\tilde{J}$+1)'ten, ..., (J-1)'ten und der J'ten Schicht sowie die Folgenglieder $\big(c_n^j\big)_{n=1, \dots, N}^{j=\tilde{J}, \dots, J}$ konstruiert. Die Basisintervalle der ($\tilde{J}$-1)'ten Schicht definieren wir analog zur Konstruktion der (J-1)'ten Schicht innerhalb der Basisintervalle der $\tilde{J}$'ten Schicht und orthogonal zu allen vorherigen Schichten so, dass die Bedingungen (1)-(4) aus Behauptung \ref{Behauptung43} mit $j=\tilde{J}-1$ erfüllt sind.
Aufbauend auf dem vorherigen Abschnitt können wir nun eine untere Abschätzung für die Dichte $p_{\tilde{J}-1}$ der Basisintervalle der ($\tilde{J}$-1)'ten Schicht in $[1, \dots, N]$ zeigen:
\begin{Behauptung}
	\label{Behauptung4.44}
	Sei $1<\tilde{J}\leq J-1$. Dann gilt für die Dichte $p_{\tilde{J}-1}$ $$p_{\tilde{J}-1}>p_{\tilde{J}}-(J-\tilde{J}+2)\delta'$$ \noindent sowie für die Dichte $p_1$
	\begin{equation*}
	p_1>1-\frac{J\cdot (J+1)}{2}\delta'.
	\end{equation*}
	\end{Behauptung}
\begin{proof}[\textbf{Beweis:}]
	\renewcommand{\qedsymbol}{}
  \noindent Nach Definition der Mengen $\mathcal{C}_{\tilde{J}}$ und $\mathcal{C}_{\tilde{J}-1}$ in \eqref{Gleichung3.31} gilt
  \begin{equation}
  \label{Gleichung4.67}
  1-p_{\tilde{J}-1}=\frac{\# \mathcal{C}_{\tilde{J}-1}}{N}=\frac{\# \Big(\mathcal{C}_{\tilde{J}}\cup (\mathcal{C}_{\tilde{J}-1}\setminus \mathcal{C}_{\tilde{J}})\Big)}{N}=\frac{\# \mathcal{C}_{\tilde{J}}}{N}+\frac{\# \big(\mathcal{C}_{\tilde{J}-1}\setminus\mathcal{C}_{\tilde{J}}\big)}{N}=1-p_{\tilde{J}}+\frac{\#\big(\mathcal{C}_{\tilde{J}-1}\setminus\mathcal{C}_{\tilde{J}}\big)}{N}.
  \end{equation}
  \noindent Indem wir in den Berechnungen \eqref{Gleichung3.36} und (\ref{Gleichung3.39}) $J$ durch $\tilde{J}$ ersetzen, erhalten wir:
  \begin{equation} 
  \label{Gleichung4.76}
  \frac{\# \big(\mathcal{C}_{\tilde{J}-1}\setminus \mathcal{C}_{\tilde{J}}\big)}{N}<\frac{\#\Big(\big(\mathcal{C}_{\tilde{J}-1}\setminus \mathcal{C}_{\tilde{J}}\big)\cap E\Big)}{N}+\frac{\delta'}{4}.
  \end{equation}
Zum Beweis der Behauptung ist es also ausreichend, eine obere Abschätzung für $\frac{\#\big((\mathcal{C}_{\tilde{J}-1}\setminus \mathcal{C}_{\tilde{J}})\cap E\big)}{N}$ zu finden. Wie in Behauptung \ref{Behauptung11} zeigen wir dazu für ein beliebiges Basisintervall $(a,b]$ der $\tilde{J}$'ten Schicht eine obere Abschätzung für $\frac{\#\big((\mathcal{C}_{\tilde{J}-1}\setminus \mathcal{C}_{\tilde{J}})\cap E\cap (a,b]\big)}{b-a}$:\\
 \noindent Zunächst sei angenommen, die Basisintervalle der ($\tilde{J}$-1)'ten Schicht wären ohne Beachtung der Orthogonalitätsbedingung konstruiert worden. 
  Indem wir in \eqref{Gleichung4.65} $J$ durch $\tilde{J}$ ersetzen, erhalten wir 
  \begin{equation}
  \label{Gleichung4.68}
\frac{\#\Big(\big(\mathcal{C}_{\tilde{J}-1}\setminus \mathcal{C}_{\tilde{J}}\big)\cap E\cap (a,b]\Big)}{b-a}< \frac{\frac{b-a}{L_{\tilde{J}-1}}(K+1)}{b-a}+\frac{M_{\tilde{J}-1}}{b-a}+\frac{K+1}{b-a}
<\frac{3}{4}\delta'.
\end{equation}
 \noindent  Nach Konstruktion erfüllen jedoch alle Basisintervalle der ($\tilde{J}$-1)'ten Schicht die Orthogonalitätsbedingung aus Behauptung \ref{Behauptung43}(3), deshalb erhöht sich die obere Abschätzung in  \eqref{Gleichung4.68}. Um diese Erhöhung zu quantifizieren, nutzen wir die in \eqref{Gleichung4.112} definierten Mengen $O_j^{\tilde{J}-1}\subset (a,b],\ \tilde{J}\leq j\leq J$. Wie in \eqref{Gleichung2.48} können wir für $\tilde{J}\leq j \leq J$ zeigen, dass
  $\frac{\# O_j^{\tilde{J}-1}}{b-a}>1-2\delta''$ und folglich\vspace{-0.15cm}
  \begin{equation} 
  \label{Gleichung4.70}\frac{\# \big( (a,b] \setminus O_j^{\tilde{J}-1}\big)}{b-a}<2\delta''
  \end{equation} 
  \noindent erfüllt ist. Analog zu \eqref{Gleichung4.69} argumentieren wir, dass sich die obere Abschätzung in \eqref{Gleichung4.68} durch die Bedingung, dass die Basisintervalle der ($\tilde{J}$-1)'ten Schicht orthogonal zu allen vorherigen Schichten konstruiert werden müssen, für jede dieser Schichten $(j=\tilde{J}, \dots, J)$ um maximal $\frac{\# \big( (a,b] \setminus O_j^{\tilde{J}-1}\big)\cdot (K+1)}{b-a}$ erhöht, insgesamt also um höchstens \vspace{-0.15cm}
  $$\sum_{j=\tilde{J}}^J \frac{\# \big( (a,b] \setminus O_j^{\tilde{J}-1}\big)\cdot (K+1)}{b-a}\overset{\eqref{Gleichung4.70}}{<}\sum_{j=\tilde{J}}^{J}  \underbrace{2\delta''\cdot (K+1)                }_{\overset{\eqref{Gleichung2.32b}}{<}\delta'}<  \sum_{j=\tilde{J}}^J \delta'=(J-\tilde{J}+1)\delta'.$$
  \noindent Deshalb ist \vspace{-0.15cm}
  \begin{equation}
   \label{Gleichung4.78}
  \frac{\#\Big(\big(\mathcal{C}_{\tilde{J}-1}\setminus \mathcal{C}_{\tilde{J}}\big)\cap E\cap (a,b]\Big)}{b-a}<\frac{3}{4}\delta'+ (J-\tilde{J}+1) \delta'\end{equation}
  \noindent erfüllt. 
 \noindent  Da nach Konstruktion der Mengen $\mathcal{C}_{\tilde{J}-1}$ und $\mathcal{C}_{\tilde{J}}$ jedes Element aus $\mathcal{C}_{\tilde{J}-1}\setminus \mathcal{C}_{\tilde{J}}$ ein Element aus einem Basisintervall der $\tilde{J}$'ten Schicht ist und die Ungleichung \eqref{Gleichung4.78} für ein beliebiges Basisintervall $(a,b]$ der $\tilde{J}$'ten Schicht gilt, folgt
 \vspace{-0.15cm}
\begin{equation} 
\label{Gleichung4.79}
\frac{\#\Big(\big(\mathcal{C}_{\tilde{J}-1}\setminus \mathcal{C}_{\tilde{J}}\big)\cap E\Big)}{N}<\frac{3}{4}\delta'+ (J-\tilde{J}+1) \delta'.
\end{equation}
\noindent Damit erhalten wir die folgende Abschätzung für die Dichte $p_{\tilde{J}-1}$:\vspace{-0.15cm}
\begin{equation} 
\label{Gleichung4.80}
\begin{split} 
1-p_{\tilde{J}-1}&\overset{\eqref{Gleichung4.67}}{=}1-p_{\tilde{J}}+\frac{\# \big(\mathcal{C}_{\tilde{J}-1}\setminus \mathcal{C}_{\tilde{J}}\big)}{N}\\
&\overset{\eqref{Gleichung4.76}}{<}1-p_{\tilde{J}}+\frac{\#\Big(\big(\mathcal{C}_{\tilde{J}-1} \setminus \mathcal{C}_{\tilde{J}}\big)\cap E\Big)}{N}+\frac{\delta'}{4}\\
&\overset{\eqref{Gleichung4.79}}{<}1-p_{\tilde{J}}+\frac{3}{4}\delta' + (J-\tilde{J}+1)\delta' + \frac{\delta'}{4}\\
&\hspace{0.2cm}=1-p_{\tilde{J}}+(J-\tilde{J}+2)\delta'.
\end{split} 
\end{equation}
  \noindent Induktiv erhalten wir damit für die Dichte $p_1$
  \vspace{-0.25cm}
\begin{align}
1-p_1&<1-p_2+(J-2+2)\delta'\nonumber\\
&<1-p_3 + (J-3+2)\delta'+(J-2+2)\delta'\nonumber \\
&\hspace{0.2cm}\vdots\nonumber  \allowdisplaybreaks\\ 
&<1-p_{J-1}+ \Big(\sum_{j=3}^J j\Big)\delta'\nonumber \\
&<1-p_{J-1}+ \bigg(\frac{J\cdot (J+1)}{2}-3\bigg)\delta'\nonumber \\
&\hspace{-0.16cm}\overset{\eqref{Gleichung4.73}}{<}\frac{J\cdot (J+1)}{2}\delta'.\nonumber
\end{align}
zeigen. Folglich gilt
\begin{equation}
\label{Gleichung2.55b}
p_1>1-\frac{J\cdot (J+1)}{2}\delta'.
\end{equation}
Damit sind die Behauptungen \ref{Behauptung4.44} und \ref{Behauptung43} bewiesen.
\end{proof}\vspace{-0.8cm}
\end{proof}

\subsubsection{Die Eigenschaften $(\alpha)$ und $(\beta)$}
\label{Abschnitt2.3.5.7}
Mithilfe dieser Überlegungen beweisen wir die folgenden beiden Eigenschaften
\begin{itemize}
	\item[$(\alpha)$] Für $j_1, j_2 \in \{1, \dots, J\}, j_1\neq j_2$ gilt $\Big|\frac{1}{N}\sum\limits_{n=1}^N c_n^{j_1}\overline{c_n^{j_2}}\Big|<\delta.$
	\item[$(\beta)$] Für alle $j\in \{1, \dots, J\}$ sowie $\mathcal{RI}$ aus Schritt I ist $\mathcal{RI}\Big(\frac{1}{N}\sum\limits_{n=1}^N c_n^jg(S^ny)\Big)>a-\delta$ erfüllt.
	\end{itemize}

\begin{proof}[\textbf{Beweis für $(\alpha)$:}]
	Ohne Beschränkung der Allgemeinheit gelte $j_1<j_2$. Der Beweis nutzt aus, dass wir die Basisintervalle der $j_1$'ten Schicht orthogonal zu allen vorherigen Schichten konstruiert haben.\\ Seien die Basisintervalle der $j_1$'ten Schicht  gegeben durch $(l_1, m_1], (l_2, m_2], \dots, (l_\kappa, m_\kappa].$

\noindent Da für Indizes $n$ außerhalb dieser Basisintervalle $c_n^{j_1}=0$ gilt, ist die folgende Gleichung erfüllt: 
\begin{equation}
\label{Gleichung4.84}
\bigg|\frac{1}{N}\sum\limits_{n=1}^N c_n^{j_1}\overline{c_n^{j_2}}\bigg|
=\bigg|\frac{1}{N}\sum\limits_{k=1}^\kappa \sum\limits_{n=l_k+1}^{m_k} \mspace{-5mu} c_n^{j_1}\overline{c_n^{j_2}}\bigg|\\
\overset{\text{Def. der}}{\underset{c_n^{j_1}}{=}} \bigg|\frac{1}{N}\sum\limits_{k=1}^\kappa \sum\limits_{n=1}^{m_k-l_k}f(T^nx)  \overline{c_{n+l_k}^{j_2}}\bigg|\\
\end{equation} 
Da die Basisintervalle der $j_1$'ten Schicht nach Konstruktion orthogonal zu allen vorherigen Schichten, also insbesondere auch orthogonal zu den Basisintervallen der $j_2$'ten Schicht sind, gilt
$$\big(c_{l_k +1}^{j_2},c_{l_k +2}^{j_2}, \dots, c_{m_k}^{j_2} \big)\in B_{m_k-l_k}^*.$$ Nach Konstruktion der Menge $B_{m_k-l_k}^*$ in Schritt II existiert folglich ein $\xi_k \in C$, sodass 
\begin{equation} 
\label{Gleichung4.85}
\big(c_{l_k +1}^{j_2},c_{l_k +2}^{j_2}, \dots, c_{m_k}^{j_2} \big)=\big(f(T\xi_k), f(T^2\xi_k), \dots, f(T^{m_k-l_k} \xi_k\big)
\end{equation}\noindent erfüllt ist. \\
\noindent In \eqref{Gleichung4.84} gilt folglich
\begin{equation}
\label{Gleichung4.86}
\bigg|\frac{1}{N}\sum\limits_{k=1}^\kappa \sum\limits_{n=1}^{m_k-l_k}f(T^nx)  \overline{c_{n+l_k}^{j_2}}\bigg|
	\overset{\eqref{Gleichung4.85}}{=}\bigg|\frac{1}{N}\sum_{k=1}^\kappa\sum_{n=1}^{m_k-l_k}f(T^nx)\overline{f(T^n\xi_k)}\bigg|.
\end{equation} 
Nach Definition der Menge $C$ in Schritt II und wegen $\xi_k\in C, \ k=1, \dots, \kappa$ sowie wegen $(L_{j_1}, M_{j_1})\ni m_k-l_k>L_{j_1}>L_1\overset{\eqref{Gleichung2.33}}{>}{N_\delta}$ ist für alle $k=1, \dots, \kappa$
\begin{equation} \label{Gleichung4.44d}\bigg|\sum_{n=1}^{m_k-l_k}f(T^nx)\overline{f(T^n\xi_k)}\bigg|\overset{\eqref{Gleichung2.29b}}{<}(m_k-l_k)\cdot \delta\end{equation} erfüllt.  Zusammen mit \eqref{Gleichung4.84} und \eqref{Gleichung4.86} erhalten wir
\begin{equation*}
\bigg|\frac{1}{N} \sum_{n=1}^N c_n^{j_1}\overline{c_n^{j_2}}\bigg|=\bigg|\frac{1}{N}\sum_{k=1}^\kappa \sum_{n=1}^{m_k-l_k}f(T^nx)\overline{f(T^n\xi_k)}\bigg|\overset{\eqref{Gleichung4.44d}}{<}\frac{\sum\limits_{k=1}^\kappa (m_k-l_k)}{N}\cdot\delta\leq\delta.\qedhere
\end{equation*}\end{proof}

\begin{proof}[\textbf{Beweis für $(\beta)$:}]
	\noindent	Seien $j\in \{1, \dots, J\}$ beliebig sowie $\mathcal{RI}\in \{\operatorname{Re}, \operatorname{Im}, -\operatorname{Re}, -\operatorname{Im}\}$ wie in Schritt I gewählt. Weiterhin seien die Basisintervalle der j'ten Schicht gegeben durch \begin{equation} \label{Gleichung4.1001}(l_1, l_1+n_j(S^{l_1}y)], (l_2,l_2+n_j(S^{l_2}y)], \dots,(l_\kappa, l_\kappa + n_j(S^{l_\kappa}y)].\end{equation} \noindent Da außerhalb der Basisintervalle der $j$'ten Schicht, $c_n^j=0$ erfüllt ist, gilt

	\begin{equation}
	\label{Gleichung4.88}
	\begin{split}
	\mathcal{RI}\Big(\frac{1}{N}\sum_{n=1}^N c_n^jg(S^ny)\Big)
	&=\mathcal{RI}\Big(\frac{1}{N}\sum_{k=1}^\kappa \sum_{n=l_k+1}^{l_k+n_j(S^{l_k}y)}c_n^jg(S^ny)\Big)\\
	&=\mathcal{RI}\Big(\frac{1}{N}\sum_{k=1}^\kappa \sum_{n=1}^{n_j(S^{l_k}y)}c_{n+l_k}^j\,g(S^{n+l_k}y)\Big)\\
	&\mspace{-14mu}\overset{\text{Def. der}}{\underset{c_n^j}{=}}\mathcal{RI}\Big(\frac{1}{N}\sum_{k=1}^\kappa \sum_{n=1}^{n_j(S^{l_k}y)}f(T^nx)\,g\big(S^n(S^{l_k}y)\big)\Big).
	\end{split}
	\end{equation}
	\noindent Nach Definition der Basisintervalle der $j$'ten Schicht ist $S^{l_k}y\in B$, $k=1, \dots, \kappa$ erfüllt. 
	\noindent Nach Wahl von $n_j(S^{l_k}y)$ in Schritt I/Behauptung \ref{Behauptung2.6} 
	\noindent gilt für alle $k\in \{1, \dots, \kappa \}$
	\begin{equation} 
	\label{Gleichung4.87}\mathcal{RI}\Big( \sum_{n=1}^{n_j(S^{l_k}y)} f(T^nx)g\big(S^n(S^{l_k}y)\big)\Big) >an_j(S^{l_k}y).
	\end{equation} 
	Folglich ist in \eqref{Gleichung4.88}
	\begin{equation}
	\begin{split}
	\mathcal{RI}\Big(\frac{1}{N}\sum_{n=1}^N c_n^jg(S^ny)\Big)
	&\hspace{0.2cm}=\mathcal{RI}\Big(\frac{1}{N}\sum_{k=1}^\kappa \sum_{n=1}^{n_j(S^{l_k}y)}f(T^nx)\,g\big(S^n(S^{l_k}y)\big)\Big)\\
	&\overset{\eqref{Gleichung4.87}}{>}a \frac{\sum\limits_{k=1}^\kappa n_j(S^{l_k}y)}{N}\\
	&\overset{\eqref{Gleichung4.1001}}{=}ap_j\\
	&\hspace{0.2cm}>ap_1\\
	&\overset{\eqref{Gleichung2.55b}}{>}a\Big(1-\frac{J\cdot (J+1)}{2}\delta'\Big)\\
	&\overset{\eqref{Gleichung2.34c}}{>}a-\delta
	\end{split}
	\end{equation}
	
\noindent 	erfüllt.
\end{proof}
\subsubsection{Der Widerspruch}
Sei die Folge $\big(c_n\big)_{n=1}^N$ definiert durch \begin{equation} \label{Gleichung4.444} c_n:=\sum\limits_{j=1}^J c_n^j.\end{equation} Dann impliziert $(\alpha)$
\begin{equation}
\label{Gleichung2.58}
\begin{split}
\frac{1}{N}\sum_{n=1}^N\big|c_n\big|^2
&=\frac{1}{N}\sum_{n=1}^{N}\Big|\sum_{j=1}^{J}c_n^j\Big|^2\\ 
&=\frac{1}{N}\sum_{n=1}^{N}\bigg(\underbrace{\sum_{j_1=1}^{J}\big|c_n^{j_1} \big|^2}_{\mathclap{\overset{(\ref{Gleichung1000})}{\leq}J||f||_{L^\infty(X, \mu)}^2}} + \sum_{j_1\neq j_2}c_n^{j_1}\overline{c_n^{j_2}}\bigg)\\
&\leq \frac{1}{N}\cdot NJ||f||_{L^\infty(X, \mu)}^2+\sum_{j_1\neq j_2}\Big|\underbrace{\frac{1}{N}\sum_{n=1}^{N}c_n^{j_1}\overline{c_n^{j_2}}}_{\overset{(\alpha)}{<}\delta}\Big|\\
&<J||f||_{L^\infty(X, \mu)}^2 + J^2\delta\\
&\leq \bigg(\sqrt{J}||f||_{L^\infty(X, \mu)}+J\sqrt{\delta}\bigg)^2.
\end{split}
\end{equation}

\noindent Aufgrund der Linearität von $\mathcal{RI}$ folgt aus $(\beta)$\vspace{-0.15cm}
\begin{equation}
\label{Gleichung2.61}
\begin{split}
\bigg|\frac{1}{N}\sum_{n=1}^N c_ng(S^ny)\bigg|
&\geq \bigg|\mathcal{RI}\bigg(\frac{1}{N}\sum_{n=1}^N c_ng(S^ny)\bigg)\bigg|\\
&\hspace{-0.23cm}\overset{\eqref{Gleichung4.444}}{=}\bigg|\mathcal{RI}\bigg(\frac{1}{N}\sum_{n=1}^N\sum_{j=1}^J c_n^jg(S^ny)\bigg)\bigg|\\
&=\Big|\sum_{j=1}^J \underbrace{\mathcal{RI}\Big(\frac{1}{N}\sum_{n=1}^N c_n^jg(S^ny)\Big)}_{\overset{(\beta)}{>}a-\delta}\Big|\\
&>J(a-\delta).
\end{split}
\end{equation}
\noindent Mithilfe der Cauchy-Schwarzschen Ungleichung erhalten wir\vspace{-0.15cm}
\begin{equation}
\label{Gleichung2.59}
\begin{split}
J(a-\delta)&\overset{(\ref{Gleichung2.61})}{<}\bigg|\frac{1}{N}\sum_{n=1}^N c_ng(S^ny)\bigg|\\
&\hspace{0.2cm}\leq\Big(\frac{1}{N}\sum_{n=1}^N \big|c_n\big|^2\Big)^{\frac{1}{2}}\cdot\Big(\frac{1}{N}\sum_{n=1}^{N}\Big|g(S^ny)\Big|^2\Big)^{\frac{1}{2}}\\
&\overset{\eqref{Gleichung2.58}}{<}\Big(\sqrt{J}||f||_{L^\infty(X, \mu)}+J\sqrt{\delta}\Big)||g||_{L^\infty(Y, \nu)}.
\end{split}
\end{equation}
Nach Wahl von $J$ in $(\ref{Gleichung2.34d})$ gilt\vspace{-0.15cm}
$$\frac{||f||_{L^\infty(X,\mu)}||g||_{L^\infty(Y,\nu)}}{\sqrt{J}}<\frac{a}{2}$$ und folglich\vspace{-0.15cm}
\begin{equation}
\label{Gleichung2.58b}
\frac{a}{2}<a-\frac{||f||_{L^\infty(X, \mu)}||g||_{L^\infty(Y, \nu)}}{\sqrt{J}}.
\end{equation}
\noindent Umstellen von \eqref{Gleichung2.59} liefert weiterhin\vspace{-0.07cm}
\begin{equation} 
\label{Gleichung4.95}
a<\frac{\big(||f||_{L^\infty(X,\mu)}+\sqrt{\delta}\sqrt{J}\big)||g||_{L^\infty(Y,\nu)}}{\sqrt{J}}+\delta.
\end{equation} 
\noindent Zusammen mit \eqref{Gleichung2.58b} erhalten wir 
 \begin{equation}
 \label{Gleichung4.96} 
 \begin{split} \frac{a}{2}&< a-\frac{||f||_{L^\infty(X, \mu)}||g||_{L^\infty(Y, \nu)}}{\sqrt{J}}\\
 &\hspace{-0.23cm}\overset{\eqref{Gleichung4.95}}{<}\frac{\big(||f||_{L^\infty(X,\mu)}+\sqrt{\delta}\sqrt{J}\big)||g||_{L^\infty(Y,\nu)}}{\sqrt{J}}+\delta-\frac{||f||_{L^\infty(X, \mu)}||g||_{L^\infty(Y, \nu)}}{\sqrt{J}}\\
 &=\delta+\sqrt{\delta}||g||_{L^\infty(Y, \nu)}.
 \end{split} 
 \end{equation} 

\noindent Da \eqref{Gleichung4.96} nach Wahl von $\delta$ in $(\ref{Gleichung2.31b})$ für alle $0<\delta<\min(a,c)$ gilt, folgt $$a\leq 0.$$ Dies steht im Widerspruch zur Definition von $$a=\frac{1}{k_0}>0$$ in \eqref{Gleichung2.22b}/Schritt I. Die zu Beginn des Beweises getroffene Annahme
$$\nu\bigg(\Big\{y\in Y\bigg|\limsup\limits_{N\to\infty}\Big|\frac{1}{N}\sum_{n=1}^N f(T^nx)g(S^ny) \Big|>0 \Big.\Big\}\bigg)>0$$ ist somit falsch und damit ist der Satz $\ref{Satz2.5}$ bewiesen.
\end{proof}
\section{Beweis für $f\in \mathcal{K}^\bot\cap L^\infty(X,\mu) $ mit beliebigem Bild}
\label{unendlich}
In diesem Abschnitt zeigen wir die folgende Verallgemeinerung von Satz \ref{Satz2.5}:
\begin{Satz}
	\label{Satz2.6}
	Seien $(X, \mathfrak{A}, \mu, T)$  und $(Y, \mathfrak{B}, \nu, S)$ invertierbare, ergodische maßtheoretische dynamische Systeme.  Weiterhin seien $f\in \mathcal{K}^\bot\cap L^\infty(X, \mu)$ und $g \in L^\infty(Y, \nu)$ beliebig gewählt. Dann existiert eine Menge $X'\in \mathfrak{A}$ mit $\mu(X')=1$, welche unabhängig vom System $(Y,\mathfrak{B},\nu, S)$ und $g\in L^\infty(Y,\nu)$ gewählt werden kann, sodass für alle $x\in X'$
	$$\lim_{N\to\infty}\frac{1}{N}\sum_{n=1}^N f(T^nx)g(S^ny)= 0 \text{ für }\nu\text{-fast alle } y\in Y$$ erfüllt ist.
\end{Satz}

\noindent In der Beweisführung orientieren wir uns an \cite[S.92-99]{Assani2003}. Die Grundidee ist, $f\in \mathcal{K}^\bot \cap L^\infty(X, \mu)$ durch eine Folge $\big(f_k\big)_{k=1}^\infty\subset L^\infty(X, \mu)$
von einfachen Funktionen in $L^\infty(X,\mu)$ zu approximieren. Anschließend beweisen wir für bestimmte $k\in \mathbb{N}$ die Behauptungen \ref{Behauptung2.6} und \ref{Behauptung3.21} mit $f_k$ anstelle von $f$. In Schritt III fixieren wir schließlich ein $k\in \mathbb{N}$ und leiten analog zu Schritt III im vorausgegangenen Kapitel \ref{endlich} einen Widerspruch her.

\begin{proof}[\textbf{Beweis:}]
Wir beginnen mit der Konstruktion der Menge $X'$ aus Satz \ref{Satz2.6}. Dazu definieren wir wie in Abschnitt \ref{Vorbetrachtungen2} zunächst die Menge $X_1\in \mathfrak{A}$ als
$$X_1:= \bigg\{x\in X\Big| \lim_{N\to \infty}\frac{1}{N} \sum_{n=1}^N f(T^nx)\overline{f(T^n\xi)}=0 \ \text{für}\ \mu \text{-fast alle } \xi\in X \bigg\}.$$ Nach Lemma \ref{Behauptung3.1} gilt 
$$\mu(X_1)=1.$$
Zur Konstruktion der Menge $X_2$ benötigen wir das Lemma \ref{X2}. Dieses ist jedoch nur auf einfache Funktionen anwendbar. Deshalb fixieren wir eine Folge $\big(f_k\big)_{k=1}^\infty \subset L^\infty(X,\mu)$ von einfachen Funktionen, für welche \begin{equation} \label{Gleichung4.81c} \lim_{k\to\infty}||f-f_k||_{L^\infty(X,\mu)}=0 \end{equation} erfüllt ist. Die Existenz einer solchen Folge ist durch Behauptung \ref{Behauptung2.10} gesichert.\\
\noindent Wir definieren für $k\in \mathbb{N}$
\begin{equation*}
%\label{Gleichung4.80b} 
\begin{split} 
\mathfrak{A}\ni X_2^k:=&\Bigg\{x\in X\bigg|\lim\limits_{N\to\infty}\frac{1}{N}\sum_{j=1}^N\big(\mathbbm{1}_A \circ\Phi_n^k\big)(T^jx)=\mu\big(\big\{\xi\in X\big|\ \big(f_k(T\xi), \dots, f_k(T^{n}\xi)\big)\in A \big\}\big)\\
&\text{ für alle }A\subseteq f_k(X)^n \text{ und alle }n\in \mathbb{N}\Bigg\},
\end{split}\end{equation*}
wobei $\Phi_n^k$ die Abbildung $\Phi_n^k: x\mapsto \big(f_k(x), f_k(Tx), \dots, f_k(T^{n-1}x)\big)$ bezeichne. Durch Anwendung des Lemmas \ref{X2} auf die Funktion $f_k\in L^\infty(X,\mu)$, $k\in \mathbb{N}$ erhalten wir $$\mu(X_2^k)=1\ \text{ für alle } k\in \mathbb{N}.$$ 
Für die Menge \begin{equation} 
\label{Gleichung4.81b}
\mathfrak{A}\ni X_2:=\bigcap_{k\in \mathbb{N}} X_2^k.
\end{equation}
gilt folglich $$\mu(X_2)=1.$$
Wir definieren die Menge $X'\in \mathfrak{A}$ mit $\mu(X')=1$ durch 
\begin{equation} 
\label{Gleichung4.83}
X'=X_1\cap X_2.
\end{equation}

\noindent Für alle $k\in \mathbb{N}$  ist
$|f-f_k|(x)\leq ||f-f_k||_{L^\infty(X,\mu)}$ sowie $|f_k|(x)\leq ||f_k||_{L^\infty(X,\mu)}$ für $\mu$"=fast alle $x\in X$ erfüllt. Da $T$ maßerhaltend ist und der abzählbare Schnitt von Mengen mit Maß $1$ wieder volles Maß hat, können wir ohne Beschränkung der Allgemeinheit für alle $x\in X'$ zudem
\begin{equation} 
\label{Gleichung4.83b}
|f-f_k|(T^nx)\leq ||f-f_k||_{L^\infty(X,\mu)} \text{ und } |f_k|(T^nx) \leq ||f_k||_{L^\infty(X,\mu)}\text{ für alle }k,n\in \mathbb{N}
\end{equation}
annehmen.\\
\noindent Nach Konstruktion der Menge $X'$ ist diese unabhängig vom System $(Y,\mathfrak{B},\nu,S)$ und der Funktion $g\in L^\infty(Y,\nu)$.\\
\noindent  Um den Satz \ref{Satz2.6} zu beweisen, genügt es zu zeigen, dass für alle $x\in X'$\vspace{-0.25cm}
 $$\lim_{N\to\infty} \frac{1}{N}\sum_{n=1}^N f(T^nx)g(S^ny) = 0 \text{ für }\nu"=\text{fast alle }y\in Y$$\vspace{-0.4cm}\\ erfüllt ist. Dies beweisen wir mithilfe eines Widerspruchsbeweises:\\
 \noindent Angenommen es existiert ein $x\in X'$, sodass für die Menge \vspace{-0.2cm}
	\begin{equation*}
	B^*:=\Big\{y\in Y\bigg|\limsup\limits_{N\to\infty}\Big|\frac{1}{N}\sum_{n=1}^N f(T^nx)g(S^ny) \Big|>0 \Big.\Big\}
	\end{equation*}
	gilt: 
	\begin{equation} 
	\label{Gleichung2.63b}
	\nu(B^*)>0.
	\end{equation}
	\noindent Wir fixieren im gesamten Beweis dieses $x\in X'$.

	\subsection{Schritt I: Konstruktion von "`schlechten"' Intervallen}
\begin{Behauptung}
	\label{Behauptung15} 
	Es existieren Konstanten $0<a,c\in \mathbb{R}$ und $N_0\in \mathbb{N}$, sodass für ein beliebiges $J\in \mathbb{N}$ sowie beliebige $\mathfrak{k}_1, \mathfrak{k}_2, \dots, \mathfrak{k}_J\in \mathbb{N}$ eine Menge $B\in \mathfrak{B}$ mit $B\subseteq B^*$ und $\nu(B)>c$, eine messbare Abbildung $\mathcal{RI}\in\{ \operatorname{Re}, \operatorname{Im}, -\operatorname{Re}, -\operatorname{Im}\}$ und eine Folge von Intervallen $\big((L_j, M_j)\big)_{j=1}^J$ aus $\mathbb{N}$ mit $L_j>\mathfrak{k}_j,\ j=1, \dots, J$ existieren, sodass gilt: Für jedes $y\in B$ und $j\in \{1, \dots, J\}$ existiert ein $n_j=n_j(y)\in (L_j, M_j)$ derart, dass für alle $k>N_0$
	\begin{equation*}
	\mathcal{RI}\Big( \sum_{n=1}^{n_j} f_k(T^nx)g(S^ny)\Big) >an_j
	\end{equation*} erfüllt ist. Dabei kann für $j\in \{1, \dots, J-1\}$ das Intervall $(L_j, M_j)$ unabhängig von den Zahlen $\mathfrak{k}_{j+1}, \dots, \mathfrak{k}_J$ konstruiert werden.
\end{Behauptung}
\begin{Bemerkung}
Die Behauptung besagt also gerade, dass die Behauptung \ref{Behauptung2.6} für $k>N_0$ auf die Funktion $f_k$ anstelle von $f$ anwendbar ist.
\end{Bemerkung}

\begin{proof}[\textbf{Beweis:}]
\noindent Seien $J\in \mathbb{N}$ sowie $\mathfrak{k}_1, \mathfrak{k}_2, \dots, \mathfrak{k}_J\in \mathbb{N}$ beliebig gewählt. Wegen Bemerkung \ref{Bemerkung4.1} können wir die Behauptung \ref{Behauptung2.6} auf die nicht notwendig einfache Funktion $f\in \mathcal{K}^\bot \cap L^\infty(X,\mu)$ anwenden. Folglich existieren $0< \tilde{a},c\in \mathbb{R}$, welche unabhängig von $J$ und $\mathfrak{k}_1, \mathfrak{k}_2, \dots, \mathfrak{k}_J$ gewählt werden können, eine Menge $B\in \mathfrak{B}$ mit $B\subseteq B^*$ und $\nu(B)>c$, eine Abbildung $\mathcal{RI}\in\{ \operatorname{Re}, \operatorname{Im}, -\operatorname{Re}, -\operatorname{Im}\}$ und eine Folge von Intervallen  $\big((L_j, M_j)\big)_{j=1}^J$ aus $\mathbb{N}$ mit $L_j>\mathfrak{k}_j,\ j=1, \dots, J$, sodass gilt: Für jedes $y\in B$ und $j\in \{1, \dots, J\}$ existiert ein $n_j=n_j(y)\in (L_j, M_j)$ derart, dass 
\begin{equation}
\label{Gleichung2.64}
\mathcal{RI}\Big( \sum_{n=1}^{n_j} f(T^nx)g(S^ny)\Big) >\tilde{a}n_j
\end{equation}
\noindent erfüllt ist. Da $S$ maßerhaltend ist und der Schnitt abzählbar vieler Mengen mit Maß $1$ wieder Maß 1 hat, können wir ohne Beschränkung der Allgemeinheit annehmen, dass für alle $y\in B$ \begin{equation} \label{Gleichung4.86b} |g(S^ny)|\leq ||g||_{L^\infty(Y,\nu)} \text{ für alle } n\in \mathbb{N} \end{equation} erfüllt ist. Weiterhin ist für ein beliebiges $k\in \mathbb{N}$
\begin{equation}
\label{Gleichung2.65b}
\begin{split}
\mathcal{RI}\Big(\sum_{n=1}^{n_j}\big(f_k-f\big)(T^nx)g(S^ny)\Big)
&\geq - \bigg|\mathcal{RI}\Big(\sum_{n=1}^{n_j}\big(f_k - f\big)(T^nx)g(S^ny)\Big)\bigg|\\
&\geq -\sum_{n=1}^{n_j}\big|f_k - f\big|(T^nx)|g|(S^ny)\\
&\hspace{-0.2cm}\overset{\eqref{Gleichung4.83b}}{\underset{\eqref{Gleichung4.86b}}{\geq}} -n_j||f_k-f||_{L^\infty(X,\mu)}||g||_{L^\infty(Y, \nu)}
\end{split}
\end{equation}
erfüllt.  Nach Wahl der Folge $\big(f_k\big)_{k=1}^\infty$ in \eqref{Gleichung4.81c} existiert ein $N_0\in \mathbb{N}$, sodass für alle $k>N_0$
 \begin{equation}
 \label{Gleichung2.66}
 ||f_k-f||_{L^\infty(X, \mu)}\cdot ||g||_{L^\infty(Y,\nu)}<\frac{\tilde{a}}{2}
 \end{equation}
erfüllt ist. Insbesondere kann $N_0$ unabhängig von $J$ und $\mathfrak{k}_1, \dots, \mathfrak{k}_J$ gewählt werden.
\noindent Mithilfe der Linearität von $\mathcal{RI}$ erhalten wir für alle $k>N_0$
\begin{align*}
\mathcal{RI}\Big(\sum_{n=1}^{n_j} f_k(T^nx)g(S^ny)\Big)
&=\mathcal{RI}\Big(\sum_{n=1}^{n_j}\big(f_k-f\big)(T^nx)g(S^ny)\Big)+\mathcal{RI}\Big(\sum_{n=1}^{n_j}f(T^nx)g(S^ny)\Big)\\
&\hspace{-0.2cm}\overset{(\ref{Gleichung2.64})}{\underset{(\ref{Gleichung2.65b})}{>}} \big(\tilde{a}-||f_k-f||_{L^\infty(X,\mu)}\cdot||g||_{L^\infty(Y,\nu)} \big)\cdot n_j\\
&\hspace{-0.2cm}\overset{\eqref{Gleichung2.66}}{>}\frac{\tilde{a}}{2}n_j.
\end{align*}
\noindent Setze \begin{equation} \label{Gleichung4.12c}\mathbb{R}\ni a:=\frac{\tilde{a}}{2}>0.\end{equation}
Damit ist die Behauptung bewiesen.
\end{proof}
\subsection{Schritt II: "`Gute"' Intervalle}
\label{SchrittIVb}
 Zu Beginn des Beweises von Satz \ref{Satz2.6} haben wir ein $x\in X'$ fixiert. Nach Definition der Menge $X'$ in \eqref{Gleichung4.83} ist also insbesondere $x\in X_1$ erfüllt. Folglich gilt für $\mu$"=fast alle $\xi\in X$ 
\begin{equation}
\label{Gleichung2.74}
\lim\limits_{N\to\infty} \frac{1}{N}\sum_{n=1}^N f(T^nx)\overline{ f(T^n\xi)}=0.
\end{equation}
Für beliebiges $\delta >0$ und $\delta''>0$ existiert deshalb nach Satz \ref{Jegorow} eine Menge $C \in \mathfrak{A}$ mit $\mu(C)>1-\delta''$, sodass  der Grenzwert in $(\ref{Gleichung2.74})$ auf $C$ gleichmäßig in $\xi$ ist. Folglich existiert ein $N_{\delta}\in \mathbb{N}$ derart, dass für alle $N>N_\delta$ und alle $\xi\in C$ \vspace{-0.15cm} 
\begin{equation}
\label{Gleichung2.68}
\bigg|\frac{1}{N}\sum_{n=1}^N f(T^nx)\overline{ f(T^n\xi)}\bigg|<\frac{\delta}{4}
\end{equation}
erfüllt ist. Da nach Wahl der Folge $\big(f_k\big)_{k\in\mathbb{N}}$\ \ $\lim\limits_{k\to\infty}||f-f_k||_{L^\infty(X, \mu)}=0$ gilt,  existiert weiterhin ein $\tilde{N}_\delta \in \mathbb{N}$ derart, dass für alle $k>\tilde{N}_\delta$ \vspace{-0.15cm} 
\begin{equation}\label{Gleichung2.71}
||f-f_k||_{L^\infty(X,\mu)}||f||_{L^\infty(X, \mu)}<\frac{\delta}{4}
\end{equation}
\noindent sowie \vspace{-0.15cm} 
\begin{equation}
\label{Gleichung2.72}
||f-f_k||^2_{L^\infty(X, \mu)}<\frac{\delta}{4}
\end{equation}
erfüllt ist. Ohne Beschränkung der Allgemeinheit können wir für die Menge $C\in \mathfrak{A}$ mit $\mu(C)>1-\delta''$ annehmen, dass für alle $\xi\in C$ und alle  $k,n\in \mathbb{N}$ gilt: $|f_k(T^n\xi)|\leq ||f_k||_{L^\infty(X,\mu)}$ sowie $|\big(f-f_k\big)(T^n\xi)|\leq ||f-f_k||_{L^\infty(X,\mu)}$. Damit erhalten wir für alle $N>N_\delta, k>\tilde{N}_\delta$ und $\xi \in C$ die folgende Ungleichung:\vspace{-0.15cm} 
\begin{align}
\label{Gleichung2.73}
&\mspace{22mu}\bigg|\frac{1}{N}\sum_{n=1}^N f_k(T^nx)\overline{f_k(T^n\xi)}\bigg|\nonumber\\
	&\leq \frac{1}{N}\sum_{n=1}^N\ |f_k-f|(T^nx)\vphantom{\bigcap}|f_k |(T^n \xi)+\frac{1}{N}\sum_{n=1}^N|f|(T^nx)|f_k-f|(T^n\xi)+\bigg|\frac{1}{N}\sum_{n=1}^N f(T^nx)\overline{f(T^n\xi)}\bigg|\nonumber\\
	&\hspace{-0.2cm}\overset{\eqref{Gleichung4.83b}}{\underset{\eqref{Gleichung2.68}}{<}}||f_k-f||_{L^\infty(X, \mu)}\big(||f_k-f||_{L^\infty(X, \mu)} + ||f||_{L^\infty(X, \mu)}\big)+||f||_{L^\infty(X, \mu)}||f_k-f||_{L^\infty(X, \mu)}+\frac{\delta}{4}\nonumber\\
	&\hspace{-0.2cm}\overset{\eqref{Gleichung2.71}}{\underset{\eqref{Gleichung2.72}}{<}}\delta.
\end{align}\vspace{-1cm}\\

\noindent Analog zu Abschnitt $\ref{SchrittIV}$ sei die Menge $B_{n,k}^*\subseteq f_k(X)^n$ für $n\in \mathbb{N}$ und $k>\tilde{N}_\delta$ definiert als\vspace{-0.15cm} 
\begin{equation} 
\label{Gleichung4.113}
B_{n,k}^*:=\Big\{\big(f_k(T\xi), f_k(T^2\xi), \dots, f_k(T^n\xi)\big)\Big| \xi\in C
\Big\}.\end{equation}
\noindent Mithilfe dieser Bezeichnungen, welche wir auch in Schritt III nutzen werden, können wir die folgende Behauptung formulieren:
\begin{Behauptung} 
	\label{Behauptung14}
Sei $(L, M)\subset \mathbb{N}$ ein beliebiges nichtleeres Intervall sowie $k>\tilde{N}_{\delta}$. Dann existiert ein $N_k(M)\in \mathbb{N}$ derart, dass für alle $N>N_k(M)$\vspace{-0.15cm} 
\begin{equation*}
\frac{1}{N}\cdot \# \Big\{j\in \{1, \dots, N\}\big| \ \big(f_k(T^jx), f_k(T^{j+1}x), \dots, f_k(T^{j+n-1}x)\big)\in B_{n,k}^*\, \forall n\in (L,M)\Big\}>1-\delta''
\end{equation*}
erfüllt ist.
\end{Behauptung}
\begin{Bemerkung}
Die Behauptung besagt, dass die Behauptung \ref{Behauptung3.21} für $k>\tilde{N}_\delta$ gültig ist für $f_k$ anstelle von $f$ sowie $B_{n,k}^*$ anstelle von $B_n^*$.
\end{Bemerkung}
\begin{proof}[\textbf{Beweis:}]
Zu Beginn des Beweises von Satz \ref{Satz2.6} haben wir ein $x\in X'$ fixiert. Nach Definition der Menge $X'$ in \eqref{Gleichung4.83} ist $x\in X_2\overset{\eqref{Gleichung4.81b}}{=}\bigcap\limits_{k\in \mathbb{N}} X_2^k$ und folglich insbesondere\vspace{-0.2cm}  $$x\in X_2^k$$\vspace{-0.7cm} \\ erfüllt. Unter Zuhilfenahme der Vorüberlegungen \eqref{Gleichung2.74} bis \eqref{Gleichung4.113} können wir den Beweis nun genauso wie im Beweis zu Behauptung \ref{Behauptung3.21} führen.
\end{proof}

\subsection{Schritt III: Der Widerspruch}
Aufbauend auf dem Abschnitt \ref{SchrittV} sowie den Behauptungen \ref{Behauptung15} und \ref{Behauptung14} beweisen wir die Behauptung \ref{Behauptung100}. Dazu benutzen wir die bereits eingeführten Bezeichnungen. Für $N_0$ aus Behauptung \ref{Behauptung15} und $\tilde{N}_\delta$ aus Schritt II \big(\eqref{Gleichung2.71}, \eqref{Gleichung2.72}\big) setzen wir\vspace{-0.15cm} 
\begin{equation}
\label{Gleichung4.115}
k:=\max\big(N_0, \tilde{N}_\delta\big)+1
\end{equation}
\noindent und leiten analog zu Schritt III in Kapitel \ref{endlich} einen Widerspruch für die Funktion $f_k$ anstelle der Funktion $f$ her. Dazu wählen wir für $a$ und $c$ aus Behauptung \ref{Behauptung15}\vspace{-0.15cm} 
\begin{equation}
\label{Gleichung2.77}
\mathbb{N}\ni J>\frac{4\big(||f||_{L^\infty(X, \mu)} +||f-f_k||_{L^\infty(X,\mu)}\big)^2||g||_{L^\infty(Y, \nu)}^2}{a^2}.
\end{equation}
und fixieren ein beliebiges \vspace{-0.15cm} 
\begin{equation}
\label{Gleichung2.76}
0<\delta<\min(a,c).
\end{equation}

\noindent  Alle weiteren Konstanten wählen wir wie in $(\ref{Gleichung2.34c})$ bis $(\ref{Gleichung2.39})$. Dabei sei für $j\in \{1, \dots, J\}$  $N(M_j):=N_k(M_j)$ gegeben durch Behauptung \ref{Behauptung14}. \\
\noindent Wir fixieren nun für den gesamten Beweisschritt ein $y\in B\cap G$ mit $|g(S^ny)|\leq ||g||_{L^\infty(Y,\nu)}, \ n=1, \dots, N$ und zeigen zunächst die folgende Behauptung:
\begin{Behauptung}
	\label{Behauptung100}
	Seien die Voraussetzungen \eqref{Gleichung4.115} bis \eqref{Gleichung2.76} sowie \eqref{Gleichung2.34c} bis \eqref{Gleichung2.39} erfüllt. Dann gilt die Behauptung \ref{Behauptung43} mit (3') anstelle von (3) und (4') anstelle von (4):
	\begin{itemize}
		\item[(3')] Für ein beliebiges Basisintervall $(l,m]$ der $j$'ten Schicht sowie $B_{m-l,\, k}^*$ aus Schritt II/\eqref{Gleichung4.113} gilt die Orthogonalitätsbedingung\vspace{-0.15cm}  $$\big(c_{l+1}^{\tilde{J}}, c_{l+2}^{\tilde{J}}, \dots, c_{m}^{\tilde{J}}\big)\in B_{m-l,\, k}^*\text{ für alle }\tilde{J}\in \{j+1, \dots, J\}.$$
		\item[(4')] Es gilt \vspace{-0.15cm} \begin{equation*}
		c_n^j:=
		\begin{cases}
		f_k(T^{n-l_r^j}x),  &n\in (l_r^j, m_r^j] \text{ für ein } r\in \{1, \dots, \kappa_j\},  \\
		0, &\text{sonst.} 
		\end{cases}
		\end{equation*}
	\end{itemize}
\end{Behauptung}

\begin{Bemerkung}
Die Folgenglieder $c_n^j$ werden in (4') nicht mithilfe der Funktion $f\in \mathcal{K}^\bot\cap L^\infty(X,\mu)$, sondern mithilfe der einfachen Funktion $f_k\in L^\infty(X,\mu)$, $k:=\max(N_0, \tilde{N}_\delta)+1$ definiert. Ein weiterer Unterschied zu Behauptung \ref{Behauptung43} besteht in der leicht veränderten Definition der Orthogonalitätsbedingung in (3') mit den Mengen $B_{n,k}^*$ anstelle der Mengen $B_n^*$. 
\end{Bemerkung}

\begin{proof}[\textbf{Beweis:}]
Die Basisintervalle sowie die Folgen $\big(c_n^j\big)_{n=1}^N, \ j=1, \dots, J$ konstruieren wir wieder in umgekehrter Reihenfolge in $j$. Die Basisintervalle der J'ten Schicht definieren wir genauso wie im Abschnitt $\ref{Abschnitt2.3.5.1}$, die Folge $(c_n^J)_{n=1}^N$ definieren wir wie in Behauptung \ref{Behauptung100}(4').\\
\noindent Die weiteren Schichten sowie die $c_n^j$ mit $j<J$ konstruieren wir induktiv: Seien für $1<\tilde{J}\leq J$ bereits die Basisintervalle der $\tilde{J}$'ten, ($\tilde{J}$+1)'ten, ..., (J-1)'ten und J'ten Schicht sowie die $\big(c_n^j\big)_{n=1}^N,\ j=\tilde{J}, \dots, J$ definiert. Die Basisintervalle der ($\tilde{J}$-1)'ten Schicht konstruieren wir analog zu Abschnitt $\ref{Abschnitt2.3.5.6}$, jedoch unter der leicht veränderten Orthogonalitätsbedingung (3') aus Behauptung \ref{Behauptung100}. Die Folge $\big(c_n^{\tilde{J}-1}\big)_{n=1}^N$ definieren wir wie in Behauptung \ref{Behauptung100}(4').\\
\noindent Die folgende Rechnung motiviert die Begriffsbildung "`Orthogonalität"':\\
\noindent Sei $\big(l, l+n_{\tilde{J}-1}(S^ly)\big]$ ein beliebiges Basisintervall der ($\tilde{J}$-1)'ten Schicht sowie $\tilde{J}-1<j\leq J.$ Wegen der Orthogonalitätsbedingung existiert ein $\xi_j\in C$ derart, dass \begin{equation}
\label{Gleichung2.79}
\big(c_{l+1}^j,c_{l+2}^j, \dots, c_{l+n_{\tilde{J}-1}(S^ly)}^j\big)=\big(f_k(T\xi_j), f_k(T^2\xi_j), \dots, f_k(T^{n_{\tilde{J}-1}(S^ly)}\xi_j)\big)
\end{equation}
\noindent erfüllt ist. Dann gilt wegen $n_{\tilde{J}-1}(S^ly)>L_1\overset{(\ref{Gleichung2.33})}{>}N_\delta$ und $k\overset{\eqref{Gleichung4.115}}{>}\tilde{N}_\delta$:
\begin{equation*}
\begin{split}
\bigg|\frac{1}{n_{\tilde{J}-1}(S^ly)}\sum_{n=l+1}^{l+n_{\tilde{J}-1}(S^ly)} c_n^{\tilde{J}-1}c_n^j\bigg|
&\overset{\text{Def. der }}{\underset{c_n^{\tilde{J}-1}}{=}} \bigg|\frac{1}{n_{\tilde{J}-1}(S^ly)}\sum_{n=1}^{n_{\tilde{J}-1}(S^ly)}f_k(T^nx)c_{n+l}^j\bigg|\\
&\hspace{0.15cm}\overset{(\ref{Gleichung2.79})}{=}\bigg|\frac{1}{n_{\tilde{J}-1}(S^ly)}\sum_{n=1}^{n_{\tilde{J}-1}(S^ly)}f_k(T^nx)f_k(T^n\xi_j)\bigg|\\
&\hspace{0.18cm}\overset{(\ref{Gleichung2.73})}{<}\delta.
\end{split}
\end{equation*}
\noindent Genauso wie im Abschnitt $\ref{Abschnitt2.3.5.6}$ zeigen wir für die Dichte $p_1$ der Basisintervalle der ersten Schicht in $[1, \dots, N]$
$$p_1>1-\frac{J(J+1)}{2}\delta'.$$
Die Behauptung ist damit bewiesen. \end{proof}
\noindent Aufbauend auf Behauptung \ref{Behauptung100} können wir genauso wie im Abschnitt \ref{Abschnitt2.3.5.7} die Eigenschaften $(\alpha)$ und $(\beta)$ zeigen: 
\begin{itemize}
	\item[$(\alpha)$] Für $j_1, j_2 \in \{1, \dots, J\}, j_1\neq j_2$ ist $\Big|\frac{1}{N}\sum_{n=1}^N c_n^{j_1}\overline{c_n^{j_2}}\Big|<\delta$ erfüllt. 
	\item[$(\beta)$]  \begin{samepage}Für $j\in \{1, \dots, J\}$ und $\mathcal{RI}\in \{\operatorname{Re}, \operatorname{Im}, -\operatorname{Re}, -\operatorname{Im}\}$ aus Schritt I gilt \vspace{-0.15cm} $$\mathcal{RI}\big(\frac{1}{N}\sum_{n=1}^N c_n^jg(S^ny)\big)>a-\delta.$$ \end{samepage}
\end{itemize}
\noindent Sei die Folge $\big(c_n\big)_{n=1}^N$ gegeben durch \vspace{-0.3cm} $$c_n:=\sum\limits_{j=1}^J c_n^j.$$
 Für alle $j\in \{1, \dots, J\}$ und $n\in \{1, \dots, N\}$ gilt $c_n^j=0$ oder $c_n^j=f_k(T^mx)$ für ein $m\in \mathbb{N}.$ Folglich ist für alle $j\in \{1, \dots, J\}$ und $n\in \{1, \dots, N\}$\vspace{-0.15cm} 
 \begin{equation}
 \label{Gleichungalpha}
  |c_n^j|\overset{\eqref{Gleichung4.83b}}{\leq}||f_k||_{L^\infty(X, \mu)}\leq ||f_k-f||_{L^\infty(X, \mu)}+ ||f||_{L^\infty(X, \mu)}
  \end{equation} erfüllt. Zusammen mit $(\alpha)$ folgt\vspace{-0.15cm}  \begin{equation}
\label{Gleichung2.81}
\begin{split}
\bigg|\frac{1}{N}\sum_{n=1}^N\big|c_n\big|^2\bigg|
&=\bigg|\frac{1}{N}\sum_{n=1}^{N}\bigg(\sum_{j_1=1}^{J}\big|c_n^{j_1} \big|^2 + \sum_{j_1\neq j_2}c_n^{j_1}\overline{c_n^{j_2}}\bigg)\bigg|\\
&\hspace{-0.3cm}\overset{(\ref{Gleichungalpha})}{\leq}J\Big(||f_k-f||_{L^\infty(X, \mu)}+||f||_{L^\infty(X, \mu)}\Big)^2+\sum_{j_1\neq j_2}\Big|\frac{1}{N}\sum_{n=1}^{N}c_n^{j_1}\overline{c_n^{j_2}}\Big|\\
&\hspace{-0.1cm}\overset{(\alpha)}{<}J\Big(||f_k-f||_{L^\infty(X, \mu)}+||f||_{L^\infty(X, \mu)}\Big)^2+\big(J^2-J\big)\delta\\
&\leq \Big(\sqrt{J}\big(||f_k-f||_{L^\infty(X, \mu)}+||f||_{L^\infty(X, \mu)}\big)+J\sqrt{\delta}\Big)^2.
\end{split}
\end{equation}
Aus $(\beta)$ gewinnen wir wie in $\eqref{Gleichung2.61}$\vspace{-0.15cm} 
\begin{equation}
\label{Gleichung2.82}
\begin{split}
\Big|\frac{1}{N}\sum_{n=1}^N c_ng(S^ny)\Big|
&>J(a-\delta).
\end{split}
\end{equation}

\noindent Es folgt\vspace{-0.21cm} 
\begin{equation}
\label{Gleichung2.83}
\begin{split}
J(a-\delta)&\overset{\eqref{Gleichung2.82}}{<}\Big|\frac{1}{N}\sum_{n=1}^N c_ng(S^ny)\Big|\\
&\hspace{-0.27cm}\overset{\text{Cauchy-}}{\underset{\text{Schwarz-U.}}{\leq}}\bigg(\frac{1}{N}\sum_{n=1}^N \big|c_n\big|^2\bigg)^{\frac{1}{2}}\cdot\bigg(\underbrace{\frac{1}{N}\sum_{n=1}^{N}\big|g(S^ny)\big|^2}_{\mathclap{<||g||_{L^\infty(Y, \nu)}}}\bigg)^{\frac{1}{2}}\\
&\overset{\eqref{Gleichung2.81}}{<}\bigg(\sqrt{J}\Big(||f_k-f||_{L^\infty(X, \mu)}+||f||_{L^\infty(X, \mu)}\Big)+J\sqrt{\delta}\bigg)||g||_{L^\infty(Y, \nu)}.
\end{split}
\end{equation}
\noindent Nach Wahl von $J$ in $\eqref{Gleichung2.77}$ gilt
\begin{equation*}
\label{Gleichung2.84}
\frac{\Big(||f||_{L^\infty(X, \mu)}+||f-f_k||_{L^\infty(X, \mu)}\Big)||g||_{L^\infty(Y, \nu)}}{\sqrt{J}}<\frac{a}{2}\cdot\frac{||f||_{L^\infty(X, \mu)}+||f-f_k||_{L^\infty(X, \mu)}}{||f||_{L^\infty(X, \mu)}+||f-f_k||_{L^\infty(X,\mu)}}=\frac{a}{2}.
\end{equation*}
Deshalb ist 
\begin{equation}
\label{Gleichung2.79b}
\frac{a}{2}<a-\frac{\Big(||f||_{L^\infty(X, \mu)}+||f-f_k||_{L^\infty(X, \mu)}\Big)||g||_{L^\infty(Y, \nu)}}{\sqrt{J}}
\end{equation}
erfüllt. Umstellen von \eqref{Gleichung2.83} nach $a$ liefert
\begin{equation} 
	\label{Gleichung2.80b}
	a<\frac{\big(||f_k-f||_{L^\infty(X,\mu)}+||f||_{L^\infty(X,\mu)}+\sqrt{\delta}\sqrt{J}\big)||g||_{L^\infty(Y,\nu)}}{\sqrt{J}}+\delta.
\end{equation} Zusammen mit \eqref{Gleichung2.79b} erhalten wir
\begin{equation*}
\frac{a}{2}<a-\frac{\Big(||f||_{L^\infty(X, \mu)}+||f-f_k||_{L^\infty(X, \mu)}\Big)||g||_{L^\infty(Y, \nu)}}{\sqrt{J}}\overset{\eqref{Gleichung2.80b}}{<}\sqrt{\delta}||g||_{L^\infty(Y, \nu)}+\delta.
\end{equation*}
 Da wir  $\delta>0$ beliebig klein wählen können, folgt $$\frac{a}{2}\leq 0.$$
\noindent Dies steht im Widerspruch zur Definition von $$a=\frac{\tilde{a}}{2}>0$$ in Schritt I/\eqref{Gleichung4.12c}. Die zu Beginn des Beweises getroffene Annahme 
$$	\nu\bigg(\Big\{y\in Y\bigg|\limsup\limits_{N\to\infty}\Big|\frac{1}{N}\sum_{n=1}^N f(T^nx)g(S^ny) \Big|>0 \Big.\Big\}\bigg)>0$$ \noindent ist deshalb falsch und damit ist Satz $\ref{Satz2.6}$ bewiesen.
\end{proof}

\chapter{Beweis für den allgemeinen Fall}
\label{allgemein}
In den vorangegangenen beiden Kapiteln \ref{Kapitel2} und \ref{Kapitel3} haben wir Bourgains Rückkehrzeitentheorem basierend auf der Darstellung in \cite{Assani2003} jeweils nur unter bestimmten Voraussetzungen bewiesen. Ziel der sich anschließenden Abschnitte \ref{Kapitel4.1} und \ref{Kapitel4.2} ist die Verallgemeinerung der gewonnenen Resultate auf das Rückkehrzeitentheorem von Bourgain:\vspace{0.3cm}\\
\noindent \textbf{Rückkehrzeitentheorem von Bourgain} \textit{
	Seien $(X, \mathfrak{A}, \mu, T)$ ein maßtheoretisches dynamisches System mit rein atomarer invarianter $\sigma$"=Algebra und $f \in L^\infty(X, \mu)$.
	Dann ist die Folge 	 $\big(f\left(T^n x\right)\big)_{n=1}^{\infty}$ für $\mu$-fast alle $x \in X$ 
	ein universell gutes Gewicht für die punktweise Konvergenz von $L^1$"=Funktionen.}\vspace{0.3cm}\\
\noindent Die dafür notwendigen Verallgemeinerungsschritte beruhen auf eigenen Überlegungen. In Kapitel \ref{Kapitel3} haben wir mithilfe eines Widerspruchsbeweises Bourgains Rückkehrzeitentheorem im Spezialfall von Satz \ref{Satz2.6}, das heißt für invertierbare, ergodische maßtheoretische dynamische Systeme $(X,\mathfrak{A}, \mu, T)$ und $(Y,\mathfrak{B}, \nu, S)$ sowie $f\in \mathcal{K}^\bot \cap L^\infty(X,\mu)$ und $g\in L^\infty(Y,\nu)$, gezeigt. Diesen Satz verallgemeinern wir in den Schritten I-IV zunächst auf den Fall  eines beliebigen maßtheoretischen dynamischen Systems $(Y,\mathfrak{B}, \nu, S)$. Anschließend lockern wir in den Schritten V - VII die Anforderungen an das System $(X,\mathfrak{A},\mu,T)$. Die Beweisschritte können wie folgt zusammengefasst werden:
\begin{enumerate}[I]
	\item Verallgemeinerung auf ein (nicht notwendig invertierbares) ergodisches maßtheoretisches dynamisches System $(Y,\mathfrak{B},\nu,S)$;\vspace{-0.1cm} 
	\item Verallgemeinerung auf ein (nicht notwendig ergodisches) maßtheoretisches dynamisches System $(Y,\mathfrak{B},\nu,S)$ auf einem Borel"=Wahrscheinlichkeitsraum;\vspace{-0.1cm}
	\item Verallgemeinerung auf ein beliebiges maßtheoretisches dynamisches System $(Y,\mathfrak{B},\nu,S)$, für welches $L^1(Y,\nu)$ separabel ist;\vspace{-0.1cm}
	\item Verallgemeinerung auf ein beliebiges maßtheoretisches dynamisches System $(Y,\mathfrak{B},\nu,S)$;\vspace{-0.1cm}
	\item Verallgemeinerung auf ein (nicht notwendig invertierbares) ergodisches maßtheoretisches dynamisches System $(X,\mathfrak{A},\mu,T)$ und $g\in L^\infty(Y,\nu)$;\vspace{-0.1cm}
	\item Verallgemeinerung auf ein  maßtheoretisches dynamisches System $(X,\mathfrak{A},\mu,T)$ mit rein atomarer invarianter $\sigma$"=Algebra  und $g\in L^\infty(Y,\nu)$;\vspace{-0.1cm}
	\item Verallgemeinerung auf den Fall $g\in L^1(Y,\nu)$ und damit auf das Rückkehrzeitentheorem von Bourgain.
\end{enumerate}

\section{Beweis für $f\in \mathcal{K}^\bot\cap L^\infty(X,\mu)$ und Gewichte für $L^\infty$-Funktionen}
\label{Kapitel4.1}
Ziel dieses Abschnittes ist es, unter der Annahme $f\in \mathcal{K}^\bot \cap L^\infty(X,\mu)$ und $g\in L^\infty(Y,\nu)$ in den Schritten I-IV die Anforderungen an das System $(Y,\mathfrak{B},\nu,S)$ zu lockern.
\subsection{Schritt I: Invertierbare Erweiterung}
\label{SchrittI}

\noindent Wir zeigen unter Rückgriff auf Satz \ref{Satz2.6} die folgende Proposition:
\begin{Proposition}
	\label{Behauptung1}
	Seien ein invertierbares, ergodisches maßtheoretisches dynamisches System $(X,\mathfrak{A},\mu,T)$ und eine Funktion $f\in \mathcal{K}^\bot\cap L^\infty(X,\mu)$ gegeben. Weiterhin seien ein ergodisches (nicht notwendig invertierbares) maßtheoretisches dynamisches System $(Y,\mathfrak{B},\nu,S)$  sowie eine Funktion $g\in L^\infty(Y,\nu)$ gegeben. Dann existiert eine Menge $X'\in \mathfrak{A}$ mit $\mu(X')=1$, welche unabhängig vom System $(Y,\mathfrak{B},\nu,S)$ und $g$ gewählt werden kann, sodass für alle $x\in X'$\vspace{0.15cm}  $$\lim\limits_{N\to\infty}\frac{1}{N}\sum\limits_{n=1}^N f(T^nx)g(S^ny)=0 \text{ für } \nu\text{-fast alle }y\in Y$$ erfüllt ist.
\end{Proposition}

\begin{proof}[\textbf{Beweis:}]
Es bezeichne $(\tilde{Y}, \tilde{\mathfrak{B}},\tilde{\nu},\tilde{S})$ die in Satz \ref{Satz99} definierte invertierbare Erweiterung des Systems $(Y,\mathfrak{B},\nu,S)$ sowie $\pi: \tilde{Y}\to Y: \ \tilde{y}\mapsto \tilde{y}_0$ die zugehörige Faktorabbildung. Zur Veranschaulichung dient das folgende Diagramm:\vspace{0.1cm} 
\begin{equation*} \begin{tikzcd}
	\tilde{Y} \arrow{d}[swap]{\pi} \arrow{r}{\tilde{S}} & \tilde{Y} \arrow{d}{\pi} \\
	Y\arrow{r}{S} & Y 
\end{tikzcd}\end{equation*}
\noindent Da $(Y,\mathfrak{B},\nu,S)$ ergodisch ist, gilt dies wegen Satz \ref{Satz99} auch für $(\tilde{Y},\tilde{\mathfrak{B}},\tilde{\nu},\tilde{S})$. Definiere\vspace{0.1cm} 
$$\tilde{g}:=g\circ \pi \in L^\infty(\tilde{Y},\tilde{\nu}).$$ Es bezeichne $X'\in \mathfrak{A}$ die in Satz \ref{Satz2.6} unabhängig von $(Y,\mathfrak{B},\nu,S)$ und $g$ gewählte Menge mit $\mu(X')=1$. Aufgrund der Ergodizität und Invertierbarkeit des Systems $(\tilde{Y},\tilde{\mathfrak{B}}, \tilde{\nu},\tilde{S})$ folgt mithilfe dieses Satzes für alle $x\in X'$\vspace{0.1cm} 
	\begin{equation} \label{asdf} \lim\limits_{N\to\infty} \frac{1}{N} \sum\limits_{n=1}^N f(T^nx) \tilde{g}(\tilde{S}^n\tilde{y}) = 0 \text{ für } \tilde{\nu}-\text{fast alle } \tilde{y}\in \tilde{Y}. \end{equation}
		Nach Definition von $(\tilde{Y}, \tilde{\mathfrak{B}}, \tilde{\nu}, \tilde{S})$, $\pi$ und $\tilde{g}$ ergibt sich \vspace{-0.15cm}
	\begin{equation}
	\label{asdf2}
	\begin{split}
	\lim\limits_{N\to\infty} \frac{1}{N} \sum\limits_{n=1}^N f(T^nx) \tilde{g}(\tilde{S}^n\tilde{y}) & = \lim\limits_{N\to\infty} \frac{1}{N} \sum\limits_{n=1}^N f(T^nx) \big(g\circ \pi\big)(\tilde{S}^n\tilde{y})\\
	&=\lim\limits_{N\to\infty} \frac{1}{N} \sum\limits_{n=1}^N f(T^nx) g\big(S^n\big(\pi (\tilde{y})\big)\big).
	\end{split}
	\end{equation}
	Wir fixieren nun ein $x\in X'$. Wegen $\eqref{asdf}$ existiert eine Menge $\tilde{Y'}\in \tilde{\mathfrak{B}}$ mit $\tilde{\nu}(\tilde{Y'})=1$, sodass für alle $\tilde{y}\in \tilde{Y}'$ \vspace{-0.1cm}\begin{equation} \label{Gleichung1c} \lim\limits_{N\to\infty} \frac{1}{N} \sum\limits_{n=1}^N f(T^nx) g\big(S^n\big(\pi (\tilde{y})\big)\big)=0 \end{equation} erfüllt ist.
	Die Funktion $y\mapsto \lim_{N\to\infty} \frac{1}{N}\sum_{n=1}^N f(T^nx)g(S^ny)$ ist als Grenzwert messbarer Funktionen messbar. Folglich ist die Menge \vspace{-0.1cm}$$Y':=\Big\{y\in Y\Big| \lim_{N\to\infty} \frac{1}{N} \sum_{n=1}^N f(T^nx)g(S^ny)=0\Big\}$$ ein Element der $\sigma$"=Algebra $\mathfrak{B}$. Aufgrund der Identität \eqref{asdf2} folgt \vspace{-0.1cm}\begin{equation}
	\label{Gleichung413} \tilde{Y}'\subseteq \pi^{-1}(Y').
	\end{equation}
	\noindent Unter Verwendung der Definition von $\tilde{\nu}$ erhalten wir schließlich\vspace{-0.1cm}
	\begin{equation} \label{Gleichung5.5}\nu(Y')=\tilde{\nu}\big(\pi^{-1}(Y')\big)\overset{\eqref{Gleichung413}}{\geq} \tilde{\nu}(\tilde{Y}')=1.\end{equation}
	Da wir $x\in X'$ beliebig gewählt haben, ist die Proposition damit bewiesen.
\end{proof}

\subsection{Schritt II: Ergodische Zerlegung}
\label{SchrittIIb}
Aufbauend auf Proposition \ref{Behauptung1} und der ergodischen Zerlegung aus Theorem \ref{Zerlegung} beweisen wir die folgende Proposition:
\begin{Proposition}
	\label{Behauptung2}
	Seien ein invertierbares, ergodisches maßtheoretisches dynamisches System $(X,\mathfrak{A},\mu,T)$ und eine Funktion $f\in \mathcal{K}^\bot\cap L^\infty(X,\mu)$ gegeben. Weiterhin seien ein (nicht notwendig ergodisches) maßtheoretisches dynamisches System $(Y,\mathfrak{B},\nu,S)$ auf einem Borel"=Wahrscheinlichkeitsraum $(Y,\mathfrak{B},\nu)$ sowie eine Funktion $g\in L^\infty(Y,\nu)$ gegeben. Dann existiert eine Menge $X'\in \mathfrak{A}$ mit $\mu(X')=1$, welche unabhängig vom System $(Y,\mathfrak{B},\nu,S)$ und $g$ gewählt werden kann, sodass für alle $x\in X'$\vspace{-0.1cm} $$\lim\limits_{N\to\infty}\frac{1}{N}\sum\limits_{n=1}^N f(T^nx)g(S^ny)=0 \text{ für } \nu\text{-fast alle }y\in Y$$ erfüllt ist.
\end{Proposition}
\begin{proof}[\textbf{Beweis:}]
Sei die Menge $X'\in \mathfrak{A}$ mit $\mu(X')=1$ aus Proposition \ref{Behauptung1} gewählt. Weiterhin sei die ergodische Zerlegung von $(Y,\mathfrak{B},\nu)$ gegeben durch den Borel"=Wahrscheinlichkeitsraum $(Z,\mathfrak{B}_Z,\lambda)$ und die messbare Abbildung $Z\to M_1(Y,\mathfrak{B}),\ z\mapsto \nu_z$  aus Theorem \ref{Zerlegung}, wobei $M_1(Y,\mathfrak{B})$ die Menge der Borel"=Wahrscheinlichkeitsmaße auf $(Y,\mathfrak{B})$ bezeichne. Dann ist $\nu_z$ für $\lambda$"=fast alle $z\in Z$ ein ergodisches Wahrscheinlichkeitsmaß auf $(Y,\mathfrak{B},\nu)$.\\
\noindent Da $g$ nach Voraussetzung ein Element aus $L^\infty(Y,\nu)$ ist, existiert eine Menge $\mathcal{N}\in \mathfrak{B}$ mit $\nu(\mathcal{N})=0$, sodass 
\begin{equation} 
\label{Gleichung4.1b}
|g(y)|\leq ||g||_{L^\infty(Y,\nu)} \text{ für alle }y\in Y\setminus \mathcal{N} \end{equation}
erfüllt ist. Wir erhalten mithilfe der Identität \eqref{Gleichung5.1} $$0=\nu(\mathcal{N})=\int_Y \mathbbm{1}_{\mathcal{N}}d\nu=\int_Z\underbrace{\int_Y\mathbbm{1}_{\mathcal{N}}d\nu_z}_{\geq 0}d\lambda(z).$$ \noindent Für $\lambda$"=fast alle $z\in Z$ ist folglich
	$\nu_z(\mathcal{N})=0$ sowie wegen \eqref{Gleichung4.1b} $g\in L^\infty(Y,\nu_z)$ erfüllt. Mithilfe von Proposition \ref{Behauptung1} erhalten wir deshalb für alle $x\in X'$ und $\lambda$"=fast alle $z\in Z$ 
$$\lim\limits_{N\to\infty}\frac{1}{N}\sum_{n=1}^N f(T^nx)g(S^ny)=0 \text{ für } \nu_z\text{-fast alle } y\in Y.$$ Dies liefert für $x\in X'$ und $\lambda$"=fast alle $z\in Z$ $$\limsup\limits_{N\to\infty} \Big|\frac{1}{N}\sum_{n=1}^Nf(T^nx)g(S^ny)\Big|=0 \text{ für }\nu_z\text{-fast alle }y\in Y$$ sowie 
\begin{equation}
\label{Gleichung4.7}
\int\limits_Y \limsup\limits_{N\to\infty} \Big|\frac{1}{N}\sum_{n=1}^Nf(T^nx)g(S^ny)\Big|d\nu_z(y)=0.
\end{equation}
Sei nun $x\in X'$ beliebig gewählt. Dann gilt
\begin{equation*}
\begin{split}
\int\limits_Y\limsup\limits_{N\to\infty}\Big|\frac{1}{N}\sum_{n=1}^Nf(T^nx)g(S^ny)\Big|d\nu(y)
&\overset{\text{Thm.}}{\underset{\ref{Zerlegung}}{=}}\int\limits_Z \underbrace{\int\limits_Y \limsup\limits_{N\to\infty}\Big|\frac{1}{N}\sum_{n=1}^Nf(T^nx)g(S^ny)\Big| d\nu_z(y)}_{\overset{\eqref{Gleichung4.7}}{=} 0 \text{ für } \lambda\text{-fast alle } z\in Z} d\lambda(z)\\
&\hspace{0.2cm}=0.
\end{split}
\end{equation*}
Dies liefert $$\limsup\limits_{N\to\infty}\Big|\frac{1}{N}\sum\limits_{n=1}^N f(T^nx)g(S^ny)\Big|=0 \text{ für } \nu\text{-fast alle } y\in Y.$$ Folglich existiert für $\nu$"=fast alle $y\in Y$ der Grenzwert $\lim\limits_{N\to\infty}\frac{1}{N}\sum\limits_{n=1}^N f(T^nx)g(S^ny)$ und ist gleich $0$.\\
\noindent Da wir $x\in X'$ beliebig gewählt haben, ist die Proposition damit bewiesen.
\end{proof}

\subsection{Schritt III: Metrisches Modell}
Mithilfe von Theorem \ref{MetrischesModell} beweisen wir aufbauend auf Proposition \ref{Behauptung2} die folgende Proposition:
 \begin{Proposition}
 	\label{Behauptung3}
 	Seien ein invertierbares, ergodisches maßtheoretisches dynamisches System $(X,\mathfrak{A},\mu,T)$ und eine Funktion $f\in \mathcal{K}^\bot\cap L^\infty(X,\mu)$ gegeben. Weiterhin seien ein maßtheoretisches dynamisches System $(Y,\mathfrak{B},\nu,S)$, für welches $L^1(Y,\nu)$ ein separabler Banachraum ist, sowie eine Funktion $g\in L^\infty(Y,\nu)$ gegeben. Dann existiert eine Menge $X'\in \mathfrak{A}$ mit $\mu(X')=1$, welche unabhängig vom System $(Y,\mathfrak{B},\nu,S)$ und $g$ gewählt werden kann, sodass für alle $x\in X'$ $$\lim\limits_{N\to\infty}\frac{1}{N}\sum\limits_{n=1}^N f(T^nx)g(S^ny)=0 \text{ für } \nu\text{-fast alle }y\in Y$$ erfüllt ist.	
 \end{Proposition}
\begin{proof}[\textbf{Beweis:}]
 Da $L^1(Y,\nu)$ nach Voraussetzung ein separabler Banachraum ist, existieren ein maßtheoretisches dynamisches System $(K,\mathfrak{B}_K,\lambda,\tilde{S})$ auf einem Borel"=Wahrscheinlichkeitsraum $(K,\mathfrak{B}_K,\lambda)$ sowie ein Markov"=Isomorphismus $$\Phi: L^1(Y,\nu)\to L^1(K,\lambda),$$ für welchen 
 \begin{equation} 
 \label{Gleichung4.13}\tilde{S}\circ \Phi=\Phi \circ S
 \end{equation}
 erfüllt ist. Zur Veranschaulichung dient das folgende Diagramm: \begin{equation*} \begin{tikzcd}
 L^1(Y,\nu) \arrow{d}[swap]{\Phi} \arrow{r}{S} & L^1(Y,\nu) \arrow{d}{\Phi} \\
 L^1(K,\lambda)\arrow{r}{\tilde{S}} & L^1(K,\lambda) 
 \end{tikzcd}
 \end{equation*}
 \noindent Es bezeichne $X'\in \mathfrak{A}$ mit $\mu(X')=1$ die Menge  aus Proposition \ref{Behauptung2}. Wegen $f\in L^\infty(X,\mu)$ können wir ohne Beschränkung der Allgemeinheit annehmen, dass $\text{für alle }x\in X' \text{ und } n\in \mathbb{N}$ \begin{equation}
 \label{Gleichung5.2}
 |f(T^nx)|\leq ||f||_{L^\infty(X,\mu)} 
 \end{equation} erfüllt ist. Sei $x\in X'$ fest gewählt. Definiere für $N\in \mathbb{N}$ Funktionen $h\hspace{-0.03cm}_{N}\in L^\infty(Y,\nu)\subseteq L^1(Y,\nu)$ durch
 \begin{equation}
 \label{Gleichung4.14}
 h\hspace{-0.03cm}_{N}(y):= \frac{1}{N}\sum_{n=1}^Nf(T^nx)g(S^ny).
 \end{equation}

\noindent Das Ziel der nachfolgenden Berechnungen ist zu zeigen, dass für $\nu$"=fast alle $y\in Y$ der Grenzwert $\lim_{N\to\infty}h\hspace{-0.03cm}_{N}(y)$ existiert sowie $\lim_{N\to\infty}h\hspace{-0.03cm}_{N}(y)=0$ erfüllt ist. Für $N\in\mathbb{N}$ betrachten wir dazu zunächst $\Phi h\hspace{-0.03cm}_{N}$. Es gilt
\begin{equation}
\label{Gleichung4.15}
\begin{split}
\Phi h\hspace{-0.03cm}_{N}&=\Phi\Big(\frac{1}{N}\sum_{n=1}^N f(T^nx)S^ng\Big)\\
&\hspace{-0.45cm}\overset{\text{Linearität}}{=}\frac{1}{N}\sum_{n=1}^N f(T^nx)\Phi(S^ng)\\
&\hspace{-0.4cm}\overset{\text{Induktion}}{\underset{\eqref{Gleichung4.13}}{=}} \frac{1}{N}\sum_{n=1}^N f(T^nx)\tilde{S}^n\underbrace{\Phi g}_{=:\tilde{g}}.
\end{split}
\end{equation}
Da $\Phi$ charakteristische Funktionen auf charakteristische Funktionen abbildet (siehe \cite[S.218]{Eisner2014}) und annahmegemäß $g\in L^\infty(Y,\nu)$ sowie weiterhin wegen \eqref{Gleichung4.16} $\Phi\geq 0$ erfüllt ist, folgt \begin{equation} \label{Gleichung5.3} \tilde{g}=\Phi g\in L^\infty(K,\lambda).
\end{equation}
\noindent Da $(K,\mathfrak{B}_K,\lambda,\tilde{S})$ ein maßtheoretisches dynamisches System auf einem Borel"=Wahrscheinlichkeitsraum ist und wir $x\in X'$ fixiert haben, folgt aus Proposition \ref{Behauptung2} für $\lambda$"=fast alle $\tilde{y}\in K$:
\begin{equation}
\label{Gleichung4.17}
0=\lim_{N\to\infty}\frac{1}{N}\sum_{n=1}^N f(T^nx)\tilde{g}(\tilde{S}^n\tilde{y})=\lim\limits_{N\to\infty}\frac{1}{N}\sum_{n=1}^N f(T^nx)\tilde{S}^n\tilde{g}(\tilde{y}).
\end{equation}
\noindent Wegen \eqref{Gleichung4.15} konvergiert deshalb die Folge $\big(\Phi h\hspace{-0.03cm}_{N}\big)_{N=1}^\infty$  punktweise $\lambda$"=fast überall gegen $0$.
Da nach Definition der Funktionen $\big(h\hspace{-0.03cm}_{N}\big)_{N=1}^\infty$ in \eqref{Gleichung4.14} für  alle $N\in \mathbb{N}$ $$||h\hspace{-0.03cm}_{N}||_{L^\infty(Y,\nu)}\leq ||f||_{L^\infty(X,\mu)}\cdot ||g||_{L^\infty(Y,\nu)}$$ erfüllt ist, konvergiert die Folge $\big(h\hspace{-0.03cm}_{N}\big)_{N=1}^\infty=\big(\frac{1}{N}\sum_{n=1}^N f(T^nx)g(S^n(\cdot)\big)_{N=1}^\infty$ nach Behauptung \ref{Behauptung17}  punktweise $\nu$"=fast überall gegen $0$. \noindent Da wir $x\in X'$ beliebig gewählt haben, ist die Proposition \ref{Behauptung3} damit bewiesen.
\end{proof}

\subsection{Schritt IV: Separabilität}
\label{AbschnittIV}
Wir beweisen die folgende Proposition:
\begin{Proposition}
	\label{Behauptung5}
Seien ein beliebiges invertierbares, ergodisches maßtheoretisches dynamisches System $(X,\mathfrak{A},\mu,T)$ und eine Funktion $f\in \mathcal{K}^\bot\cap L^\infty(X,\mu)$ gegeben. Weiterhin seien ein beliebiges maßtheoretisches dynamisches System $(Y,\mathfrak{B},\nu,S)$ und eine Funktion $g\in L^\infty(Y,\nu)$ gegeben. Dann existiert eine Menge $X'\in \mathfrak{A}$ mit $\mu(X')=1$, welche unabhängig vom System $(Y,\mathfrak{B},\nu,S)$ und $g$ gewählt werden kann, sodass für alle $x\in X'$ $$\lim\limits_{N\to\infty}\frac{1}{N}\sum\limits_{n=1}^\infty f(T^nx)g(S^ny)=0 \text{ für } \nu\text{-fast alle }y\in Y$$ erfüllt ist. Das heißt, für ein beliebiges ergodisches, invertierbares maßtheoretisches dynamisches System $(X,\mathfrak{A},\mu,T)$ und $f\in \mathcal{K}^\bot\cap L^\infty(X,\mu)$ sowie $\mu$"=fast alle $x\in X$ ist die Folge $\big(f(T^nx)\big)_{n=1}^N$ ein universell gutes Gewicht für die punktweise Konvergenz von $L^\infty$"=Funktionen.
\end{Proposition}
\begin{proof}[\textbf{Beweis:}]
	Die Beweisidee ist, eine abzählbar erzeugte, $S$"=invariante Unter"=$\sigma$"=Algebra von $\mathfrak{B}$ zu konstruieren, bezüglich welcher die Abbildungen $S^k g,\ k\in \mathbb{N}\cup \{0\}$ messbar sind. Anschließend ersetzen wir $\mathfrak{B}$ durch die Vervollständigung dieser Unter"=$\sigma$"=Algebra und wenden Proposition \ref{Behauptung3} in Verbindung mit Behauptung \ref{Behauptung4} an.\\
	\noindent Ohne Beschränkung der Allgemeinheit können wir annehmen, dass $g$ eine reellwertige Funktion ist. (Falls diese Voraussetzung nicht erfüllt ist, führen wir die nachfolgende Konstruktion zunächst getrennt für den Real"= und Imaginärteil von $g$ durch und vereinigen anschließend die beiden abzählbaren Erzeuger.) Es ist bekannt, dass die Borelsche $\sigma$"=Algebra $\mathfrak{B}_{\mathbb{R}}$ in $\mathbb{R}$ abzählbar erzeugt ist von den offenen Intervallen mit rationalen Endpunkten. Das heißt, es gilt 
	$$\mathfrak{B}_{\mathbb{R}}=\sigma\Big(\{(a,b)\big| \ a,b\in \mathbb{Q}, \ a<b\}\Big).$$ 
Außerdem ist eine Abbildung $h: Y\to \mathbb{R}$ genau dann $\mathfrak{B}$"=$\mathfrak{B}_{\mathbb{R}}$"=messbar, wenn für alle Intervalle $(a,b)\subset \mathbb{R}$ mit $a,b\in \mathbb{Q},a<b$ 
$$h^{-1}\big((a,b)\big)\in \mathfrak{B}$$ erfüllt ist. Wir betrachten nun die Menge 
$$E:=\big\{(S^kg)^{-1}\big((a,b)\big)| \ k\in \mathbb{N}\cup\{0\},\ a,b\in \mathbb{Q}, a<b\big\}.$$
Da die Menge der offenen Intervalle aus $\mathbb{R}$ mit rationalen Endpunkten abzählbar ist, ist auch die Menge $E$ abzählbar. Weiterhin definieren wir $\mathfrak{D}\subseteq \mathfrak{B}$ als die von $E$ erzeugte $\sigma$"=Algebra, das heißt, es gelte $$\mathfrak{D}:=\sigma(E).$$ 
Die $\sigma$"=Algebra $\mathfrak{D}$ ist so konstruiert, dass die Abbildungen $S^kg: Y\to \mathbb{R},\ k\in \mathbb{N}\cup \{0\}$\  $\mathfrak{D}$"=$\mathfrak{B}_{\mathbb{R}}$"=messbar sind. Wir überprüfen außerdem leicht, dass $\mathfrak{D}$ $S$"=invariant ist. Diese Eigenschaften bleiben auch erhalten, wenn wir zur Vervollständigung $\overline{\mathfrak{D}}$ der $\sigma$"=Algebra $\mathfrak{D}$ übergehen. Da $\mathfrak{B}$ annahmegemäß vollständig ist, gilt weiterhin $\overline{\mathfrak{D}}\subseteq\mathfrak{B}$.  Im Folgenden betrachten wir nun den Wahrscheinlichkeitsraum $(Y,\overline{\mathfrak{D}}, \nu_{\mid_{\overline{\mathfrak{D}}}})$. Dabei bezeichne $\nu_{\mid_{\overline{\mathfrak{D}}}}$ die Einschränkung von $\nu$ auf $\overline{\mathfrak{D}}$.\\
\noindent Wegen Behauptung \ref{Behauptung4} in Verbindung mit Bemerkung \ref{Bemerkung5.5} ist $L^1(Y,\overline{\mathfrak{D}}, \nu_{\mid_{\overline{\mathfrak{D}}}})$ ein separabler Banachraum und wir können $g\in L^\infty(Y,\nu)$  als Element von $L^\infty(Y,\overline{\mathfrak{D}},\nu_{\mid_{\overline{\mathfrak{D}}}})$ auffassen. Folglich ist die Proposition \ref{Behauptung3} auf das maßtheoretische dynamische System $(Y,\overline{\mathfrak{D}},\nu_{\mid_{\overline{\mathfrak{D}}}},S)$ sowie die Funktion $g\in L^\infty(Y,\overline{\mathfrak{D}},\nu_{\mid_{\overline{\mathfrak{D}}}})$ anwendbar.\\ Es existiert also eine Menge $X'\in \mathfrak{A}$ mit $\mu(X')=1$, welche wir unabhängig vom System $(Y,\overline{\mathfrak{D}},\nu_{\mid_{\overline{\mathfrak{D}}}},S)$ und $g$ wählen können, sodass für ein beliebiges $x\in X'$ eine Menge $Y'\in \overline{\mathfrak{D}}$ mit $\nu_{\mid_{\overline{\mathfrak{D}}}}(Y')=1$ existiert, für welche $$\lim_{N\to\infty}\frac{1}{N}\sum_{n=1}^N f(T^nx)g(S^ny)=0 \text{ für alle }y\in Y'$$ erfüllt ist. Da nach Konstruktion $\overline{\mathfrak{D}}\subseteq\mathfrak{B}$ gilt, folgt $Y'\in \mathfrak{B}$ sowie nach Definition von $\nu_{\mid_{\overline{\mathfrak{D}}}}$ $$\nu(Y')=\nu_{\mid_{\overline{\mathfrak{D}}}}(Y')=1.$$
Die Proposition ist damit bewiesen.

\end{proof}

\section{Beweis für $f\in L^\infty(X,\mu)$ und Gewichte für $L^1$-Funktionen}
\label{Kapitel4.2}
Nach Proposition \ref{Behauptung5} ist die Folge $\big(f(T^nx)\big)_{n=1}^\infty$ für ein invertierbares, ergodisches maßtheoretisches dynamisches System $(X,\mathfrak{A},\mu,T)$ und eine Funktion $f\in \mathcal{K}^\bot\cap L^\infty(X,\mu)$ sowie für $\mu$"=fast alle $x\in X$ ein universell gutes Gewicht für die punktweise Konvergenz von $L^\infty$"=Funktionen. Diese Aussage verallgemeinern wir auf ein maßtheoretisches dynamisches System $(X,\mathfrak{A},\mu,T)$ mit rein atomarer invarianter $\sigma$"=Algebra und eine Funktion $f\in L^\infty(X,\mu)$. Weiterhin zeigen wir, dass  die Folge $\big(f(T^nx)\big)_{n=1}^\infty$ für $\mu$"=fast alle $x\in X$ sogar ein universell gutes Gewicht für die punktweise Konvergenz von $L^1$"=Funktionen ist.\\
\noindent Bevor wir mit Schritt V fortfahren, zeigen wir aufbauend auf Proposition \ref{Behauptung5} und Satz \ref{Satz2.2b} zunächst die folgende Proposition:
\begin{Proposition}
	\label{Behauptung6}
	Seien $(X,\mathfrak{A},\mu,T)$ ein invertierbares, ergodisches maßtheoretisches dynamisches System sowie $f\in L^\infty(X,\mu).$ Dann ist die Folge $\big(f(T^nx)\big)_{n=1}^\infty$ für $\mu$"=fast alle $x\in X$ ein universell gutes Gewicht für die punktweise Konvergenz von $L^\infty$"=Funktionen.
\end{Proposition}
\begin{proof}[\textbf{Beweis:}]
	 Wegen \eqref{Kronecker} existieren Funktionen $f_1\in \mathcal{K}\cap L^\infty(X,\mu)$ und $f_2\in \mathcal{K}^\bot\cap L^\infty(X,\mu)$, sodass 
\begin{equation}
\label{Gleichung4.21}
f=f_1+f_2
\end{equation}
erfüllt ist. Nach Satz \ref{Satz2.2b} existieren für $f_1$ eine Menge $X_1\in \mathfrak{A}$ mit $\mu(X_1)=1$ und nach Proposition \ref{Behauptung5} für $f_2$ eine Menge $X_2\in \mathfrak{A}$ mit $\mu(X_2)=1$, sodass  die Folgen $\big(f_1(T^nx)\big)_{n=1}^\infty$ und $\big(f_2(T^nx)\big)_{n=1}^\infty$ für alle $x\in X_1\cap X_2$ universell gute Gewichte für die punktweise Konvergenz von $L^\infty$"=Funktionen sind. Insbesondere gilt $\mu(X_1\cap X_2)=1$.\\
Seien nun ein beliebiges maßtheoretisches dynamisches System $(Y,\mathfrak{B},\nu,S)$ sowie $g\in L^\infty(Y,\nu)$ gegeben. Wir fixieren ein $x\in X_1\cap X_2$. Nach Definition der Mengen $X_1$ und $X_2$ existiert eine Menge $Y'\in \mathfrak{B}$ mit $\nu(Y')=1$, sodass für alle $y\in Y'$ die Grenzwerte
\begin{equation} 
\label{Gleichung4.22}
\lim\limits_{N\to\infty}\frac{1}{N}\sum_{n=1}^N f_1(T^nx)g(S^ny)
\end{equation}
und 
\begin{equation}
\label{Gleichung4.23}
\lim_{N\to\infty}\frac{1}{N}\sum_{n=1}^N f_2(T^nx)g(S^ny)
\end{equation}
existieren. Für $y\in Y'$ existiert folglich auch der Grenzwert
\begin{equation*}
\lim_{N\to\infty}\frac{1}{N}\sum_{n=1}^N\big(f_1+f_2\big)(T^nx)g(S^ny)\overset{\eqref{Gleichung4.21}}{=}\lim_{N\to\infty}\sum_{n=1}^N f(T^nx)g(S^ny).
\end{equation*}
Die Proposition ist damit bewiesen. 
\end{proof}
\subsection{Schritt V: Invertierbare Erweiterung}
Ähnlich wie in Schritt I zeigen wir die folgende Proposition:
 \begin{Proposition}
 	\label{Behauptung7}
 Seien $(X,\mathfrak{A},\mu,T)$ ein ergodisches maßtheoretisches dynamisches System  sowie $f\in L^\infty(X,\mu)$. Dann ist für $\mu$"=fast alle $x\in X$ die Folge $\big(f(T^nx)\big)_{n=1}^\infty$ ein universell gutes Gewicht für die punktweise Konvergenz von $L^\infty$"=Funktionen.
 \end{Proposition} 
\begin{proof}[\textbf{Beweis:}] Sei $(\tilde{X},\tilde{\mathfrak{A}},\tilde{\mu},\tilde{T})$ die  invertierbare Erweiterung des Systems  $(X,\mathfrak{A},\mu,T)$ aus Satz \ref{Satz99} sowie $\pi$ die Faktorabbildung $ \pi: \tilde{X}\to X:\ \tilde{x}\mapsto \tilde{x}_0 $. Zur Veranschaulichung betrachten wir das folgende kommutative Diagramm:
\begin{equation*}
\begin{tikzcd}
\tilde{X} \arrow{d}[swap]{\pi} \arrow{r}{\tilde{T}} & \tilde{X} \arrow{d}{\pi} \\
X\arrow{r}{T} & X 
\end{tikzcd}
\end{equation*}	
	 \noindent  Da $(X,\mathfrak{A},\mu,T)$ ergodisch ist, gilt dies nach Satz \ref{Satz99} auch für $(\tilde{X},\tilde{\mathfrak{A}},\tilde{\mu},\tilde{T})$. Wir definieren nun $\tilde{f}\in L^\infty(\tilde{X}, \tilde{\mu})$ durch
\begin{equation}
\label{Gleichung4.24}
\tilde{f}:=f\circ \pi.
\end{equation}
Nach Proposition \ref{Behauptung6} existiert eine Menge \begin{equation} \label{Gleichung4.1}\tilde{X}'\in \tilde{\mathfrak{A}}\end{equation} mit $\tilde{\mu}(\tilde{X}')=1$, sodass die Folge $\big(\tilde{f}(\tilde{T}^n\tilde{x})\big)_{n=1}^\infty$ für alle $\tilde{x}\in \tilde{X}'$ ein universell gutes Gewicht für die punktweise Konvergenz von $L^\infty$"=Funktionen ist. Folglich gilt für ein beliebiges maßtheoretisches dynamisches System $(Y,\mathfrak{B},\nu,S)$ und eine Funktion $g\in L^\infty(Y,\nu)$ sowie alle $\tilde{x}\in \tilde{X}'$:
\begin{equation} \label{Gleichung5.4}\lim_{N\to\infty}\frac{1}{N}\sum_{n=1}^N \tilde{f}(\tilde{T}^n\tilde{x})g(S^ny) \text{ existiert für }\nu\text{"=fast alle }y\in Y.\end{equation}
Da $\pi$ eine Faktorabbildung ist, existiert deshalb für alle $\tilde{x}\in \tilde{X}'$ und $\nu$"=fast alle $y\in Y$ der Grenzwert
\begin{equation}
\label{Gleichung4.113b}
\begin{split}
\lim_{N\to\infty}\frac{1}{N}\sum_{n=1}^N \tilde{f}(\tilde{T}^n\tilde{x})g(S^ny)&\overset{\eqref{Gleichung4.24}}{=}\lim_{N\to\infty}\frac{1}{N}\sum_{n=1}^N \big(f\circ \pi\big)(\tilde{T}^n\tilde{x})g(S^ny)\\
&\hspace{0.3cm}=\lim_{N\to\infty}\frac{1}{N}\sum_{n=1}^N f\big(T^n\pi(\tilde{x})\big)g(S^ny).
\end{split}
\end{equation}
Wir definieren $$\tilde{X}_0\hspace{-0.1cm}:=\hspace{-0.1cm}\big\{\tilde{x}\in \tilde{X}\big| \big(\tilde{f}(\tilde{T}^n\tilde{x})\big)_{n=1}^\infty \hspace{-0.1cm}\text{ ist universell gutes } \text{Gewicht für die punktw. Konv. von }L^\infty\text{-Fkt.}\big\}$$ sowie 
$$X'\hspace{-0.1cm}:=\hspace{-0.1cm}\big\{x\in X\big| \big(f(T^nx)\big)_{n=1}^\infty \hspace{-0.1cm}\text{ ist universell gutes Gewicht für die punktw. Konv. von }L^\infty\text{-Fkt.}\big\}.$$ Für die in \eqref{Gleichung4.1} definierte Menge $\tilde{X}'$ gilt deshalb $\tilde{X}'\subseteq \tilde{X}_0$ sowie weiterhin wegen \eqref{Gleichung4.113b} $\pi^{-1}(X')=\tilde{X}_0,$  $\tilde{X}_0$ ist also eine Produktmenge. Mithilfe der Definition von $\tilde{\mu}$ erhalten wir
$$1=\tilde{\mu}(\tilde{X}')\leq \tilde{\mu}( \tilde{X}_0)=\tilde{\mu}\big(\pi^{-1}(X')\big)=\mu(X').$$
Die Proposition ist damit bewiesen.
\end{proof}

\subsection{Schritt VI: Zerlegung in Atome}
Aufbauend auf Behauptung  \ref{Behauptung16} und Proposition \ref{Behauptung7}  beweisen wir nun die folgende Proposition:
\begin{Proposition}
\label{Behauptung8}
Seien $(X,\mathfrak{A},\mu,T)$ ein maßtheoretisches dynamisches System mit rein atomarer invarianter $\sigma$"=Algebra  sowie $f\in L^\infty(X,\mu)$. Dann ist  die Folge $\big(f(T^nx)\big)_{n=1}^\infty$ für $\mu$"=fast alle $x\in X$ ein universell gutes Gewicht für die punktweise Konvergenz von $L^\infty$"=Funktionen.
\end{Proposition}
\begin{proof}[\textbf{Beweis:}]
Es bezeichne $\mathfrak{E}$ die invariante $\sigma$"=Algebra aus Definition \ref{invariant}. Da diese $\sigma$"=Algebra nach Voraussetzung rein atomar ist, existiert nach Definition \ref{Atom} eine Indexmenge $I\subseteq \mathbb{N}$ und eine Folge $(E_k)_{k\in I}\subseteq \mathfrak{E}$  von $\mu_{\mid_{\mathfrak{E}}}$"=Atomen, für welche 
\begin{equation} \label{Gleichung2}1=\mu_{\mid_{\mathfrak{E}}}\big(\bigcup_{k\in I}E_k\big)\overset{X\in \mathfrak{E}}{=}\mu\big(\bigcup_{k\in I}E_k\big)
\end{equation} erfüllt ist. Wir können ohne Beschränkung der Allgemeinheit annehmen, dass die Mengen $(E_k)_{k\in I}$ paarweise disjunkt sind. Da die Mengen $(E_k)_{k\in I}$ $\mu_{\mid_{\mathfrak{E}}}$-Atome sind, gilt nach Definition $0<\mu_{\mid_{\mathfrak{E}}}(E_k)=\mu(E_k)$, $k\in I$. Wir definieren die Maße $\mu_k,\, k\in I$ durch 
$$\mu_k(B):=\frac{\mu(B\cap E_k)}{\mu(E_k)},\ \  B\in \mathfrak{A}.$$ Nach Behauptung \ref{Behauptung16} sind diese Maße $T$-invariant und ergodisch.\\
\noindent Sei nun ein $k\in I$ beliebig gewählt (nach Voraussetzung ist $I\neq \emptyset$). Das System $(X,\mathfrak{A},\mu_k,T)$ ist ergodisch und wir erhalten mithilfe von Proposition \ref{Behauptung7} eine Menge $X_k\in \mathfrak{A}$ mit \begin{equation} \label{Gleichung4.123}\mu_k(X_k)=1,\end{equation} sodass die Folge $\big(f(T^nx)\big)_{n=1}^\infty$ für alle $x\in X_k$ ein universell gutes Gewicht für die punktweise Konvergenz von $L^\infty$"=Funktionen ist. Für alle $x\in X_{k}$ und ein beliebiges maßtheoretisches dynamisches System $(Y,\mathfrak{B},\nu,S)$ sowie eine Funktion $g\in L^\infty(Y,\nu)$ gilt also: \begin{equation}
\label{Gleichung4.30b} \lim_{N\to\infty}\frac{1}{N}\sum_{n=1}^N f(T^nx)g(S^ny)\text{ konvergiert für }\nu\text{"=fast alle }y\in Y.\end{equation} \noindent Wir definieren  \begin{equation} \label{Gleichung4.2}X':=\bigcup_{k\in I} X_{k}\in \mathfrak{A}.\end{equation} \noindent Da die Konvergenz in \eqref{Gleichung4.30b} unabhängig vom Maß auf dem messbaren Raum $(X,\mathfrak{A})$ ist, ist die Folge $\big(f(T^nx)\big)_{n=1}^\infty$ für alle  $x\in X'$ ein universell gutes Gewicht für die punktweise Konvergenz von $L^\infty$"=Funktionen. Zum Beweis von Proposition \ref{Behauptung8} verbleibt deshalb zu zeigen, dass $\mu(X')=1$ erfüllt ist.\\
\noindent Wegen \eqref{Gleichung2} gilt $\mu(X')=\mu\big(X'\cap (\bigcup_{k\in I}E_k)\big)=\mu\big(\bigcup_{k\in I}(X'\cap E_k)\big)$. Da die Mengen $(E_k)_{k\in I}$ paarweise disjunkt sind, erhalten wir somit
\begin{equation*}
\mu(X')=\sum_{k\in I}\mu\big(X'\cap E_k\big)
=\sum_{k\in I}\mu(E_k)\cdot \mu_k(X')
\overset{\eqref{Gleichung4.2}}{\geq} \sum_{k\in I}\mu(E_k)\cdot \mu_k(X_k)
\overset{\eqref{Gleichung2}}{\underset{\eqref{Gleichung4.123}}{=}}1.
\vspace{-0.9cm}
\end{equation*}
\end{proof}

\begin{Bemerkung}
	Wir könnten die Proposition auch mithilfe der ergodischen Zerlegung aus Theorem \ref{Zerlegung} beweisen. Um die Menge $X'$ in \eqref{Gleichung4.2} als abzählbare Vereinigung schreiben und damit sicherstellen zu können, dass diese ein Element von $\mathfrak{A}$ ist, wäre jedoch ebenfalls die Annahme einer rein atomaren invarianten $\sigma$"=Algebra erforderlich.\\
	\noindent Nach Behauptung \ref{Behauptung9} ist die invariante $\sigma$"=Algebra insbesondere dann rein atomar, wenn
	 $L^\infty(X,\mathfrak{E},\mu_{\mid_{\mathfrak{E}}})$ separabel ist. Eine weitere hinreichende Bedingung ist die Ergodizität des Systems $(X,\mathfrak{A},\mu,T)$.
\end{Bemerkung}

\subsection{Schritt VII: Abschluss des Beweises des Rückkehrzeitentheorems}
\label{SchrittVII}
In diesem Abschnitt zeigen wir, dass die mithilfe von Proposition \ref{Behauptung8} gewonnenen Gewichte sogar universell gute Gewichte für die punktweise Konvergenz von $L^1$"=Funktionen sind. Mit diesem letzten Verallgemeinerungsschritt schließen wir den Beweis des Rückkehrzeitentheorems von Bourgain für maßtheoretische dynamische Systeme mit rein atomarer invarianter $\sigma$"=Algebra ab.
\begin{proof}[\textbf{Beweis des Rückkehrzeitentheorems:}]
	Es bezeichne $X'\in \mathfrak{A}$ mit $\mu(X')=1$ die Menge aus Proposition \ref{Behauptung8}, für welche die Folge $\big(f(T^nx)\big)_{n=1}^\infty$ für alle $x\in X'$ ein universell gutes Gewicht für die punktweise Konvergenz von $L^\infty$"=Funktionen ist.
	Wir können ohne Beschränkung der Allgemeinheit annehmen, dass für alle $x\in X'$ und alle $n\in \mathbb{N}$\vspace{-0.1cm} $$|f(T^nx)|\leq ||f||_{L^\infty(X,\mu)}$$ erfüllt ist.
Seien weiterhin ein beliebiges maßtheoretisches dynamisches System $(Y,\mathfrak{B},\nu,S)$ und eine Funktion $g\in L^1(Y,\nu)$ gegeben. Wir definieren für $k\in \mathbb{N}$\vspace{-0.1cm}
$$g_k:=g\cdot \mathbbm{1}_{|g|\leq k}.$$ Nach Behauptung \ref{Lemma9} existiert eine Menge $Y'_1\in\mathfrak{B}$ mit vollem Maß, sodass für alle $y\in Y'_1$\vspace{-0.1cm}
\begin{equation}
\label{Gleichung4.33}
\lim_{k\to\infty}\lim_{N\to\infty}\frac{1}{N}\sum_{n=1}^N |g-g_k|(S^ny)=0
\end{equation}
erfüllt ist.\vspace{-0.1cm}
Sei nun \begin{equation} \label{Gleichung4.34} x\in X' \end{equation} beliebig gewählt. Da die Folge $\big(f(T^nx)\big)_{n=1}^\infty$ nach Proposition \ref{Behauptung8} ein universell gutes Gewicht für die punktweise Konvergenz von $L^\infty$"=Funktionen ist und der abzählbare Schnitt von Mengen mit vollem Maß wieder Maß $1$ hat, existiert eine Menge $Y'_2\in \mathfrak{B}$ mit vollem Maß derart, dass der Grenzwert \vspace{-0.1cm}
$$\lim_{N\to\infty}\frac{1}{N}\sum_{n=1}^N f(T^nx)g_k(S^ny)$$ für alle $y\in Y'_2$ und $k\in \mathbb{N}$ existiert. Wir setzen\vspace{-0.1cm}
$$\mathfrak{B}\ni Y':=Y'_1\cap Y'_2.$$ Es gilt $\nu(Y')=1.$
Nun zeigen wir, dass die Folge $\big(\frac{1}{N}\sum_{n=1}^N f(T^nx)g(S^ny)\big)_{N=1}^\infty$ für alle $y\in Y'$ eine Cauchyfolge ist.\\ \noindent  Sei dazu $\epsilon>0$ sowie $y\in Y'$ beliebig gewählt.  Wir betrachten für $N,M,k\in \mathbb{N}$ die Ungleichung\vspace{-0.4cm}
\begin{equation}
\label{Gleichung2.23b}
\begin{split}
0&\leq \bigg|\frac{1}{N}\sum_{n=1}^N f(T^nx)g(S^ny)-\frac{1}{M}\sum_{n=1}^M f(T^nx)g(S^ny)\bigg|\\
&\leq \underbrace{\frac{1}{N}\sum_{n=1}^N||f||_{L^\infty(X, \mu)}\big|g-g_k\big|(S^ny)}_{I}+\underbrace{\frac{1}{M}\sum_{n=1}^M||f||_{L^\infty(X, \mu)}\big|g-g_k\big|(S^ny)}_{II}\\&+\underbrace{\bigg|\frac{1}{N}\sum_{n=1}^Nf(T^nx)g_k(S^ny)-\frac{1}{M}\sum_{n=1}^M f(T^nx)g_k(S^ny)\bigg|}_{III}.
\end{split}
\end{equation}\vspace{-0.4cm}\\
\noindent $\textbf{Zu I und II:}$  Nach Wahl von $y\in Y'$ in Verbindung mit \eqref{Gleichung4.33} existieren ein $k_0 \in \mathbb{N}$ sowie ein $N_{k_0}\in \mathbb{N}$, sodass für alle $N>N_{k_0}$ \vspace{-0.3cm}$$||f||_{L^\infty(X, \mu)}\cdot\frac{1}{N}\sum_{n=1}^N \big|g-g_{k_0}\big|(S^ny)<\frac{\epsilon}{3}$$\vspace{-0.6cm}\\ erfüllt ist. \\
\noindent $\textbf{Zu III:}$ Aufgrund der Konstruktion von $Y'$ existiert der Grenzwert \vspace{-0.22cm}$$\lim_{N\to\infty}\frac{1}{N}\sum_{n=1}^N f(T^nx)g_{k}(S^ny)$$\vspace{-0.42cm}\\ für alle $y\in Y'$ und alle $k\in \mathbb{N}$. Deshalb ist für $k=k_0$ die Folge $\big(\frac{1}{N}\sum_{n=1}^N f(T^nx)g_{k_0}(S^ny)\big)_{n=1}^\infty$ eine Cauchyfolge. Also existiert ein $N_0 \in \mathbb{N}$, sodass für alle $N,M>N_0$ \vspace{-0.18cm}$$\Big|\frac{1}{N}\sum_{n=1}^N f(T^nx)g_{k_0}(S^ny)-\frac{1}{M}\sum_{n=1}^M f(T^nx)g_{k_0}(S^ny)\Big|<\frac{\epsilon}{3}$$ erfüllt ist. 
Indem wir in $\eqref{Gleichung2.23b}$ $k=k_0$ setzen, erhalten wir deshalb für alle $N, M >\max\{N_{k_0}, N_0\}$ \vspace{-0.18cm}
$$0\leq \Big|\frac{1}{N}\sum_{n=1}^N f(T^nx)g(S^ny)-\frac{1}{M}\sum_{n=1}^M f(T^nx)g(S^ny)\Big|< \frac{\epsilon}{3}+\frac{\epsilon}{3}+\frac{\epsilon}{3}=\epsilon.$$
\noindent Weil wir $\epsilon>0$ beliebig gewählt haben, existiert folglich für alle  $y\in Y'$ der Grenzwert\vspace{-0.18cm} $$\lim_{N\to\infty}\frac{1}{N}\sum_{n=1}^N f(T^nx)g(S^ny).$$ Da wir auch $x\in X'$ beliebig gewählt haben, ist das Rückkehrzeitentheorem von Bourgain damit bewiesen.
\end{proof}

\chapter{Schlusswort}
Der in dieser Arbeit vorgestellte Beweis des Rückkehrzeitentheorems von Bourgain beruht auf der Zerlegung einer Funktion $f\in L^\infty(X,\mu)$ bezüglich des Kroneckerfaktors $\mathcal{K}$, um die Fälle $f\in \mathcal{K}$ und $f\in \mathcal{K}^\bot$ getrennt behandeln zu können. Für Eigenfunktionen zeigen wir das Rückkehrzeitentheorem durch Anwendung des punktweisen Ergodensatzes auf das Produktsystem  $({\mathbb{T}}\times Y, \mathbb{B}_{{\mathbb{T}}}\times \mathfrak{B}, \lambda_{{\mathbb{T}}}\otimes \nu, \phi_{\lambda })$ (siehe Kapitel \ref{Kapitel2}). Den Beweis für $f\in \mathcal{K}^\bot \cap L^\infty(X,\mu)$  führen wir in Kapitel \ref{Kapitel3} indirekt für einen Spezialfall. Dazu konstruieren wir eine Folge von Intervallen $\big((L_j, M_j)\big)_{j=1}^J$, auf welcher sich der im Rückkehrzeitentheorem betrachtete Grenzwert "`schlecht"' (vgl. Schritt I), der zur Definition der Menge $X_1$ verwendete Grenzwert jedoch "`gut"' (vgl. Schritt II) verhält. Darauf aufbauend leiten wir die  Eigenschaften $(\alpha)$ und $(\beta)$ her und führen diese zu einem Widerspruch (Schritt III). In Kapitel \ref{allgemein} verallgemeinern wir aufbauend auf eigenen Überlegungen die in Kapitel \ref{Kapitel2} und \ref{Kapitel3} bewiesenen Spezialfälle auf das Rückkehrzeitentheorem von Bourgain. Die Voraussetzung, die invariante $\sigma$"=Algebra sei rein atomar, benötigen wir, um die Menge $X'$ in Beweisschritt VI (Kapitel 5) als messbare Menge mit vollem Maß konstruieren zu können. Ersetzen wir diese Voraussetzung durch die stärkere Annahme der Ergodizität sowie die Funktion $f\in L^\infty(X,\mu)$ durch die Funktion $\mathbbm{1}_A$, so erhalten wir das Rückkehrzeitentheorem in der Formulierung von \cite{Bourgain1989}:
 \vspace{0.15cm}\\
\noindent \textbf{Rückkehrzeitentheorem.}\quad \textit{Seien $(X, \mathfrak{A}, \mu, T)$ ein  ergodisches maßtheoretisches dynamisches System und $A\in \mathfrak{A}$ eine Menge mit $\mu(A)>0$. 
	Dann ist die Folge $\big(\mathbbm{1}_A\left(T^n x\right)\big)_{n=1}^{\infty}$ für $\mu$-fast alle $x \in X$ 
	ein universell gutes Gewicht für die punktweise Konvergenz von $L^1$"=Funktionen.}\vspace{0.15cm}\\
\noindent In \cite[S.65]{Assani2003} ist das Rückkehrzeitentheorem sogar für ein beliebiges maßtheoretisches dynamisches System $(X,\mathfrak{A},\mu,T)$ formuliert, allerdings wird die Aussage nur im ergodischen Fall bewiesen. Ein alternativer Beweis, welcher mithilfe der 1967 von Furstenberg eingeführten \textit{joinings} geführt wird, befindet sich in  \cite{Rudolph1994}).\vspace{0.2cm}\todo{schauen, dass der extra-Abstand hier nicht komisch ist}\\ Seit der Veröffentlichung des Rückkehrzeitentheorems im Jahr 1988 wurden zahlreiche Verallgemeinerungen bewiesen sowie weiterführende Fragestellungen aufgeworfen. Einen Überblick über aktuelle Entwicklungen findet man in \cite{Assani2014}. Unter Ausnutzung des Banach"=Prinzips und der Hölder"=Ungleichung lässt sich das Rückkehrzeitentheorem auf den Fall  $f\in L^p(X,\mu)$ und $g\in L^q(Y,\nu)$ mit $\frac{1}{p}+\frac{1}{q}\leq 1$ verallgemeinern (\cite[S.44]{Assani2014}). Eine bisher nur teilweise beantwortete Frage ist, welche Konvergenzaussagen im Fall $\frac{1}{p}+\frac{1}{q}>1$ zutreffend sind. In \cite{Assani2003b} wird das folgende negative Resultat gezeigt:\vspace{0.2cm}\\
\noindent \textbf{Theorem}. \quad \textit{Sei $(X,\mathfrak{A},\mu,T)$ ein ergodisches maßtheoretisches dynamisches System. Dann existiert eine Funktion $f\in L^1(X,\mu)$ und eine Menge $X'\subseteq X$ mit vollem Maß, sodass für $x\in X'$ und ein beliebiges ergodisches maßtheoretisches dynamisches System $(Y,\mathfrak{B},\nu,S)$ eine Funktion $g\in L^1(Y,\nu)$ derart existiert, dass der Grenzwert $\lim_{N\to\infty} \frac{1}{N}\sum_{n=0}^{N-1} f(T^nx)g(S^ny)$ für $\nu$"=fast alle $y\in Y$ divergiert.} \vspace{0.2cm}\\
\noindent 
 In \cite{Demeter2008} wird hingegen gezeigt, dass das Rückkehrzeitentheorem mit $f\in L^p(X,\mu)$ und $g\in L^q(Y,\nu)$ für $p>1$ und $q\geq 2$ Gültigkeit besitzt. Aufbauend darauf beweist Demeter in \cite{Demeter2012}, dass das Rückkehrzeitentheorem für beliebige $1<p,q \leq \infty$ mit $\frac{1}{p}+\frac{1}{q}<\frac{3}{2}$ wahr ist.\\
 \noindent Eine andere Fragestellung ist, unter welchen Voraussetzungen man für alle $x\in X$ durch $\big(f(T^nx)\big)_{n=1}^\infty$  universell gute Gewichte für die punktweise Konvergenz von $L^1$"=Funktionen erhält. Eine positive Aussage dazu ist in  \cite[S.125]{Assani2003} formuliert.\\ \noindent  Als eine weitere Verallgemeinerung des Rückkehrzeitentheorems von Bourgain können Mittel der Form $\frac{1}{N} \sum_{n=1}^N \prod_{j=1}^{J}f_j(T_j^nx_j)$ untersucht werden. Ein zentrales Resultat zu multiplen Rückkehrzeiten entstammt \cite{Rudolph1997} und besagt, dass für $J\in \mathbb{N}$ und beliebige maßtheoretische dynamische System $(X_1, \mathfrak{A}_1, \mu_1,T_1)$, $(X_2, \mathfrak{A}_2, \mu_2, T_2), \dots, (X_J, \mathfrak{A}_J, \mu_J, T_J)$ und Funktionen $f_1\in L^\infty(X_1,\mu_1)$, $f_2\in L^\infty(X_2,\mu_2),$ $\dots, f_J\in L^\infty(X_J,\mu_J)$ die Folge $\big(f_1(T_1^nx_1)\cdot f_2(T_2^nx_2)\cdot\vphantom{y} \dots \vphantom{y} \cdot f_J(T_J^nx_J)\big)_{n=1}^\infty$ für $\mu_j$"=fast alle $x_j\in X_j$, $j=1, \dots, J$ ein universell gutes Gewicht für die punktweise Konvergenz von $L^\infty$"=Funktionen ist.\\
 \noindent 
Ersetzt man im Rückkehrzeitentheorem von Bourgain die Folge $\big(g(S^ny)\big)_{n=1}^\infty$ durch eine nicht notwendig von einem maßtheoretischen dynamischen System erzeugte Zahlenfolge $(a_n)_{n=1}^\infty$, wirft dies die Frage auf, für welche Folgen $(a_n)_{n=1}^\infty$ sich analoge Aussagen beweisen lassen. In \cite{Host2009} wird gezeigt, dass für ein beliebiges ergodisches maßtheoretisches dynamisches System $(X,\mathfrak{A},\mu,T)$ und eine Funktion $f\in L^1(X,\mu)$ eine Menge $X'\in \mathfrak{A}$ mit $\mu(X')=1$ existiert, sodass der Grenzwert $\lim_{N\to\infty}\frac{1}{N}\sum_{n=1}^N a_nf(T^nx)$ für ein beliebiges Element $(a_n)_{n=1}^\infty$ aus einer bestimmten Klasse von Folgen (sogenannte \textit{nilsequences}) existiert. Eine Verallgemeinerung dieses Resultats für die Konvergenz entlang von temperierten F\o lner"=Folgen wird in \cite{Eisner2012} bewiesen. Eine Aussage für den multiplen Fall findet man in \cite{Zorin2014} sowie weitere Resultate unter anderem für sogenannte lineare Folgen in \cite{Eisner2013}.\\
\noindent Die hier vorgestellten Resultate und Problemstellungen behandeln nur einen kleinen Teil aktueller mit dem Rückkehrzeitentheorem verknüpfter Forschungsfragen. Jean Bourgain konnte mit dem Beweis des Rückkehrzeitentheorems nicht nur ein für eine lange Zeit offenes Problem lösen, sondern übt damit auch heute noch einen großen Einfluss auf aktuelle Entwicklungen in der Ergodentheorie aus.
\clearpage
%\setstretch{1.03}

\bibliography{library}
\addcontentsline{toc}{section}{Literaturverzeichnis}

\end{document}